\documentclass[12pt]{article}

\usepackage{fancyhdr, amsmath, amssymb, amsfonts, url, subfigure, dsfont, mathrsfs, graphicx, epsfig, subfigure, amsthm, tikz}
\usepackage{epsfig}
\usepackage{graphicx}
\usepackage{color}
\usepackage{multirow,hhline}

\usepackage{hyperref}
\usepackage{longtable}
\usepackage{tikz}
\usetikzlibrary{shapes,arrows.meta,chains,positioning}
\tikzstyle{decisionL} = [diamond, draw, fill=blue!10, shape aspect =2.5, 
text width=10em, inner sep=0pt]
\tikzstyle{decisionM} = [diamond, draw, fill=blue!10, shape aspect =2.5, 
text width=5.5em, inner sep=2pt]
\tikzstyle{decisionS} = [diamond, draw, fill=blue!10, shape aspect =2.5, 
text width=3.0em, inner sep=2pt]
\tikzstyle{blockS} = [rectangle, draw, fill=blue!10, 
text width=6em, text centered, rounded corners, minimum height=2.5em]
\tikzstyle{blockM} = [rectangle, draw, fill=blue!10, 
text width=6em, text centered, rounded corners, minimum height=2.5em]
\tikzstyle{blockL} = [rectangle, draw, fill=blue!10, text width=7em, text
centered, rounded corners, minimum height=5.0em, inner sep=6pt]
\tikzstyle{cloud} = [draw, ellipse,fill=red!10, minimum height=3.0em, minimum
width = 5.5em]

\usepackage{todonotes}

\newlength{\guillotine}
\makeatletter
\def\ds@whichfont{dsrom}
\relax

\DeclareMathAlphabet{\mathds}{U}{\ds@whichfont}{m}{n}

\makeatother

\newcommand\bbone{\mathds{1\!}}

\newtheorem{thm}{Theorem}[section]
\newtheorem{cor}[thm]{Corollary}
\newtheorem{lemma}[thm]{Lemma}

\newtheorem{prop}[thm]{Proposition}

\newtheorem{notation}[thm]{Notation}
\newtheorem{definition}[thm]{Definition}
\newtheorem{example}[thm]{Example}
\newtheorem{conjecture}[thm]{Conjecture}
\theoremstyle{remark}
\newtheorem{rem}[thm]{Remark}
\begin{document}



\title{Hausdorff dimension estimates applied to Lagrange and Markov spectra, Zaremba
theory, and limit sets of Fuchsian groups} 

\author{M. Pollicott and P. Vytnova\thanks{ The first author is partly
supported by ERC-Advanced Grant 833802-Resonances and EPSRC grant
EP/T001674/1 the second author is partly supported by
EPSRC grant EP/T001674/1. }}

\date{ }
\maketitle

\abstract{In this note we will describe a simple and practical approach to get 
rigorous bounds on the Hausdorff dimension of limits sets for some one dimensional  
Markov  iterated function schemes.    The general  problem has attracted 
considerable attention, but we are particularly concerned with the role of the
value of the Hausdorff dimension in solving conjectures and problems in other
areas of mathematics. 
As our first  application we confirm, and often strengthen, conjectures 
on the difference of the Lagrange and Markov\footnote{Markov's name will appear
in this article in two contexts, namely 
those of Markov spectra in number theory and the Markov condition from 
probability theory. Since the both notions are associated with the same
person~(A.~A.~Markov, 1856--1922), we have chosen to use the same spelling, despite 
the conventions often used in these different areas.}
spectra in Diophantine analysis, which appear in the work of Matheus and 
Moreira~\cite{MM}.  As a second application we (re-)validate and improve estimates 
connected  with the Zaremba conjecture in number theory, used in the work 
of Bourgain--Kontorovich~\cite{BK14}, Huang~\cite{Huang} and
Kan~\cite{Kan19}.
As a third more geometric application, we rigorously bound the bottom of the 
spectrum of the Laplacian for  infinite area surfaces, as illustrated by an 
example studied by McMullen~\cite{McMullen}.

In all approaches to estimating the dimension of limit sets there are questions 
about the efficiency of the algorithm, the computational effort required and 
the rigour of the bounds.  The approach we use has the virtues of being simple 
and efficient and we present it in Section~\ref{sec:estimates} in a way that is
straightforward to implement. 

These estimates apparently cannot be obtained by other known methods.
}

\section{Introduction}

We want to consider some interesting  problems where a knowledge of the  exact  
value of the Hausdorff dimension of some appropriate set plays an important role in an 
apparently unrelated area. For instance we consider the  applications  to 
Diophantine approximation and the difference between the Markov and Lagrange 
spectra, denominators of finite continued 
fractions and the Zaremba conjecture, and the spectrum of the Laplacian on certain Riemann surfaces.
   A common feature is that the progress on
these topics depends on accurately computing the dimension of certain limit sets for iterated function 
schemes.

The sets in question are dynamically defined sets given by Markov iterated function schemes.
The traditional approach to estimating the dimension of such sets is to use a variant of 
what is sometimes called a finite section method. This typically involves 
approximating the associated transfer operator by a finite rank operator and 
deriving approximations to the dimension from its maximal eigenvalue.  
This method originated with traditional  Ulam method and there are various 
applications and refinements due to~Falk---Nussbaum~\cite{FN18}, Hensley~\cite{Hensley}, 
McMullen~\cite{McMullen} and others.  
A second approach, which we will call  the periodic point  approach,  uses 
fixed points for combinations of contractions in the iterated function 
schemes~\cite{JP02}.  This  approach works best  for a small number of 
analytic branches whereas the finite section method often works more generally.
 However, in both of these approaches additional work is needed to address the 
important issue of validating numerical results. In the case of 
periodic point method there has been recent progress in getting rigorous estimates 
for Bernoulli systems~\cite{JP20}, but 
it can still be particularly difficult to get rigorous bounds in the case of Markov 
maps.  In the case of Ulam's method the size of matrices involved in
approximation can make it  hard to obtain reasonable bounds.

In this note we want to use a different approach which has the twin  merits  of 
giving both effective estimates on the dimension and ensuring the rigour of these values.
This is based on combining elements of the methods of Babenko--Yur'ev~\cite{Babenko} and
Wirsing~\cite{Wirsing} originally developed for the Gauss map. We will describe this in more
detail in \S~\ref{sec:estimates}.

To complete the introduction we will discuss our main  applications.

\subsection*{Application I:  Markov and Lagrange spectra}
As our first application, we can consider the work of Matheus and
Moreira~\cite{MM} on estimating  the size of the  difference of 
two subsets of the real line called the Markov spectrum $\mathcal M \subset
\mathbb R^+$ and the Lagrange spectrum $\mathcal L  \subset \mathbb R^+$.
 The two sets play an important role in Diophantine approximation theory 
and an excellent introduction to topics  in  this subsection is~\cite{CF89}.

By a classical result of Dirichlet from 1840, for any irrational number $\alpha$ 
there are  infinitely many rational numbers $\frac{p}{q}$ 
satisfying
$
|\alpha - \frac{p}{q}| \leq \frac{1}{q^2}.
$
For  each irrational $\alpha$ we can choose the largest value
$\ell(\alpha) > 1$ such that the inequality
$
\left| \alpha - \frac{p}{q} \right| \leq \frac{1}{\ell(\alpha) q^2}
$
still has infinitely many solutions with $\frac{p}{q} \in \mathbb Q$.
An equivalent definition would be
$$
\ell(\alpha) \colon = \left(\inf_{ p,q  \in \mathbb Z, q \neq 0} 
|q (q\alpha - p)| \right)^{-1}.
$$
For example, we know that $\ell\left(\frac{1 + \sqrt{5}}{2} \right) = \sqrt{5}$,
$\ell\left(1-\sqrt{2} \right) = \sqrt{8}$, etc. 
{\color{black}
The Hurwitz irrational number theorem states that for any irrational~$\alpha$
there are infinitely many rationals $\frac{p}q$ such that
$\left|\alpha-\frac{p}q\right|<\frac{1}{q^2\sqrt5}$. This implies, in
particular, that $\ell(\alpha) \ge \sqrt5$ for all $\alpha \in \mathbb R
\setminus \mathbb Q$. }
\begin{definition}
    The set $ \mathcal L = \{\ell(\alpha) \colon \alpha \in \mathbb R\setminus \mathbb Q \} $
    is called the \emph{Lagrange spectrum}.
\end{definition}  
There is another characterisation of elements of Lagrange spectrum in terms of
continued fractions~\cite{CF89}.   
We denote
the infinite continued fraction of $\alpha \in \mathbb R\backslash \mathbb Q$ by 
$$
\alpha =  [a_0;a_1,a_2,\ldots] = a_0+\cfrac{1}{a_1+\cfrac{1}{a_2+\cfrac{1}{a_3+\dots}}}
$$
where $a_0 \in \mathbb Z$ and $a_n  \in \mathbb N$ for $n\ge1$.
Assume that for $\alpha \in  \mathbb R\setminus \mathbb Q$ we have
$$
\alpha = [a_0;a_1,a_2,\ldots] =
[a_0;a_1,a_2,\ldots,a_n,\mathbf{\alpha_{n+1}}],
$$ 
in other words for the $n$'th rational approximation we may write
$$
\left|\alpha - \frac{p_n}{q_n}\right| = \frac{1}{\left(\alpha_{n+1} +
\frac{q_{n-1}}{q_n}\right)q_n^2}. 
$$
Then 
$$
\ell(\alpha) = \limsup_{n\to\infty} \left(\alpha_{n+1} +
\frac{q_{n-1}}{q_n}\right).
$$
Replacing $\limsup$ in the latter formula by supremum, we get the definition of the
Markov spectrum. 
\begin{definition}
In the notation introduced above, let
$$
\mu(\alpha) = \sup_n \left(\alpha_{n+1} + \frac{q_{n-1}}{q_n}\right).
$$
    The set $ \mathcal M = \{\mu(\alpha) \colon \alpha \in \mathbb R\setminus \mathbb Q \} $
    is called the \emph{Markov spectrum}.
\end{definition}

 There is an equivalent definition of the Markov spectrum in terms of quadratic forms.
Both notions were suggested by Markov in 1879--80
~\cite{gm1},~\cite{gm2}.

Naturally, the sets $\mathcal L$ and $\mathcal M$ have many similarities.
The smallest value for each is~$\sqrt{5}$ and in $[\sqrt5,3]$ both sets are
countable and agree, i.e., 
$$
\mathcal L \cap [\sqrt{5},3] = \mathcal M \cap [\sqrt{5},3]
= \{\sqrt{5},\sqrt{8}, \sqrt{221}/5 \cdots \}.
$$
Furthermore, Freiman~\cite{Freiman2}, following earlier work of Hall~\cite{Hall}, 
computed an explicit constant, called Freiman constant $\sqrt{20}<c_F<\sqrt{21}$, such that 
$$
\mathcal L \cap [c_f, +\infty)
= \mathcal M \cap [c_F \ldots, +\infty)  =  [c_F, + \infty).
$$
The half-line $[c_F,+\infty)$ is known as Hall's ray. 
Nevertheless, these two sets are actually different.
In particular, Tornheim~\cite{Tornheim} showed  $\mathcal L \subseteq \mathcal M$ and 
Freiman~\cite{Freiman1} showed $\mathcal L \neq \mathcal M$.

In a recent work Matheus and Moreira~\cite{MM}, \S B.2 give upper bounds on the
Hausdorff dimension $\dim_H(\mathcal M \setminus \mathcal L)$ in terms of the
Hausdorff dimension of limits sets of specific Markov Iterated Function Schemes. Using the approach 
presented in this article 
we compute the Hausdorff dimensions of the sets concerned, and
combining our numerical estimates  
in \S\ref{ssec:Markoff} with  the  intricate analysis of~\cite{MM} we obtain
the following result (the proof is computer-assisted).

\begin{thm}
    \label{MMtheorem}
    We have the following  bounds on the dimension of parts of $\mathcal M \setminus \mathcal L$ 
    \begin{enumerate}
        \item
            $\dim_H( (\mathcal M \setminus \mathcal L) \cap (\sqrt{5},
            \sqrt{13}))\phantom{0} < 0.7281096$;
        \item 
            $\dim_H( (\mathcal M \setminus \mathcal L) \cap (\sqrt{13}, 3.84)) <
            0.8552277$;
        \item
            $\dim_H( (\mathcal M \setminus \mathcal L) \cap (3.84, 3.92)) <
            0.8710525$;
        \item
            $\dim_H( (\mathcal M \setminus \mathcal L) \cap (3.92, 4.01)) <
            0.8110098$; and 
        \item
            $\dim_H( (\mathcal M \setminus \mathcal L) \cap (\sqrt{20},
            \sqrt{21})) < 0.8822195$.
    \end{enumerate}
\end{thm}

In particular, taking into account the known bound  $ \dim_H(\mathcal M \setminus \mathcal L \cap (
4.01, \sqrt{20})) < 0. 873316$ (\cite{MM}, (B.6)) on the remaining interval
we obtain  an upper bound of 
$$
\dim_H(\mathcal M \setminus \mathcal L) < 0.8822195.
$$
Note that this  confirms the conjectured upper bound $\dim_H(\mathcal M
\setminus \mathcal L)<0.888$~(\cite{MM}, (B.1)) and improves on the  earlier rigorous 
bound of $\dim_H(\mathcal M \setminus \mathcal L) < 0.986927$
(\cite{MM}, Corollary 7.5 and~\cite{Moreira}, Theorem 3.6).

For the purposes of comparison, we present the bounds in Theorem 1.3 with   the
previous rigorous bounds given
on different portions of $\mathcal M \setminus \mathcal L$ in
Figure~\ref{fig:compare}.
\begin{figure}[h!]
    \includegraphics{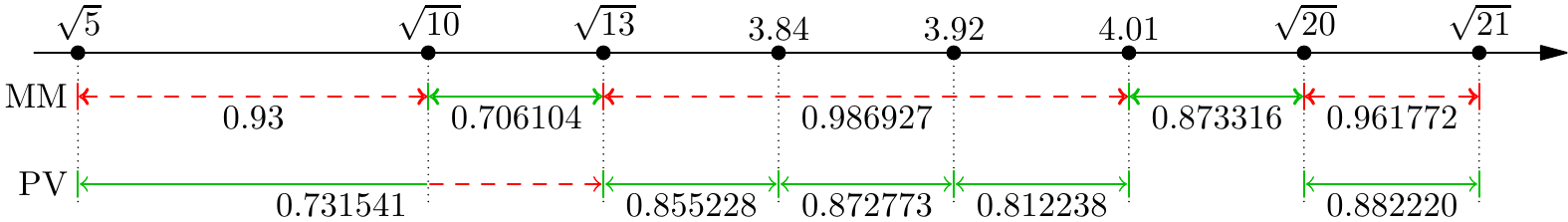}
    \caption{Comparison of old and new upper bounds.}
    \label{fig:compare}
\end{figure}

In the present work we will be looking for both lower and upper bounds. In order to
give a clearer presentation of the results we will use the following notation
\begin{notation}
  We abbreviate  $ a \in [b-c, b+c]$ as $a = b \pm c$.
\end{notation}

In order to get a lower bound on the difference of the 
Lagrange and Markov spectra
Matheus and Moreira~(\cite{MM}, Theorem 5.3) showed that there is a lower bound 
$\dim_H(\mathcal M \setminus \mathcal L) \geq \dim_H(E_2)$ where $E_2
\subset [0,1]$ denotes the Cantor set of irrational numbers with
infinite continued fraction expansions whose digits are either~$1$
or~$2$.\footnote{ A slight improvement on this lower bound is described in the 
book ``Classical and Dynamical Markov and Lagrange Spectra: Dynamical, Fractal
and Arithmetic Aspects'' by D. Lima, C. Matheus, C. G Moreira, S. Romana.}  
The study of the dimension of this set was initiated by Good in~1941~\cite{Good}.
There are various estimates on $\dim_H(E_2)$ including~\cite{JP20}
where the dimension was computed to $100$ decimal places using 
periodic points. In \S7 we will recover and improve on this estimate
giving an estimate accurate to~$200$ decimal places and thus we 
deduce the following result.
\begin{thm}\label{good}\footnote{ The calculation was done using Mathematica. The validity
    of the estimate depends on the internal error estimates of the software.}
    $$
    \begin{aligned} 
        \dim_H(E_2)        &= 0.
        5312805062\,7720514162\,
        4468647368\,4717854930\,
        5910901839\,8779888397\,
        \cr
        &\qquad 
        8039275295\,3564383134\,
        5918109570\,1811852398\,
        8042805724\,3075187633\,
        \cr
        &\qquad 
        4223893394\,8082230901\,
        7869596532\,8712235464\,
        2997948966\,3784033728\,
        \cr
        &\qquad 
        7630454110\,1508045191\,
        3969768071\,3
        \pm 10^{-201}.\cr
    \end{aligned}
    $$
\end{thm}
Details on the proof of this bound appear in 
\S\ref{ss:e2proof}.
Whereas it may not be clear why a knowledge of $\dim_H(E_2)$
to $200$ decimal places is beneficial, it at least serves to illustrate  
the effectiveness of the method we are using compared with earlier
approaches.

\subsection*{Application II: Zaremba theory}
A second application is to problems on finite continued fractions related to the Zaremba conjecture.  
The Zaremba conjecture~\cite{Zaremba} was formulated in 1972, motivated by problems in numerical analysis.  
It deals with the denominators that can occur in finite continued
fraction expansions using a uniform bound on the digits.  A nice account
appears in the very informative survey of Kontorovich~\cite{Kontorovich}.   

\bigskip
\noindent
{\bf Zaremba conjecture.}  For any natural number $q \in \mathbb N$ there exists $p$ 
(coprime to $q$) and $a_1, \cdots, a_n \in\{1,2,3,4,5\}$ such that 
$$
\frac{p}{q} = [0; a_1, \cdots, a_n]
\colon =  \cfrac{1}{a_1 + \cfrac{1}{a_2+  \cfrac{1}{\cdots + \cfrac{1}{a_n}}}}.
$$
 Let us denote for each  for $N \geq 1$ and  $m \geq 2$, 
$$
\begin{aligned}
    &D_m(N):=\cr
    & \mbox{Card} \left\{1 \!\leq q\! \leq N \mid  \exists p \in \mathbb N,
    (p,q)\!=\!1,\, 
    a_1, \cdots, a_n\! \in\!\{1,2,\cdots, m\}
    \mbox{ with } \frac{p}{q} \!=\! [0;a_1, \cdots, a_n]
    \right\},
\end{aligned}
$$
i.e., the number of $1\leq q \leq N$ which occur as denominators of finite continued fractions using digits $|a_i| \leq m$.
The Zaremba conjecture would correspond to $D_5(N) = N$ for all $N \in \mathbb
N$.
The conjecture remains open, but Huang~\cite{Huang}, building on  work
of Bourgain and Kontorovich~\cite{BK14}, proved  the following 
version of Zaremba conjecture.

\begin{thm}[Bourgain---Kontorovich, Huang]\label{thm:Huang} There is a density one version of the Zaremba conjecture, i.e., 
    $$
    \lim_{N \to +\infty} \frac{D_5(N)}{N}=1.
    $$
\end{thm}
There have been other important refinements on this result by
Frolenkov--Kan~\cite{FK14},~\cite{FK14-2},
Kan~\cite{Kan16},~\cite{Kan17}, Huang~\cite{Huang} and 
Magee--Oh--Winter~\cite{MOW}.

Let us introduce for each $m \geq 2$, 
$$
E_m \colon = \left\{[0;a_1,a_2,\cdots ] \mid a_n \in\{1,2, \cdots,
m\} \mbox{ for all } n \in \mathbb N \right\}
$$
which is a Cantor set in the unit interval.
Originally, Bourgain---Kontorovich~\cite{BK14} proved an analogue to Theorem~\ref{thm:Huang} for
$D_{50}(N)$. 
Amongst other things, their argument, related to the circle method, used the
fact that the Hausdorff dimension $\dim_H(E_{50})$ is sufficiently close to~$1$ 
(more precisely, $ \dim_H(E_{50})>\frac{307}{312}$). 
In Huang's refinement of their approach, he reduced $m$ to $5$, i.e. replaced
the alphabet $\{1,2, \cdots, 50\}$ with $\{1,2,3,4,5\}$, as in the statement of
Theorem~\ref{thm:Huang}. In Huang's approach, it was sufficient to show that
$\dim_H(E_5)>\frac56$. 
In~\cite{JP20} there is an explicit rigorous bound on the Hausdorff
dimension of this set which confirms this inequality.  The approach used
there is the periodic point method, whereas in this article  we  use a
different  method  to confirm and improve on these bounds. 

As another example, we recall the following result for $m=4$ and the smaller 
alphabet $\{1,2,3,4\}$. 

\begin{thm}[Kan~\cite{Kan17}]\label{thm:kan} For the alphabet  $\{1,2,3,4\} $ there
    is a positive density version of the Zaremba conjecture, i.e.,  
    $$
    \liminf_{N \to +\infty}\frac{D_4(N)}{N} > 0.
    $$
\end{thm}
The proof of the result is conditional on the lower bound $\dim_H(E_4) >
\frac{\sqrt{19}-2}{3}$.  In~\cite{Kan17} this  inequality
is attributed to Jenkinson~\cite{Jenkinson}, where this value was, in fact, only 
heuristically estimated.
In~\cite{JP20} there is an explicit rigorous bound on the Hausdorff
dimension of this set which confirms this inequality.  The approach  used
there is the periodic point method, whereas in this article  we  give a 
different  method to confirm and improve on these  bounds, as well as give new  examples. 
These results are presented in \S\ref{sec:Zaremba}.

\subsection*{Application III: Schottky  group limit sets}
A third application belongs to the area of hyperbolic geometry.  
The two dimensional hyperbolic space 
can represented as the Poincar\'e disc 
$\mathbb D^2 = \{z \in \mathbb C \colon |z| < 1\}$ 
with the Poincar\'e metric $ds^2 = 4(1-|z|^2)^{-2}$. 
A Fuchsian group $\Gamma$ is a discrete group of isometries of the  two dimensional hyperbolic space.   
In particular, the factor space $\mathbb D^2/\Gamma$ is a surface of constant curvature $\kappa = -1$.

We can associate to the Fuchsian group $\Gamma$  the limit set
$X_\Gamma \subseteq \partial D = \{z \in \mathbb C \colon |z|=1\}$  defined to be the Euclidean limit points 
of the orbit $\{g0 \colon g \in \Gamma\}$.   In the event that $\Gamma$ is cocompact, the  quotient 
$\mathbb D^2/\Gamma$ is a compact surface, and thus the limit set will
be equal to the entire unit circle.  On the other hand, if~$\Gamma$ is a
Schottky group then the limit set will be a Cantor subset of the unit
circle (of Hausdorff dimension strictly smaller than~$1$). 

In the particular case that  $\Gamma$ is a Schottky group the space 
$\mathbb D^2/\Gamma$ is a surface of infinite area.  
It is known~\cite{Chavel} that the classical Laplace---Beltrami operator has positive real
spectra and in particular, its smallest eigenvalue $\lambda_\Gamma > 0$ is strictly positive.
There is a close connection between  the spectral value $\lambda_\Gamma$
and the Hausdorff dimension $\dim_H(X_\Gamma)$. 
More precisely, we  have a classical result (see~\cite{Sullivan})
$$
\lambda_\Gamma = \min\left\{\frac{1}{4}, \dim_H(X_\Gamma)
(1-\dim_H(X_\Gamma))\right\}.
$$
Next we want to consider a concrete example of a Schottky group.  
\begin{example}
    McMullen~\cite{McMullen} considered the Schottky  group $\Gamma =
    \langle R_1, R_2, R_3\rangle$ generated by reflections $R_1, R_2,
    R_3: \mathbb D^2 \to \mathbb D^2$ in three symmetrically placed
    geodesics with end points~$e^{\pi i (2j + 1)/6}$
    with  $j=1, \cdots, 6$ on the unit circle (Figure~\ref{schottky}).

    
    \begin{figure}[h!]
        \begin{center}
          \includegraphics{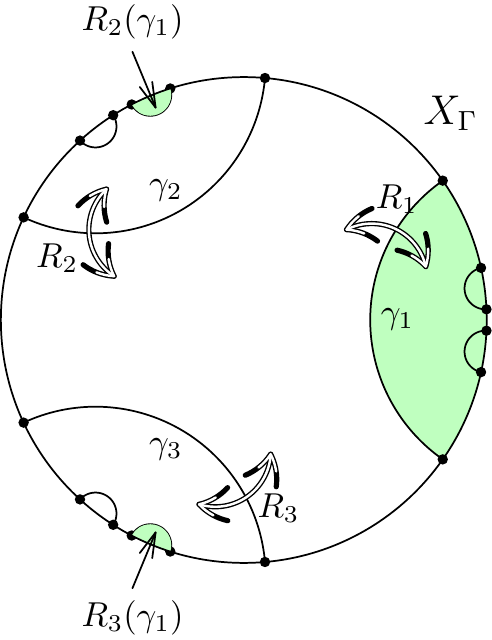}
            \end{center}
            \caption{Group $\Gamma$ generated by reflections $R_1$, $R_2$, and
            $R_3$ in the geodesics $\gamma_1$, $\gamma_2$, and $\gamma_3$. }
         \label{schottky}
        \end{figure}
\end{example}
The limit set $X_{\Gamma}$ can be written as a limit set of a suitable Markov
iterated function scheme, or, more precisely, a directed Markov graph
system~\cite{MU03}.
Applying the method described in this article we can estimate the dimension of the limit set and thus
the lowest eigenvalue of the Laplacian. 

\begin{thm}\label{mcm}
   In notation introduced above, the dimension of the limit set of the Schottky group $\Gamma$ satisfies 
    $$ 
    \dim_H(X_{\Gamma}) = 0.295546475\pm 5\cdot10^{-9} 
    $$ 
    and the smallest value of the Laplacian satisfies 
    $$
    \lambda_{\Gamma} = 0.2081987565\pm2.5\cdot10^{-9}
    $$
\end{thm}


Finally, the same approach we have used in these applications can also be used
to estimate the Hausdorff dimension of the various dynamically
defined limit sets of iterated function schemes which have been
considered by other authors cf.~\cite{Hensley}.
We will return to this in \S\ref{sec:moreExamples}, where we verify and improve
estimates of Hensley~\cite{Hensley} and Jenkinson~\cite{Jenkinson} for some
iterated function schemes and estimates of  Chousionis at
al.~\cite{CLU20} for some countable iterated function systems.

\bigskip
In the next section we will describe the general setting of one dimensional Markov iterated function schemes 
which is the main focus of our study and the key to proving these theorems.


\section[Definitions]{Definitions and Preliminary results}
In this section we collect together some of the background material we require.

\subsection{Hausdorff Dimension}
\label{ss:hdim}
The following classical definition of the Hausdorff dimension is well known, and
an excellent general reference is the textbook of Falconer~\cite{Falconer}.
Given $X \subset \mathbb R^+$ for each $0 < s <1$ and $\delta > 0$ we define the
$s$-dimensional \emph{Hausdorff content} of~$X$ by
$$
H_\delta^s (X) =  \inf_{\substack{\mathcal U = \{U_i\},\\ |U_i|\leq \delta}}
\sum_{i} |U_i|^s, 
$$
where the infimum is taken over all open covers $\mathcal U = \{U_i\}$
of~$X$ with each open set~$U_i$ having diameter~$|U_i|\leq\delta$. 
The $s$-dimensional \emph{Hausdorff outer measure} of~$X$ is given by
$$
H^s(X) = \lim_{\delta \to 0} H_\delta^s(X)  \in [0, +\infty].
$$
Finally, the \emph{Hausdorff dimension} of $X$ is defined as infimum of
values~$s$ for which the outer measure vanishes:
    $$
    \dim_H(X) = \inf \{s \mid H^s(X) = 0\}.
    $$

\subsection[Markov IFS]{Markov iterated function schemes}
We say that the contractions  $T_j\colon [0,1] \to [0,1]$ ($j=1, \cdots, d$) satisfy the 
{\it Open Set Condition} if there exists a non-empty open
set~$U\subset[0,1]$ such that the images $\{T_j U\}_{j=1}^d$ are pairwise
disjoint.  This will be the case in all  the examples we consider. 

We begin with the definition of a one dimensional  Markov iterated function scheme.
Recall that a matrix~$M$ is called aperiodic if there exists~$n \geq 1$
such that~$M^n > 0$, i.e., all of the entries are strictly positive.

\begin{definition}\label{def:mifs}
   Let $d \geq 2$.
   A Markov iterated function scheme consists of:
   \begin{enumerate}
\item 
 a family $T_j\colon [0,1] \to [0,1]$ ($j=1, \cdots, d$) of~$C^{1+\alpha}$ contractions satisfying the Open
    Set Condition; and 
    \item an aperiodic~$d\times d$ matrix~$M$ with entries~$0$ and~$1$, which
        gives the \emph{Markov condition}.
    \end{enumerate}
     We can define \emph{ the limit set} of
    $\{T_j\}_{j=1}^d$ with respect to the matrix~$M$ by 
    $$
    X_M = \left\{ \lim_{n\to +\infty} T_{j_1}\circ \cdots \circ
    T_{j_n}(0) \mid j_k \in \{1, \cdots, d\}, M(j_k, j_{k+1})\!=\!1, 1
    \!\leq \!k \!\leq\! n-1
    \right\}.
    $$
\end{definition}
{\color{black}
\begin{rem}
    In one of the examples we consider, the matrix~$M$ has an entire column of
    zeros. In this case, we can remove the contraction corresponding to this
    column from the iterated function scheme without changing the limit set.
    This corresponds to removing the row and the column corresponding to this
    contraction from the matrix~$M$. 
\end{rem}

 \begin{rem}
More generally, let~$M$ be a~$k\times k$ matrix with entries~$0$ or~$1$, and
assume its rows and columns are indexed by~$\{1, \cdots, k\}$. Given $r,s \in
\{1, \cdots, k\}$ we say~$s$ is accessible from~$r$ if there exists~$n \geq 1$
with~$M^n(r,s) \geq 1$. After reordering the index set, if necessary,  we can
assume that if $n \geq 1$ with $M^n(r,s) \geq 1$ then~$s \geq r$. 
We can then define an equivalence relation on $\{1, \cdots, k\}$ by~$r \sim s $
if both~$s$ is accessible from~$r$ and also~$r$ is accessible from~$s$ 
and assume that there are~$q$ distinct equivalence classes  $[j_1], \cdots, [j_q]$.  
The matrix~$M$ then takes the form of sub-matrices $M_1, \cdots, M_q$ 
on the diagonal indexed by the equivalence classes $[j_1], \cdots, [j_q]$, 
say,  with any other non-zero entries appearing only above the main
diagonal~\cite[Ch. 1]{seneta}. 

In particular, each matrix~$M_j$ is irreducible, i.e. for pair of indices $(r,s)$
there exists~$n = n(j,r,s)$ such that~$M_j^n(r,s) \ge 1$. 
The period~$d_j$ of~$M_j$  
is the greatest common divisor of~$n(j,r,s)$ for all pairs of indices $(r,s)$. 
Finally,  after further reordering of the index set, if necessary, the $d$th power~$M_j^d$  takes the form of 
aperiodic sub-matrices $M_{j1}, M_{j2}, \cdots, M_{jp}$ on the diagonal (i.e., 
there exists $n=n(j,k) \geq 1$ such that~$\forall r,s$ in the index set
for~$M_{jk}$ we have~$M_{jk}^n(r,s) \geq 1$).

We can also  consider the iterated function schemes~$X_{M_{jk}}$
associated to the matrices~$M_{jk}$, where the corresponding contractions being 
$d$-fold compositions of the original contractions.   
Then the iterated function scheme with contractions~$T_i$ where $i
\in \{1, \cdots, k\}$ with the limit set~$X_{M}$ has dimension $\dim(X_M) =
\max_{j,k} \dim(X_{M_{jk}})$.

\end{rem}

}

Definition \ref{def:mifs} is a special case of a more general  graph directed Markov
system~\cite{MU03}, where the contractions $T_j$ may have different domains and
ranges, to which our analysis also applies.
However, the above definition suffices for the majority of our
applications, although in the case of Fuchsian---Schottky groups a more general setting 
is implicitly used. 


\begin{rem}
    For Markov iterated function schemes, the  Hausdorff dimension coincides with
    the box counting dimension~\cite[Ch. 5]{P97}, see also~\cite{PW1}, which has a slightly easier
    definition. More precisely, for~$\varepsilon > 0$ we denote by
    $N(\varepsilon)$ the smallest number of~$\varepsilon$-intervals required to
    cover~$X$.  We define the Box dimension to be 
    $$
    \dim_B(X_M) = \lim_{\varepsilon \to 0} \frac{\log N(\varepsilon)}{\log(1/\varepsilon)}
    $$
    provided the limit exists.  Then $\dim_H(X_M) = \dim_B(X_M)$.
     However, we needed to introduce the definition of Hausdorff
    dimension,  for the benefit of the statements of results on Markov and Lagrange spectra. 
\end{rem}

\subsection{Pressure function}
We would like to use the Bowen---Ruelle formula~\cite{Ruelle82} to compute the
value of the Hausdorff dimension of the limit set of a Markov iterated function
scheme. We will use the following notation. 
{\color{black}
\begin{notation}
In what follows, $A(t)\lessapprox B(t)$
denotes that there exists $C>0$ such that $A(t) \le C \cdot B(t)$.
We write $A(t) \asymp B(t)$ if
 there exist constants $C_1,C_2 > 0$ such that $C_1 A(t) \le B(t) \le C_2 A(t)$
 for all sufficiently large~$t$.
\end{notation} 
}

Before giving the statement we first need to introduce the following defintion.
\begin{definition}\label{def:pressure}
    A pressure function $P\colon[0, 1] \to \mathbb R$ associated to a system of
    contractions $\{T_j\}_{j=1}^k$ with Markov condition defined by a matrix $M$
    is given by
    $$
    P(t) = \lim_{n\to +\infty}  \frac{1}{n}\log 
    \left(
     \sum_{\substack{ M(j_1, j_2) =  \cdots = \\ M(j_{n-1}, j_{n}) = 1}}
    \left|(T_{j_1} \circ \cdots \circ T_{j_n})'(0)\right|^t \right), 
    $$
    where the summation is taken over all compositions $T_{j_1} \circ
    \cdots \circ T_{j_n}\colon [0,1] \to [0,1]$ which are allowed by
    the Markov condition and the summands are the absolute
    values of the derivatives of these contractions at a fixed
    reference point (which for convenience we take to be~$0$) 
    raised to the power~$t$. 
\end{definition}
The pressure function depends on the matrix~$M$ but we omit this in the
notation. 
In the present context the function is  well defined as 
the limit in this definition of $P(t)$ always exists.
There are various other definitions of the pressure, and we
refer the reader to the books~\cite{PP90} and~\cite{Walters}
for more details. A sketch of the graph of the pressure function is
given in Figure~\ref{fig:pressure}.

The following well known connection between $\dim_H(X_M)$ and the pressure
function is useful for practical applications.
\begin{lemma}[Bowen~\cite{bowen}, Ruelle~{\cite[Proposition 4]{Ruelle82}} ]
  \label{bowenRuelle}
    In the setting introduced above, the pressure function of a Markov Iterated
    Function Scheme has the following properties:
    \begin{enumerate}
        \item
            $P(t)$ is a smooth monotone strictly decreasing  analytic function; and 
        \item 
            The Hausdorff dimension of the limit set is the unique zero of the
            pressure function i.e., $P(\dim_H(X_M)) = 0$.  
    \end{enumerate}
\end{lemma}
The general result of Ruelle extended a more specific posthumous result of Bowen on the Hausdorff dimension for quasi-circles.
Thus the problem of estimating $\dim_H(X_M)$ is reduced to the problem of locating the zero of the pressure function. 

\begin{rem}
  \label{rem:pdim1}
To understand the connection between the pressure and the Hausdorff dimension,
described in
Lemma~\ref{bowenRuelle}, we can consider covers of $X_M$
of the form $\mathcal U = \{ T_{i_1}\circ  \cdots  \circ T_{i_n} U\}$, where
$T_{i_1}\circ  \cdots  \circ T_{i_n}$ are allowed compositions {\color{black}(i.e.
$M(i_1,i_2) = \ldots = M(i_{n-1},i_n) = 1$)}
   and $U \supset [0,1]$ is an open set.
 The diameters of the elements of this cover for large~$n$ are comparable to the
 absolute values of the derivatives
 $ (T_{i_1}\circ  \cdots  \circ T_{i_n})'(0)$. {\color{black} In particular, by
 the mean value theorem, $\mathop{\rm diam} (T_{i_1}\circ \ldots \circ T_{i_n}(U)) \le
 \sup_{y \in U} |(T_{i_1}\circ \ldots \circ T_{i_n})^\prime (y) |$ and taking
 into account
 $$
 \frac{|(T_{i_1}\circ \ldots \circ T_{i_n})^\prime(y)| }{|(T_{i_1}\circ \ldots
 \circ  T_{i_n})^\prime(0)|}  
 \le \sup_{z \in U} \exp\left(\left( \log |(T_{i_1}\circ \ldots \circ  T_{i_n})^\prime(z)| \right)^\prime \right) < + \infty.
 $$
 }
 If  $P(t) = 0$ then for any $t_0 > t$
 we have that $P(t_0) < 0$ and therefore the Hausdorff content
$$
H_\delta^{t_0}(X_M) \lessapprox  \sum_{\substack{M(i_1,i_2) =  \cdots \\=M(i_{n-1},
i_n)=1} } \left|(T_{i_1}\circ  \cdots  \circ T_{i_n})'(0)\right|^{t_0}
$$
for $n$ sufficiently large.  In particular, letting $n \to +\infty$
we can deduce  from the definition of pressure that the right hand side of the
inequality tends to zero and therefore $H^{t_0}(X) = 0$. We then
conclude that outer measure vanishes and thus $\dim_H(X_M) \leq t_0$ and, since $t_0>t$ was arbitrary, we have
$\dim_H(X_M) \leq t$ (see \S\ref{ss:hdim}). 
For the reverse inequality see Remark~\ref{rem:pdim2}.
\end{rem}

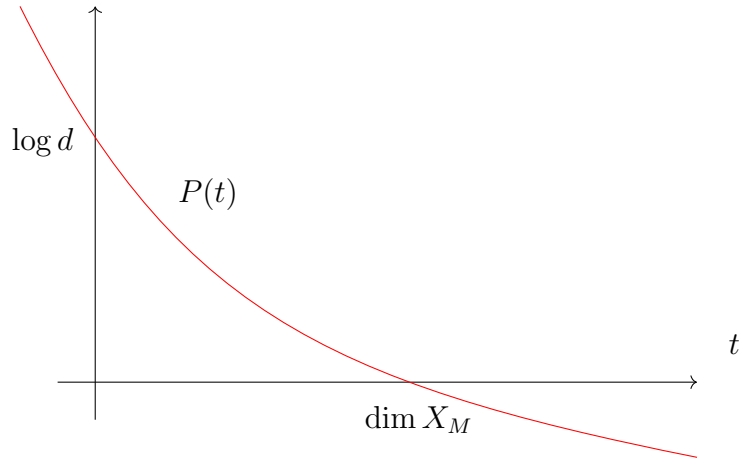
\begin{figure}
    \centerline{
    \begin{tikzpicture}
        \draw[<-] (0,5) --(0,-0.5); 
        \draw[->] (-0.5,0) --(8,0);
        \draw[red](-1,5) .. controls (1,1)  and (3,0) .. (8,-1.0); 
        \node at (4.3,  -0.5) {$\dim X_M$};
        \node at (1.5,  2.5) {$P(t)$};
        \node at (8.5,  0.5) {$t$};
        \node at (-0.7,  3.2) {$\log d$};
    \end{tikzpicture}
    }
    \caption{A typical plot of the pressure function $P$.}
    \label{fig:pressure}
\end{figure}

\subsection[Transfer operators]{Transfer operators for Markov iterated function schemes}
We explained in the previous subsection  that the dimension
$\dim_H(X_M)$ corresponds to the zero of the pressure
function $P$. This function can also be expressed in terms
of a linear operator on a suitable space of H\"older functions. 

In order to  accommodate the Markov condition it is convenient to consider the
space consisting of $d$ disjoint copies of $[0,1]$, which we denote by
$S : =\oplus_{j=1}^d [0,1]\times \{j\}$ and to introduce the maps 
$$
T_{j} \colon S \to S, \qquad  
T_{j} \colon (x,k) \mapsto (T_j(x),j).
$$
We omit the dependence on $d$ where it is clear.

The transfer operator associated to $\{T_j\}$ is a linear operator acting on the
space of H\"older functions $ C^\alpha(S) \colon = \oplus_{j=1}^d
C^\alpha([0,1]\times\{j\}) $, where $0 < \alpha \le 1$. 
This is a Banach space with the  norm $\|f_1, \cdots, f_d\| = \max_{1 \leq j
\leq d} \left\{ \|f_j\|_\alpha + \|f_j\|_\infty\right\}$ where 
 $\|f_j\|_\alpha = \sup_{x\neq y} \frac{|f_j(x) - f_j(y)|}{|x-y|^\alpha}$.
 For many applications we can take~$\alpha=1$. 
\begin{definition}
    \label{def:trop}
    For $0 \leq t \leq 1$ the {\it transfer  operator} $\mathcal L_t \colon
    C^\alpha(S) \to C^\alpha(S)$ 
    for the scheme $T_j \colon S \to [0,1] \times
    \{j\}$ 
    is defined by the formula  $\mathcal L_t\colon (f_1, \cdots, f_d) \mapsto (F_1^t, \cdots, F_d^t)$ where 
    $$
    F_k^t (x,k) = \sum_{j=1}^d M(j,k) \cdot f_j(T_{j} (x,k)) |T_{j}'( x,k )|^t. 
    $$
    We omit the dependence on $M$ where it is clear.
\end{definition}

\begin{rem}
    For some applications where the contractions form a
    Bernoulli system (i.~e. there are no restrictions and  all
    the entries of $M$ are $1$) it is sufficient to take one
    copy of the interval $[0,1]$, i.~e. $d=1$. This applies, for example, in
    all of the Zaremba examples. In this case the transfer operator takes
    simpler form 
    \begin{equation}
      \label{trop:eq}
      \mathcal L_t : C^\alpha([0,1]) \to C^\alpha([0,1]) \qquad 
      \mathcal L_t : f \mapsto \sum_{j=1}^d f(T_j) |T_j^\prime|^t. 
    \end{equation}
\end{rem}
\begin{definition}
    We say that a function $\underline f \in C^\alpha(S)$ is positive if $f =
    (f_1,\ldots, f_d)$ and each $f_j \in C^\alpha([0,1] \times \{j\})$ takes only
    positive values. 
\end{definition}
The connection between the linear operator $\mathcal L_t$ and
the value of the pressure function~$P(t)$ comes from the first part of the following
version of a standard result, cf.~\cite{Ruelle}. 
\begin{lemma}[after Ruelle]\label{spectral}
    Assume that the matrix~$M$ is aperiodic. 
    In terms of the
    notation introduced above, the spectral radius of  $\mathcal L_t$ 
    is~$e^{P(t)}$. Furthermore,  
    \begin{enumerate}
        \item $\mathcal L_t$ has an isolated maximal eigenvalue $e^{P(t)}$
            associated to a positive eigenfunction $\underline h \in
            C^\alpha(S)$
            and a positive eigenprojection\footnote{For a general
            introduction see~\cite{Kato}, Ch.3,
            \S6; we say that the eigenprojection is positive if elements of the positive cone
            are mapped into $\mathbb R^+ \underline h$.} 
            $    \eta \colon  C^\alpha(S) \to \langle \underline h\rangle $; and
        \item
            for any  $\underline f \in C^\alpha(S)$ we have 
            $$
            \|e^{-nP(t)}\mathcal L_t ^n \underline f 
            - \eta(\underline f) \|_\infty \to 0 \mbox{ as } n \to+\infty.
            $$
    \end{enumerate}
\end{lemma}

\begin{proof}
For the reader's convenience we sketch a proof adapted to the present context 
(see~\cite{R89} for a different proof). 

\paragraph{Part 1.}
{\color{black}

Let $c = \max_{1 \leq i \leq d} \|T_i'\|_\infty  < 1$
and $A = \max_{1 \leq i \leq d}  \|\log|T_i'\|_\alpha < 1$
and choose~$B$ sufficiently large that $A + Bc < B$.
We can define a convex space
$$
\mathcal C  = \left\{ \underline f = (f_1, \cdots, f_d) = C^\alpha (S)
\colon  
\begin{aligned}
&\forall 1 \leq j \leq d, \forall 0 \leq x \leq 1, f_j(x) \leq 1\cr
&\forall 1 \leq j \leq d,\forall 0 \leq x,y \leq 1,  f_j(x) \leq f_j(y) e^{B|x-y|^\alpha}\cr
\end{aligned}
\right\}
$$
 which is uniformly compact by the Arzela--Ascoli theorem.  
 Given $\underline f \in \mathcal C$  and $n \geq 1$ we can define
 $\Phi_n(\underline f ) = \frac{\mathcal L_t(\underline f + \frac{1}{n} )}{\|\mathcal
 L_t(\underline f + \frac{1}{n} )\|_\infty} \in C^\alpha(S)$.  For all $0 \leq
 x,y \leq 1$ 
 and $\underline f = (f_1, \cdots, f_d) \in \mathcal C$ and each $1 \leq j \leq
 d$, using $0 \le t \le 1$, we get
 $$
 F_k^t(x) \leq e^{(A+Bc)|x-y|^\alpha} \sum_{j=1}^d M(j,k) \cdot f_j(T_j(y))
 |T_j'(y)|^t \leq e^{B|x-y|^\alpha} F_k^t(y).
 $$
 Hence for each $n \geq 1$ and $0 \leq x,y \leq 1$ we have $\Phi_n(\underline
 f)(x) \leq e^{B|x-y|^\alpha}\Phi_n(\underline f)(y)$ and thus $\|\Phi_n(\underline
 f)(x) \|_\infty =1$.  
 Therefore, for each $n \geq 1$ the map $\Phi_n: \mathcal C \to \mathcal C$ is well defined and has a 
 nontrivial fixed point 
 $\underline h_n = \Phi_n(\underline h_n) \neq 0$ by the Schauder fixed point theorem.  
 We can take an accumulation point $\underline h$ of $\{\underline h_n\} \subset \mathcal C$.
 This is an eigenfunction of $\mathcal L_t$  for the eigenvalue $\lambda = \|\mathcal L_t\underline h\|_\infty$ 
and we can choose $x_0$ and $j$, say,  with 
 $|h_j(x_0)| = \|h\|_\infty$. 

 Since~$M$ is aperiodic we can choose~$N$ sufficiently large such that for
 arbitrary~$k$ and $j_1$ we have a  contribution 
 $h_{k}(T_{j_1} \circ \cdots \circ T_{j_N} x_0) |(T_{j_1} \circ \cdots \circ T_{j_N})'(x_0)|^t > 0$
to $\mathcal L_t^N$. It follows from the second condition of the definition of
$\mathcal C$ that $h_{k}(x_1) > 0$ where we let $x_1=T_{j_1} \circ \cdots \circ T_{j_N}x_0$.
Since $\underline h \in \mathcal C$ we conclude that for any~$x$ we have 
$ h_{k}(x) \geq h_{k}(x_1) e^{-B} > 0$.
 By Definition~\ref{def:pressure}  of the pressure function, 
 $$
 P(t) = \lim_{n \to +\infty} \frac{1}{n} \log \|\mathcal L_t^n \underline 1 \| 
 = \lim_{n \to +\infty} \frac{1}{n} \log \|\mathcal L_t^n \underline h \| 
 = \log \lambda.
 $$
 
 }

 \paragraph{Part 2.} Let $\Delta(\underline h): C^\alpha(S) \to C^\alpha(S)$ denote multiplication by 
$\underline h$ and then introduce the linear operator 
 $\mathcal K = \frac{1}{\lambda}\Delta(\underline h)^{-1} \mathcal L_t \Delta(\underline h)$, which now satisfies $\mathcal K 
 \underline 1 = \underline 1$.   
 This implies that for  $\underline f \in C^\alpha(S)$ we have that for each $\underline x\in S$ :
 \begin{enumerate}
 \item[(a)]
 the sequence 
 $ \inf_{0 \leq x \leq 1}\mathcal K^n\underline f (\underline x)$  ($n\geq 1$)
 is monotone increasing and bounded, and 
  \item[(b)]
  the sequence $\sup_{0 \leq x \leq 1}\mathcal K^n\underline f (\underline x)$
  ($n\geq 1$) is monotone decreasing and bounded.
  \end{enumerate}
The limits are fixed points for $\mathcal K$, and thus without too much effort  we can deduce that
$\mathcal K^n\underline f$ converges uniformly to a constant function.
Reformulating this for the original  operator $\mathcal L_t$ gives the claimed
result. 
  \end{proof}

 \begin{rem}
 The properties of the transfer operator
 in Lemma~\ref{spectral} 
  now clarify the reasoning behind the remaining parts of  Lemma~\ref{bowenRuelle}.  
 With a little more work (and the Fortet---Doeblin inequality) one can show that
 $e^{P(t)}$ is an isolated eigenvalue of $\mathcal L_t$ and thus has an analytic
 dependence on $t$ (compare with~\cite{PP90}).
 \end{rem} 
 
 \begin{rem}
   \label{rem:pdim2}
   To explain the idea behind Lemma~\ref{bowenRuelle}, it remains to recall
   why if $P(t)=0$ then $\dim_H(X) \geq t$.  

 Let $\mathcal M$ be the space of probability measures supported on~$S$ with the weak star topology. 
 By Alaoglu's theorem this space  is compact.  The map $\Psi: \mathcal M \to \mathcal M$
 defined by $[\Psi \mu](\underline g) = e^{-P(t)} \mu(\mathcal L_t \underline g)$
 for  $\underline g \in C^\alpha\left( S\right)$ has the fixed point~$\eta$.

 We would like to apply the  mass distribution principle (cf.~\cite[p.~67]{Falconer}) to the
 measure~$\eta$. {\color{black} In other words, to show that~$t$ corresponding
 to~$P(t)=0$ is a lower bound on~$\dim_H(X)$, it is sufficient to show that there exists $C>0$ for which on any
 small interval~$V$ we have $\eta(V) \le C |V|^t$. Moreover, it suffices to
 consider intervals of the form $V = T_{i_1}\circ \ldots \circ T_{i_n} U$ for
 some $i_1, \ldots, i_n$.} In particular, providing $P(t)=0$ for an allowed sequence
 $i_1,\ldots,i_n$ we have
 $$
\eta (T_{i_1} \circ \cdots \circ T_{i_n} U)
\asymp \int \mathcal  L_t^n \chi_{T_{i_1} \circ \cdots \circ T_{i_n} U} d\eta
\asymp |(T_{i_1} \circ \cdots \circ T_{i_n})'(0)|^t.
$$ 
 {\color{black} By compactness of the closure $cl( T_{i_1}\circ \ldots \circ T_{i_n} U)$ for~$x\in U$ we have
 $$
 \frac{|(T_{i_1}\circ \ldots \circ T_{i_n})^\prime(0)|}{|(T_{i_1}\circ
 \ldots\circ T_{i_n})^\prime(x)|} \le \sup_{y \in U} \exp\left(\left(\log |(T_{i_1}\circ \ldots \circ
 T_{i_n})^{\prime}(y)|\right)^\prime\right) < +\infty
 $$ 
 together with the mean value theorem we get $|(T_{i_1}\circ \ldots \circ T_{i_n})^\prime(0)|
 \asymp \mathop{\rm diam}(T_{i_0}\circ\ldots \circ T_{i_n} U)$.} Hence we
 conclude 
 $ \eta (T_{i_1} \circ \cdots \circ T_{i_n} U)
 \asymp\hbox{\rm diam} (T_{i_1} \circ \cdots \circ T_{i_n} U)^{t} $.  The mass distribution principle gives
 $$
 \dim_H(X) \geq \lim_{n\to +\infty} \frac{\log  \eta (T_{i_1} \circ \cdots \circ T_{i_n} U)}{\log \hbox{\rm diam} (T_{i_1} \circ \cdots \circ T_{i_n} U)^{t}} =  t.
 $$
  \end{rem}


This completes the preparatory material.  In the next section we will explain our basic methodology.

\section{Hausdorff dimension estimates}
\label{sec:estimates}

Effective estimates on the Hausdorff dimension come from two ingredients.  
\begin{enumerate}
\item  Min-max type inequalities presented in~\S\ref{ssec:minmax}, which give a way to rigorously
bound the largest eigenvalue $e^{P(t)}$ of the transfer operator~$\mathcal L_t$
using a \emph{suitable test function}~$\underline f$. 
\item The Lagrange---Chebyshev interpolation scheme described
in~\S\ref{test}, which gives a polynomial that can serve as the test function~$\underline f$. 
\end{enumerate}

The first is inspired by the corresponding result for
 the second eigenvalue of the transfer operator for the Gauss map in the work of
Wirsing~\cite{Wirsing}. The second part is inspired by the work of
Babenko--Yur'ev~\cite{Babenko} on the problem of Gauss. 

The accuracy of the estimates which come out from the min-max inequalities 
depend on the test function.
Lagrange---Chebyshev interpolation is a very classical method of approximating
holomorphic functions and perhaps first has been used in this setting by
Babenko--Yur'ev~\cite{Babenko}. Whereas various interpolation schemes have been used by
many authors to estimate $e^{P(t)}$, it is the combination of these two
ingredients that leads to particularly effective and accurate estimates. 
\subsection{The min-max inequalities}
\label{ssec:minmax}
Our analysis is based on the maximal eigenvalue $e^{P(t)}$ for $\mathcal L_t$
being bounded  using the following simple result. 

\begin{lemma}\label{lem:minmax}
\begin{enumerate}
\item
Assume there exists $a > 0$ and a positive function 
 $ \underline f  \in  C^\alpha(S) $ such that for all $x \in S$
  $$
    a \underline f(x) \leq \mathcal  L_t  \underline f(x)
    $$
then $a \leq e^{P(t)}$.
\item
Assume there exists $b > 0$ and a positive function 
 $ \underline g  \in  C^\alpha(S) $ such that for all $x \in S$
  $$
\mathcal  L_t  \underline g(x) \leq    b   \underline g(x)
    $$
then $e^{P(t)} \leq b$.
\end{enumerate}
\end{lemma}

\begin{proof}
    By iteratively applying $\mathcal L_t$ to both sides  of the inequality in part 1
     we have that for all $x \in S $ and $n \geq 1$
    $$
    a^n \underline f(x) \leq \mathcal  L_t^n \underline f(x)   
    $$
    and thus taking $n$'th roots and passing to the limit as $n \to +\infty$
    we have for all  $x \in S$ 
    $$
    a  \leq \limsup_{n \to +\infty}|\mathcal  L_t^n \underline f(x)|^{1/n} = e^{P(t)}     
    $$
    since~$\underline f(x)$ is strictly positive and
    $e^{-nP(t)}\mathcal  L_t^n \underline f$ converges uniformly
    to $\eta(\underline f) > 0$ by the second part
    of Lemma~\ref{spectral}.   This completes the proof of part 1.
    
    The proof of part 2 is similar.
    \end{proof}

Below we will use the following shorthand notation when working with
the Banach space $C^\alpha(S)$. 
\begin{notation}
    \label{not:rats}
  Given $f, g \in C^\alpha(S)$ we abbreviate 
    $$
         \sup_S\frac{ f}{ g }: = \sup_{1\le j \le d} \,\sup_{ x \in S } \frac{f_j(x)}{g_j(x)}  
       \qquad \mbox { and }  \qquad 
         \inf_S\frac{  f }{ g } : = \inf_{1\le j \le d} \, \inf_{ x \in S }
         \frac{f_j(x)}{g_j(x)}.
    $$    
\end{notation}

In particular, we can use Lemma~\ref{lem:minmax} to deduce a technical fact,
which is a basis for validation of all numerical results in the present work. 
{\color{black}
\begin{lemma}\label{tech}
    The Hausdorff dimension  $\dim_H X \in (t_0,t_1)$ if and only if 
    there exists two positive functions $\underline  f, \underline g \in C^\alpha(S)$  
    such that the following inequalities hold
    \begin{equation}
        \label{eq:tech}
\inf_S\frac{ \mathcal L_{t_0}\underline f}{\underline f } > 1
\quad \mbox{ and } \quad 
\sup_S\frac{ \mathcal L_{t_1}\underline g}{\underline g }< 1,
\end{equation}
\end{lemma}
\begin{proof}
    Assume first that $\dim_H X \in (t_0,t_1)$ then $P(t_0) > 0 >
    P(t_1)$ which is equivalent to $e^{P(t_0)} > 1 > e^{P(t_1)}$. 
    Then by part~1 of Lemma~\ref{spectral} there exist 
    positive eigenfunctions $\underline{h}_0$ and $\underline{h}_1$ of $\mathcal
    L_{t_0}$ and  $\mathcal L_{t_1}$, respectively, such that 
    $$
    \frac{ \mathcal L_{t_0}\underline h_0}{\underline h_0 } = e^{P(t_0)} > 1
\quad \mbox{ and } \quad 
\frac{ \mathcal L_{t_1}\underline h_1}{\underline h_1 } = e^{P(t_1)} < 1.
$$

Now assume that~\eqref{eq:tech} hold true for some $\underline f$ and
$\underline g$. Then by Lemma~\ref{bowenRuelle}
the first inequality in~\eqref{eq:tech} implies that the hypothesis of
    part~1 of Lemma~\ref{lem:minmax} holds with $a=\inf_S\frac{ \mathcal
    L_{t_0}\underline f}{\underline f } > 1 $. 
    We deduce that $e^{P(t_0)} \geq a > 1$ and thus $P(t_0) >  0$.  The second
   inequality  in~\eqref{eq:tech} implies that the hypothesis of  part 2 of
   Lemma~\ref{lem:minmax} holds with $b =  \sup_S\frac{ \mathcal L_{t_1}\underline g}{\underline g } < 1$. 
   We deduce that $e^{P(t_1)}  \leq b <  1$, thus $P(t_1) <  0$. 
    By the intermediate value theorem applied to the strictly monotone decreasing 
    continuous function~$P$ we see that the unique zero $t = \dim_H(X)$ for~$P$ satisfies
    $ t_0< \dim_H  X < t_1$, as required. 
\end{proof}
}

Therefore our aim in applications is to make choices of  $\underline  f = (f_i)_{i=1}^d > 0$ and  $\underline g = (g_i)_{i=1}^d > 0 $
in Lemma~\ref{tech} so that~$t_0$ and~$t_1$ are close, in order to get good estimates on $\dim_H(X)$.

\begin{rem}[Domains of test functions]
    It is apparent from the proof that we need only to consider the
    minima and maxima of the ratios $\mathcal L \underline
    f(x)/\underline f(x)$ and $\mathcal L \underline
    g(x)/\underline g(x)$ for  those $x\in S$ lying in the limit set.
    However, a compromise which simplifies the use of calculus would
    be to consider the minima and maxima over the smallest
    interval containing the limit set. 
\end{rem}

\begin{rem}[Min-max theorem]
    A more refined version which we won't require  is the  min-max result:
    $$
    e^{P(t)}= 
    \sup_{\underline f >0} \inf_{ S } 
    \frac{\mathcal L_{t} \underline f}{\underline f} 
    = 
    \inf_{\underline f >0}\sup_{ S } 
    \frac{\mathcal L_{t}\underline f}{\underline f}. 
    $$
 
  To see this, observe that by Lemma~\ref{lem:minmax} for any $\underline f,
  \underline g > 0$
    we have 
    $$
     \inf_{x \in S} 
     \frac{\mathcal L_t \underline f(x)}{\underline f(x)}
    \leq e^{P(t)} \leq  \sup_{x \in S} \frac{\mathcal L_t \underline
    g(x)}{\underline g(x)},
    $$
    and therefore 
     $$
     \sup_{\underline f>0}
     \inf_{S} \frac{\mathcal L_t \underline f}{\underline f}
    \leq e^{P(t)} \leq  
      \inf_{\underline g>0}
    \sup_{S} \frac{\mathcal L_t \underline g}{\underline g}.
    $$
    
    In particular, by the Ruelle operator theorem, the equality is realized when
    $\underline f = \underline g = \underline h>0$ is the eigenfunction associated to
    $e^{P(t)}$. We refer the reader to~(\cite{Mayer}, p.88) and (\cite{IK02}, \S2.4.2)   for more details.
\end{rem}

\subsubsection{Applying Lemma~\ref{tech} in practice}
In order to obtain good estimates on the Hausdorff dimension based on
Lemma~\ref{tech} it is necessary to construct a pair of functions $\underline f$
and~$\underline g$ and to rigorously verify the inequalities~\eqref{eq:tech}.
{\color{black}
We shall now explain how the verification has been realised in practice, i.e. in
our computer program. To simplify the exposition, we demonstrate the method
in the case of Bernoulli scheme. It will be clear how to generalise it to treat
a more general Markov case. 

Evidently for any interval $[t_0,t_1] \subset [0,1]$ and for any $x \in [t_0,t_1]$ we
have\footnote{Here by $\|f\|_\infty$ we understand $\sup_{[t_0,t_1]}|f|$.}
\begin{equation}
    \label{eq:ubound}
  \left| \frac{[\mathcal L_t f](t_0)}{f(t_0)} - \frac{[\mathcal L_t f](x)}{f(x)}
  \right| \le \left\| \Bigl( \frac{\mathcal L_{t} f}{ f} \Bigr)^\prime \right\|_\infty
    (t_1-t_0),
\end{equation}
therefore if we can get an upper bound on $\bigl\| \bigl( \frac{\mathcal L_{t} f}{
f} \bigr)^\prime \bigr\|_\infty$, then we can rigorously estimate the ratio by
taking a partition of $[0,1]$ and applying~\eqref{eq:ubound} on each interval of
the partition. 
Furthermore, it is clear that providing  $\bigl\| \bigl( \frac{\mathcal L_{t} f}{
f} \bigr)^\prime \bigr\|_\infty$ is small, we can allow a relatively coarse 
partition. 

We have the following simple useful fact. 
\begin{lemma}
    \label{lem:ratio}
    Let~$h$ be a positive eigenfunction of~$\mathcal L_t$ corresponding to the
    eigenvalue~$\lambda$. Then there exists a constant $r_1>0$ such that 
    for any approximation~$f$ such that $\|f - h\|_{C^1}
    < \varepsilon$ we have
    $\bigl\| \bigl( \frac{\mathcal L_{t} f}{ f} \bigr)^\prime
    \bigr\|_\infty < r_1 \varepsilon$.
\end{lemma}
        \begin{proof}
        By the hypothesis of the Lemma, we may write $f = h + f_\varepsilon$, where
        $\|f_\varepsilon\|_{C^1} = \max |f_\varepsilon| + \max
        |f_\varepsilon^\prime | < \varepsilon$, and $f > r_0 > 0$.
        Then the condition $\|f_\varepsilon\|_{C^1} < \varepsilon$ implies
        $\|f\|_{C^1} \le \|h\|_{C_1}+\varepsilon$, furthermore, we calculate
        \begin{align*}
        \left|([\mathcal L_t f_\varepsilon](x))^\prime \right| & \le \sum_{j=1}^d
         \left|\left(|T_j^\prime(x)|^s f_\varepsilon(T_j(x)) \right)^\prime
         \right| \\ &\le
         \sum_{j=1}^d \left( \left|(|T_j^\prime(x)|^s)^\prime
         f_\varepsilon(T_j(x))\right| +  \left|
         |T_j^\prime(x)|^{s+1} f_\varepsilon^\prime(T_j(x)) \right| \right) \le r_2 \varepsilon,
        \end{align*}
        where $r_2 = 2d\max(\|(|T_j^\prime|^s)^\prime\|_\infty,
        \||T_j^\prime(x)|^{s+1}\|_\infty)$.  
        By linearity of the transfer operator we have that 
        $(\mathcal L_t f )^\prime = (\mathcal L_t h )^\prime + (\mathcal L_t
        f_\varepsilon)^\prime = \lambda h ^\prime + (\mathcal L_t
        f_\varepsilon)^\prime$. 
        Now we have for the derivative of the ratio
        \begin{multline*}
        \Bigl|\Bigl(\frac{\mathcal L_t f }{f } \Bigr)^\prime \Bigr| = \Bigl|\frac{(\mathcal L_t
        f )^\prime f  - f ^\prime (\mathcal L_t f )}{f ^2}\Bigr| \\ \le \Bigl| \frac{ (\lambda
        h ^\prime + (\mathcal L_t f_\varepsilon)^\prime)\cdot (h  +
        f_\varepsilon) - (h ^\prime + f^\prime_\varepsilon)\cdot( \lambda h  +
        \mathcal L_t f_\varepsilon)}{r_0^2}\Bigr| \\ \le \frac{\varepsilon}{r_0^2}  (r_2
         \|f \|_\infty +  \lambda \|h ^\prime\|_\infty
        + \|\mathcal L_t f \|_\infty + r_2\|h ^\prime\|_\infty  ).
    \end{multline*}
    Taking into account that $\|f\|_{C^1} \le \|h\|_{C^1}+\varepsilon$ we may
    now choose $r_1$, which depends on the norm of the
    eigenfunction~$\|h\|_{C^1}$, but is independent of the choice of
    approximation~$f$, such that 
    $\bigl\|\bigl(\frac{\mathcal L_t f }{f } \bigr)^\prime \bigr\|_\infty \le r_1
    \varepsilon$.
  \end{proof}

In practice, the derivative $\bigl( \frac{\mathcal L_{t} f}{ f} \bigr)^\prime$
can be effectively estimated using the following computer-assisted approach to
arbitrary precision and without excessive effort. 
We construct our test function~$f$ as a polynomial of degree~$m$ and our IFS
consists of linear-fractional transformations $T_j(x) = \frac{a_j x + b_j}{c_j x
+ d_j}$, the image functions $F^t = \mathcal L_t f$ can be written as 
$$
F^t (x) = \sum_{j=1}^d \frac{|a_j d_j - b_j c_j|^t}{|c_j x+d_j|^{2t} } f
\left(  \frac{a_j x + b_j}{c_j x + d_j}\right) = 
\sum_{j=1}^d  \frac{|a_j d_j - b_j c_j|^t}{(c_j x+d_j)^{2t} (c_j x+d_j)^m} p_j(x),
$$
where $p_j$ are polynomials of degree $m$, whose coefficients can be computed
with arbitrary precision. Then for the derivative of the ratio we
obtain 
\begin{equation}
    \label{eq:dif-err}
\left( \frac{F^t(x)}{f(x)}\right)^\prime  = \sum_{j=1}^d \frac{  |a_j d_j - b_j
c_j|^t }{(c_j x +
d_j)^{2t}(c_j x+d_j)^{m+1} f^2(x)} \widehat p_{j}(x),
\end{equation}
where 
$$
\widehat p_{j}(x) = -(m+2t) p_j(x) f(x) - (c_j x+ d_j) (f (x)p_j^\prime(x) - f^\prime(x) p_j (x) ) 
$$
is a polynomial of degree~$2m$ whose coefficients can be computed explicitly
with arbitrary precision chosen. The computation of these coefficients, together
with the coefficients of~$f$, allowed us to obtain accurate estimates on the derivative
 $\bigl(\frac{F^t}{f}\bigr)^\prime$ on the entire interval~$[0,1]$ 
 using ball arithmetic~\cite{JohArb} in all the examples we considered.  
 
\begin{rem} 
Since the ratios are analytic functions, one would expect that they can be
approximated by polynomials. A heuristic observation suggests that for an IFS of
analytic contractions, $\varepsilon \sim 10^{-3m/4}$ in Lemma~\ref{lem:ratio}, where~$m$ is the degree of the
approximating polynomial. 
\end{rem}

\begin{rem}
        The formulae for~$F^t$, $\frac{F^t}{f}$ and~$\widehat p_{j}$ given above have been used in the actual computer program written 
        in~C to study the Examples we have in the paper. For the iterated function schemes which are not 
        linear fractional transformations, the formulae, indeed, will be
        different, and in particular,~$\widehat p_{j}$ may not be
        a polynomial. The same method applies, but the computation might require more
        computer time. The implicit constant~$r_2$ also affects the accuracy of
        the estimate. 
    \end{rem}
    }  
\subsection[Interpolation]{Lagrange---Chebyshev Interpolation}\label{test}
We might fancifully note that if we had an
{\it a priori} knowledge of the true eigenfunction~$\underline
h$ for $\mathcal L_{t}$ corresponding to the maximal eigenvalue  and took this choice for 
$\underline f$  in Lemma  \ref{lem:minmax} then we would immediately
have $a = b = e^{P(t)}$.   However, in the absence of
a knowledge of the eigenfunction  our strategy is to find an approximation.

\begin{enumerate}
    \item[(a)]  
        Choose $t_0 < t_1$ as potential lower and upper bounds, respectively, in Lemma \ref{tech} based on 
        an heuristic estimate on the  dimension  (for example
        using the periodic point method, or bounds that we
        would like to justify from any other source); and   
    \item[(b)]
        Find candidates for $\underline f$ and $\underline g$ which 
        are close to the eigenfunctions $\underline h_{t_1}$ and $\underline h_{t_0}$ for the operators 
        $\mathcal L_{t_1}$ and $\mathcal L_{t_0}$, respectively.
        We do this by approximating the two operators by finite
        dimensional versions and using their eigenfunctions  for~$\underline f$
        and~$\underline g$ (in the next section). 
\end{enumerate}

There are different possible ways to find the functions we require in (b) in the previous section.  
We will use  classical Lagrange   interpolation~\cite{Trefethen}.

\bigskip
\noindent
{\bf Step 1 (Points)}.  Fix $m \geq 2$.  We can then consider  the 
zeros of the Chebyshev polynomials: 
$$x_k = \cos\left( \frac{\pi (2 k -1)}{2 m}\right)  \in [-1,1], \mbox{ for } 1 \leq k \leq m.$$  
In the present context it  is then convenient to translate them
to the unit interval by setting $y_k = (x_k+1)/2 \in [0,1]$.  

\bigskip
\noindent
{\bf Step 2 (Functions)}.
We can  use the values $\{y_k\}$ to define the associated Lagrange interpolation polynomials:
\begin{equation}
    \label{lagcheb:eq}
lp_k(x) =  \frac{\prod_{i\neq k} (x - y_i)}{ \prod_{i\neq k} (y_k - y_i)}, \mbox{ for }
1 \leq k \leq m
\end{equation}
which are the polynomials of the minimal degree with the property that
$lp_k(y_k)=1$ and $lp_k(y_j) = 0$ for $j \neq k$ for all $1 \le j,k\le m$.
These polynomials span an $m$-dimensional subspace $\langle lp_1, \cdots, lp_m \rangle \subset C^\alpha([0,1])$. 

To allow for the Markov condition we need to consider the $(d\times m)$-dimensional subspace of $C^\alpha(S)$. 
To define it we simply consider $d$ copies $lp_{k,i}: [0,1]\times\{i\} \to
\mathbb R$ ($1 \leq i \leq d$)
of the Lagrange polynomials given by $lp_{k,i}(x,i) \equiv lp_k(x)$. 

\bigskip
\noindent
{\bf Step 3 (Matrix)}.
The polynomials $lp_{k,i}$ for $i = 1,\ldots, d$ and $ k=1,\ldots,m$ 
constitute a basis of an $(d\times m)$-dimensional subspace $\mathcal E : = \langle
lp_{k,i} \rangle \subset C^\alpha(S)$. 
Using definition~\ref{def:trop} of the transfer operator we can write for any
$0<t<1$ and $(f_1, \ldots, f_d) \in \mathcal E$:  
\begin{align}
\label{ltfi:eq}
(\mathcal L_t (f_1, \ldots, f_d))_j 
&= \sum_{i=1}^d M(i,j) f_i ( T_{i} ) |T_{i}^\prime|^t, \quad 1 \le i \le d, \\  
(\mathcal L_t (f_1, \ldots, f_d))_j & \colon  [0,1] \times \{j\} \to
\mathbb{R}. \notag
\end{align}
For each $1 \leq i,j \le d$ and $1 \leq k,l \leq m$ we can introduce the $m\times
m$ matrix $B^t_{ij}(l,k)$ 
\begin{equation}
    \label{Bsmall:eq}
    B^t_{ij}(l,k) \colon =
    lp_{k,i}(T_{i}(y_{l,j}))\cdot|T_{i}^\prime(y_{l,j})|^t .
\end{equation}  
Then we can apply the operator $\mathcal L_t$ to a basis function
$(0, \cdots, 0, lp_{k,i},0, \cdots, 0) \in \mathcal E$ and evaluate the
resulting function at the nodes $y_{l,j} = (y_l,j) \in [0,1]\times\{j\}$, $1 \le
l \le m$, $ 1 \le j \le d$ 
\begin{align*}
(\mathcal L_t (0, \cdots, 0, lp_{k,i},0, \cdots, 0)
)_j(y_{l,j}) &= M(i,j) lp_{k,i}(T_{i}(y_{l,j})) |T^\prime_{i}(y_{l,j})|^t 
\\ &= M(i,j) \cdot B^t_{i,j}(l,k).
\end{align*}

Taking into account that the polynomials $lp_{k,i}$ constitute a basis of the
subspace $\mathcal E$ we can approximate the restriction $\mathcal
L_t|_{\mathcal E}$ by a $md \times md$ matrix: 
\begin{equation}
    \label{eq:Btmatrix}
B^t = 
\begin{pmatrix}
    M(1,1)B^t_{11} & \cdots & M(1,d)B^t_{1d}\cr
    \vdots & \ddots & \vdots\cr
    M(d,1)B^t_{d1} & \cdots & M(d,d)B^t_{dd}\cr
\end{pmatrix}.  
%
\end{equation}
For large values of~$m$ the maximal eigenvalue of the
matrix~$B^t$ will be arbitrarily close to the
maximal eigenvalue $e^{P(t)}$ of the transfer operator~$\mathcal L_t$, see~\S\ref{ss:verify}. 

\bigskip
\noindent
{\bf Step 4 (Eigenvector)}.
We can consider the left eigenvector of $B^t$ \begin{equation}
    \label{eq:vtvect}
v^t = (v^t_{1,1}, \cdots , v^t_{1m}, v^t_{21}, \cdots, v^t_{2,m}  \cdots, v^t_{d1}, \cdots v^t_{dm})
\end{equation}
corresponding to the maximal eigenvalue and use it to define a function $(f^t_1, \cdots, f^t_d) \in \mathcal E$ as 
a linear combination of Lagrange polynomials
\begin{equation}
    \label{eq:ftfunc}
f^t_i = \sum_{j=1}^m v^t_{ij} lp_{i,j} , \quad 1 \le i \le d.
\end{equation}

\begin{rem}
    In many cases the calculation can be simplified. More precisely, in the
    construction above, the points $y_{l,j} \equiv y_l$ do not depend on $j$.
    Therefore the matrices $B_{ij}(l,k)$ do not depend on $j$ either, and instead
    of computing $d^2$ matrices $B_{ij}(l,k)$ it is sufficient to compute $d$
    matrices $B^t_i(l,k)$.
\end{rem}

{
\color{black}
It is not immediately clear that the polynomials given by~\eqref{eq:ftfunc} are
positive. In Proposition~\ref{prop:analytic} in \S\ref{ss:verify} below we show
that for an iterated
function scheme of analytic contractions the algorithm presented above gives
positive functions provided~$m$ is sufficiently large. However, we don't have
a priori bounds on~$m$. Therefore, for every example we consider, we rigorously
verify that the function constructed is positive using the following simple
method (and if the function turns out not to be positive, we increase~$m$). 

 Since our~$f_j^t$'s are polynomials, their
        derivatives are easy to compute symbolically. We then take a uniform partition of the
        interval into $2^{10}$ intervals. For each interval $(a,b)$ of the
        partition, we compute the following:
        \begin{enumerate}
            \item The middle point $c = \frac12(a+b)$ and half-length $r =
                \frac12(b-a)$.
            \item The first~$m-1$ derivatives at~$c$: $f_j^{(k)}(c)$ for $k = 1, \ldots,
                m-1$.
            \item The image of the interval under the~$m$'th derivative: $(a_1, b_1) = |f_j^{(m)}(a,b)|$, this is done using ball
                arithmetic. The inequality $\max_{(a,b)} |f_j^{(m)}| \le b_1$
                is guaranteed by the Arb library~\cite{JohArb}.
        \end{enumerate}
        Then we can calculate a lower bound on $f_j$ on $(a,b)$:
        $$
        f_j(x) \ge f_j(c) - (r |f_j^{(1)}(c)| + r^2 |f_j^{(2)}(c)|+\ldots + r^{m-1}
        |f_j^{(m-1)}(c)| + r^m b_1 ) \mbox{ for all } x \in (a,b).
        $$
}


\subsection{Bisection method}
\label{bis}
The approach described in the previous two sections can be used not only to verify
given estimates $t_0 < \dim_H X < t_1$ but also to compute the Hausdorff
dimension of a limit set of a Markov iterated function scheme  with any desired accuracy using a basic
bisection method. Assume that given $\varepsilon>0$ we would like to find an
interval $\dim_H X \in (d_0,d_1)$ of length $d_1-d_0 = \varepsilon$. 

We begin by fixing a value of $m$, say $m=6$. Then we pick $t_0 < t_1$  for which we know $P(t_0) > 1 >
P(t_1)$ (for a one-dimensional Iterated Function Scheme  a safe choice is $t_0=0$, $t_1 = 1$) and compute $q = \frac{1}{2} (t_0 + t_1)$.
Using Lagrange--Chebyshev interpolation, the method
described in~\S\ref{test}, we compute the matrix $B^{q}$ defined
by~\eqref{eq:Btmatrix}. We then calculate its left eigenvector $v^{q}$ using the
classical power method  
and construct the corresponding function~$\underline f^{q} \in \mathcal E$
according to~\eqref{eq:ftfunc} as a sum of Lagrange polynomials. We verify that
$\underline f^q$ is positive, using the approach explained above. Having the function
$\underline f^{q}$ we compute the minima and the maxima of the ratio  
$$
a^\prime \colon = \inf_{S} \frac{\mathcal L_{q} \underline
f^{q}}{\underline f^{q}}, \qquad
b^\prime \colon = \sup_{S} \frac{\mathcal L_{q} \underline f^{q}}{\underline f^{q}}
$$
Then there are three possibilities
\begin{enumerate}
    \item if $a^\prime > 1$ we deduce by Lemma~\ref{tech} that $\dim_H X \ge
        q$ and move the left bound of the interval to $t_0 = q$,
    \item if $b^\prime < 1$ we deduce from Lemma~\ref{tech} that $\dim_H X \le
        q$ and move the right bound of the interval to $t_1 = q$,
    \item if $a^\prime \le 1 \le b^\prime$ we increase $m$,
\end{enumerate}
and repeat the process described above, see Figure~\ref{fig:flowchart} for a
flowchart. 

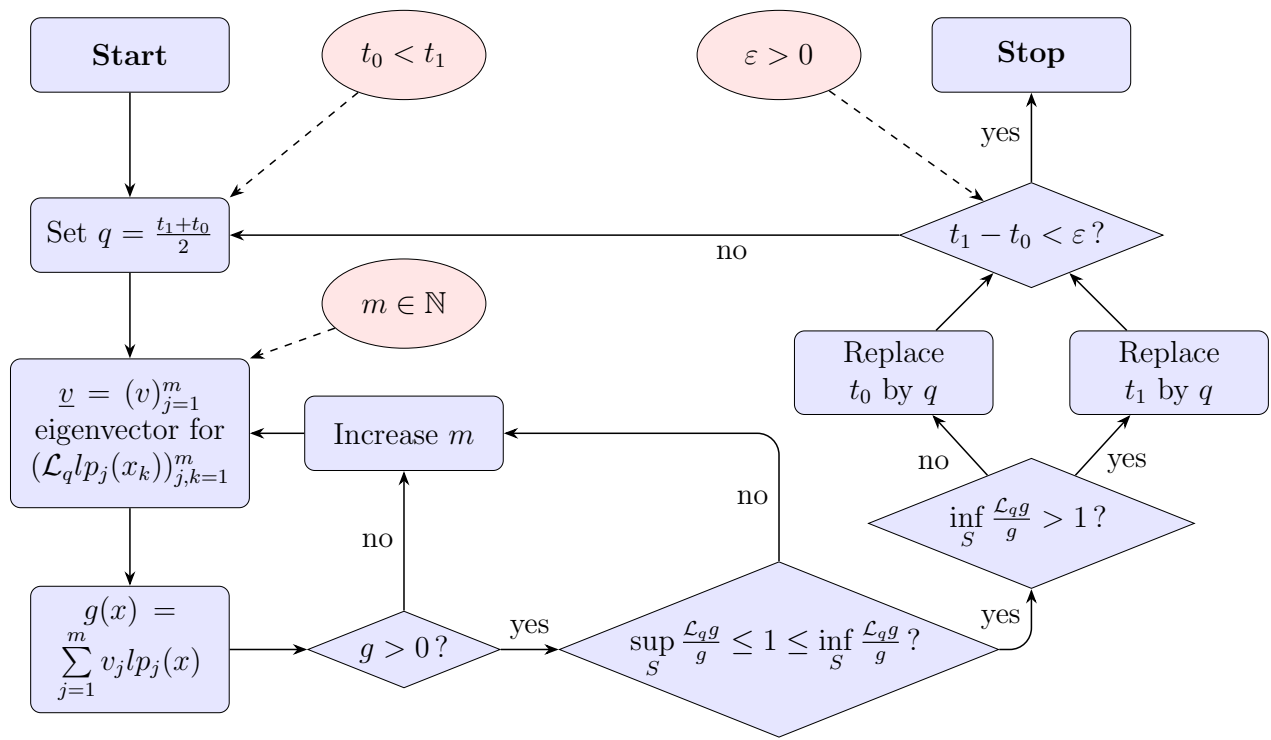
\begin{figure}
    \resizebox{17cm}{!}{
    \begin{tikzpicture}[node distance = 2cm, auto,
        arr/.style = {semithick, rounded corners=3mm, -Stealth},
        ]

        \node [blockM] (expert) {\bf Start};
        \node [cloud, right of=expert, node distance=3.8cm] (ts) {$t_0 < t_1$};
        \node [blockM, below of=expert, node distance=2.5cm] (init) {Set $q = \frac{t_1+t_0}{2}$};
        \draw [arr] (expert) -- (init);
        \draw [arr,dashed] (ts) -- (init.north east);
        \node [blockL, below of=init, node distance=2.75cm] (eigv) {
        $\underline v = (v)_{j=1}^m$ eigenvector for $(\mathcal L_q lp_j(x_k))_{j,k=1}^m$};
        \draw [arr] (init) -- (eigv);
        \node [blockM, below of=eigv, node distance=3.0cm] (evaluate) {$g(x)
        = \sum\limits_{j=1}^m v_j lp_j(x)$};
        \draw [arr] (eigv) -- (evaluate);
        \node [decisionS, right of=evaluate, node distance=3.8cm] (pos) {$g>0$\,? };
        \draw [arr] (evaluate) -- (pos);
        \node [decisionL, right of=pos, node distance=5.2cm] (decide) {$\sup\limits_S \frac{\mathcal L_qg}{g}
        \le 1 \le \inf\limits_S \frac{\mathcal L_qg}{g} $\,?};
        \draw [arr] (pos) -- node {yes} (decide);
        \node [blockS, above of=pos, node distance=3.0cm] (update) {Increase $m$ };
        \node [cloud, above of=update, node distance=1.8cm] (m) {$m \in \mathbb N$};
        \draw [arr,dashed] (m) -- (eigv.north east);
        \draw [arr] (decide) |- node [near start] {no} (update);
        \draw [arr] (pos) -- node {no} (update);
        \draw [arr] (update) -- (eigv);
        \node [decisionM, right of=init, node distance=12.5cm] (decideagainagain) {$t_1 - t_0 < \varepsilon$\,?};
        \draw [arr] (decideagainagain) -- node [near start] {no} (init);
        \node [blockM, below left of=decideagainagain, node distance=2.7cm] (updatetagain) {Replace $t_0$ by $q$};
        \node [blockM, below right of=decideagainagain, node distance=2.7cm] (updatet) {Replace $t_1$ by $q$};
        \node [decisionM, below of=decideagainagain, node distance=4.0cm] (decideagain) {$\inf\limits_S \frac{\mathcal L_qg}{g} > 1 $\,?};
        \draw [arr] (decide) -| node [near end] {yes} (decideagain);
        \draw [arr] (decideagain) -- node [pos=0.25,right,shift={(0.1cm,-0.05cm)}] {yes} (updatet);
        \draw [arr] (decideagain) -- node {no} (updatetagain);
        \draw [arr] (updatet) -- (decideagainagain);
        \draw [arr] (updatetagain) -- (decideagainagain);
        \node [cloud, right of=expert, node distance=9.0cm] (varepsilon) {$\varepsilon > 0$};
        \draw [arr,dashed] (varepsilon) -- (decideagainagain);
        \node [blockM, right of=varepsilon, node distance=3.5cm] (stop) {\bf Stop};
        \draw [arr](decideagainagain) -- node {yes} (stop) ;
    \end{tikzpicture}
    }
    \caption{The flow diagram summarizes how a computer programme implements the
    bisection procedure to find a small interval 
    $[t_0,t_1]$ of small size $\varepsilon > 0$ containing $\dim_H(X)$.}
    \label{fig:flowchart}
\end{figure}

{
\color{black}
\subsection{The convergence of the algorithm}
\label{ss:verify}
In this section we show that our method gives estimates 
on the dimension which are arbitrarily close to the true value, thus opening up
the possibility of arbitrarily close estimates with further computation. 

All of the examples we consider in the present work have an iterated scheme
consisting of one-dimensional real analytic contractions.
For simplicity we state the next Proposition for a single interval~$I$, which
corresponds to the Bernoulli case with $d=1$ in \S\ref{test}, but it will be clear how to extend this case to
the more general Markov setting. 
\begin{prop}
    \label{prop:analytic}
    Let $T_1, \cdots, T_d: I \to I$ be an iterated  function scheme with real analytic contractions with
    $ 0 < \inf_{x\in I}|T_j^\prime(x)|   \leq   \sup_{x\in I}|T_j^\prime(x)| < 1$ for $j=1,
    \cdots, d$. Assume that $P(t) \ne 0$ (i.e. $\dim_H(X) \ne t$). Then for any~$m$ sufficiently large the polynomial~$f^t$ defined
    by~\eqref{eq:ftfunc} satisfies one of the inequalities of Lemma~\ref{tech},
    in other words we have either $\inf_I \frac{\mathcal L_t f^t}{f^t} > 1$ or 
    $\sup_I \frac{\mathcal L_t f^t}{f^t} < 1$. 
\end{prop}
The proof of the Proposition which we will give here consists of two steps:
The first step is to construct a subspace of analytic functions in $C^\alpha(I)$ such that the
restriction of the transfer operator onto it has the right spectral properties. 
The second step is to construct an approximation of
the operator acting on the subspace of analytic functions 
by a finite rank operator acting on the subspace of polynomials of degree~$m$.
We begin our preparations for the proof of the Proposition by introducing the
subspace of analytic functions and defining the operator there. 
First, we need to introduce a suitable domain of analyticity. 
\begin{definition}
Given~$\rho>1$ we define an ellipse with the foci~$0$ and~$1$ by
\begin{equation}
    \label{eq:defDu}
 \partial U_\rho = \left\{ z = \frac{1}{2} + \frac{1}{4} \left(\rho e^{i \theta} + \frac{e^{-i\theta}}{\rho}\right)  
  \colon  0 \leq \theta < 2\pi  \right\}.
\end{equation}
     It is often referred to as a \emph{Bernstein ellipse}~\cite{Trefethen}.
\end{definition}
In the setting and under the hypothesis of Proposition~\ref{prop:analytic}
without loss of generality we may assume the following:
\begin{enumerate}
\item[(a)]
 There exists $\rho>1$ such that each contraction~$T_j$ extends to a complex
 domain $U_\rho \supset [0,1]$ bounded by the ellipse~$\partial U_\rho$, 
 such that $T_j^\prime(z) \neq 0$ for any~$z \in U_\rho$; and 
 \item[(b)] 
 The closures~$\mbox{cl}(T_jU_\rho)$ of the images~$T_jU_\rho$
 satisfy~$\mbox{cl}(T_iU_\rho) \subset U_\rho$
 for all $j=1, \cdots, k$. 
 \end{enumerate}
It is clear that an ellipse satisfying (a) and (b) exists. More precisely, since  we assume real
analyticity  of the~$T_j$ we can choose an 
elliptical domain sufficiently close to $[0,1]$ and by the hypotheses of the
Proposition we can deduce~(a). Since the~$T_j$ contract we can choose the
ellipse in~(a) sufficiently close to~$I$ (by making~$\rho$ close to~$1$) that~(b) holds.   
In what follows, we shall simplify notation and omit the index~$\rho$.

After introducing the domain of analyticity, we now define a Banach space of analytic functions.
\begin{definition}
Let $H^\infty$ denote the space of bounded analytic functions on~$U_\rho$
with the norm $\|f\|= \sup_{z\in U_\rho}|f(z)|$. 
\end{definition}
The space~$H^\infty$ is special case of Hardy spaces (hence the choice of
notation), and it is known to be a Banach space~\cite{koosis}.  
For any function~$f \in H^\infty$ the restriction~$f|_I$ is a
continuously differentiable function on~$I$ and, in particular, it is
contained in~$C^\alpha(I)$ for any $0<\alpha\le1$. More precisely, we can choose a
simple closed curve $\Gamma \subset U_\rho$ 
close to $\partial U_\rho$ and use Cauchy's theorem to write the
derivative for the restriction $f|_I$ by
$$
f^\prime(x) = \frac{1}{2\pi i} \int_\Gamma \frac{f(z)}{(z-x)^2} dz,
$$
which is continuous as a function of $x\in I$.

Note that~$T_j^\prime$ is real analytic and non-vanishing on~$I$ and the same holds true for
$|T_j^\prime|^t$. Thus by slight abuse of notation we may also denote by
$|T_j^\prime(z)|^t$ its analytic extension to the domain bounded by the sufficiently small Bernstein
ellipse~$U_\rho$. Therefore we may now introduce a restriction of the transfer operator~\eqref{trop:eq}
onto~$H^\infty$. 
\begin{definition}
For $t>0$, the transfer operator $\mathcal  L_{t}: H^\infty \to H^\infty$
will again be given by the formula
\begin{equation}
    \label{eq:Lext}
[\mathcal  L_{t} f] (z) = \sum_{j=1}^k |T_j'(z)|^t f(T_jz), \quad  z\in  U_\rho.
\end{equation}
\end{definition}

The operator $\mathcal L_{t}\colon H^\infty \to H^\infty$ given by~\eqref{eq:Lext} 
is actually compact and even nuclear~\cite{Mayer},\cite{Ruelle} although this will not be
needed.  We will use the following simpler fact instead. 
\begin{lemma} 
    \label{lem:maxeig}
    The operator $\mathcal L_{t}\colon H^\infty \to H^\infty$ shares the same
    maximal eigenvalue~$\lambda = e^{P(t)}$ as 
  $\mathcal  L_{t}\colon C^\alpha(I) \to C^\alpha(I)$ and the rest of the spectrum is contained
  in a disk of strictly smaller radius. 
  \end{lemma}
  \begin{proof}
  Indeed, if we let~$\mathds{1}$ denote
  the constant function on~$[0,1]$ then by Lemma~\ref{spectral} we see that
  $e^{-nP(t)} \mathcal L_n^t \mathds{1}$ converges uniformly to
  $\eta(\mathds{1}) \in \langle \underline h \rangle$. However, since
  $\mathds{1} \in H^\infty$ and~$\mathcal L_t^n$ preserves~$H^\infty$ we can
  conclude that~$\eta(\mathds{1}) \in H^\infty$ and 
  thus $\underline h$ has an extension in $H^\infty$ and~$e^{P(t)}$ is a simple
  eigenvalue for~$\mathcal L_t \colon H^\infty \to H^\infty$. 
  Similarly, any eigenvalue for~$\mathcal L_t\colon H^\infty \to
  H^\infty$ must be an eigenvalue for~$\mathcal L_t \colon C^\alpha(I)
  \to C^{\alpha}(I)$ since $H^\infty \subseteq C^{\alpha}(I)$.
  Therefore, since~$\mathcal L_t \colon C^\alpha(I)\to C^\alpha(I)$ has the rest
  of the spectrum in a disk of the radius strictly small than~$e^{P(t)}$ this property persists for the
  operator on~$H^\infty$. 
  \end{proof}

\noindent This completes the first step of the proof of
Proposition~\ref{prop:analytic} outlined above. 

Recall that the Lagrange polynomials $lp_1, \dots,lp_m$ given
by~\eqref{lagcheb:eq} form a basis of  
the subspace of polynomials of degree~$m-1$ in~$H^\infty$; in \S\ref{test} we named this
subspace~$\mathcal E$. Let $\mathcal P_m: H^\infty  \to \mathcal E$ be 
the natural projection given by the collocation formula
\begin{equation}
    \label{eq:project}
[\mathcal P_m f](x) = \sum_{j=1}^m f(x_j)  lp_j(x), \quad x \in I.
\end{equation}
We see immediately that the restriction~$\mathcal P_m|_{\mathcal E}$ of~$\mathcal P_m$
to~$\mathcal E$ is the identity. 

The second step is to approximate the transfer operator
on~$H^\infty$ by a finite rank operator. We now recall an estimate on the norm of the
difference $\|\mathcal L_t - \mathcal L_t \mathcal P_m\|_{H^\infty}$. 
\begin{lemma}[see \cite{bs}, Theorem 3.3]
    \label{lem:approx}
   For the transfer operator on~$H^\infty$ given by~\eqref{eq:Lext} there exist $C  > 0$ and $0 < \theta < 1$ such that 
   $\| \mathcal L_t - \mathcal L_t \mathcal P_m\|_{H^\infty}  \leq C
   \|\mathcal L_t\|_{H^\infty} \theta^m$ for $m \geq 1$.
\end{lemma}
This is also implicit in~\cite[\S2.2]{wormell}.
\begin{rem}
The proof of Lemma~\ref{lem:approx} relies on the fact that
for any function~$f$ analytic on a domain bounded by a Bernstein ellipse the image~$\mathcal L_tf$ is
analytic on a larger domain bounded by another Bernstein ellipse.
This form of analyticity improving property is also essential in showing
that~$\mathcal L_t$ is nuclear~\cite{Ruelle}. 
\end{rem}

It follows from Lemmas~\ref{lem:maxeig},~\ref{lem:approx} and classical analytic 
perturbation theory (see the book of Kato~\cite[Chapter VII]{Kato}) that 
we have the following:
\begin{cor}\label{apert}
For any $\varepsilon>0$ there exists~$\delta>0$ sufficiently small such that for
all~$m$ sufficiently large we have
\begin{enumerate}
\item
 $\mathcal L_t \mathcal P_m:  H^\infty  \to H^\infty$ has a simple
 maximal eigenvalue $\lambda_m$ with $|\lambda_m - \lambda| < \varepsilon$; 
 \item  The rest of the spectrum of~$\mathcal L_t \mathcal P_m$ is contained in $\{z\in \mathbb C 
\colon |z|\ < \lambda - 2\delta\}$; and 
\item The eigenfunctions~$h_m$ of~$\mathcal L_t \mathcal P_m$ converge to the
     eigenfunction~$h$ of~$\mathcal L_t$, more precisely $\|h_m - h\|_{H^\infty} \to 0$ as
    $m \to \infty$. 
\item The exists a constant~$c>0$, independent of~$m$, such that the eigenfunction~$h_m$  for $\mathcal L_t
    \mathcal P_m$ corresponding to~$\lambda_m$  satisfies $|h_m(z)|> c$ for all $z \in U_\rho$. 
\end{enumerate}
\end{cor}
We are now ready to prove Proposition~\ref{prop:analytic}. \\
{\it Proof} (of Proposition~\ref{prop:analytic}).
It follows from~\eqref{eq:project} that the restriction $\mathcal P_m \mathcal L_t|_{\mathcal E} $ to $\mathcal E$ is a finite rank operator 
 $\mathcal P_m \mathcal L_t\colon \mathcal E \to \mathcal E$ given by 
  $$
\mathcal P_m   \mathcal L_t\colon  f \mapsto \sum_{j=1}^n [\mathcal L_t f](x_j)
\cdot lp_{j}. 
 $$
 In the basis of Lagrange polynomials $\{lp_j\}_{j=1}^m$ the operator $\mathcal
 P_m \mathcal L_t$ is given by the $m \times m$-matrix $B^t = B^t(j,k)$,
 $j,k=1,\ldots,m$, where
  $$
  B^t(j,k) = \sum_{i=1}^m [\mathcal L_t
  lp_j](x_i) \cdot lp_i(x_k) =  [\mathcal L_t lp_j](x_k) \qquad \mbox{ for } 1
  \leq j,k \leq m;
  $$
which agrees with the matrix given by~\eqref{Bsmall:eq} in a special case~$d=1$.
A straightforward computation gives that the eigenvalue~$\lambda_m$ for~$\mathcal L_t\mathcal P_m$ is also an
eigenvalue for the matrix~$B^t$ corresponding to the eigenvector~$\mathcal P_m h_m \in
\mathcal E$. Since we have chosen the basis of Lagrange polynomials to define
the matrix~$B^t$, we conclude that 
$$
[\mathcal P_m h_m](x) = \sum_{j=1}^m v_j lp_j(x),
$$
where~$(v_1,\dots,v_m)$ is the eigenvector of~$B^t$. 

To see that~$\mathcal P_m h_m$ is a positive function, we apply a classical
result by Chebyshev~\cite{Trefethen}, which gives
\begin{equation}
    \label{eq:chebint}    
\sup_I |h_m - \mathcal P_m h_m | \le \frac{1}{2^{m-1} m!} \sup_I |h_m^{(m)}|.
\end{equation}
Since~$h_m \to h$ as~$m\to \infty$ in the space of analytic functions
$H^\infty$, the derivatives~$h_m^{(m)}$ are uniformly bounded on~$I$. Therefore 
the positivity of~$h_m$ on~$I$ guarantees the positivity of the projection $\mathcal P_m h_m
\in \mathcal E$ on~$I$ for sufficiently large~$m$.  

It remains to show that one of the inequalities in Lemma~\eqref{tech} holds
true for~$\mathcal P_m h_m$. Without loss of generality we may assume
that~$P(t)>0$, which implies $\frac{\mathcal L_t h}{h} = e^{P(t)}>1$. Therefore the
part~4 of Corollary~\ref{apert} together with~\eqref{eq:chebint} gives $\|\mathcal P_m
h_m - h\|_{H^\infty} \to 0$ as $m \to \infty$. Hence for~$m$ sufficiently large we shall have 
$\frac{\mathcal L_t \mathcal P_m h_m}{\mathcal P_m h_m} > 1$ everywhere on~$I$.
The case $P(t)<0$ is similar.

This completes the proof of Proposition~\ref{prop:analytic}. 
 \mbox{}\hfill $\square$ \par 
\begin{rem}
To avoid confusion we should stress that positivity of the polynomial $P_m
h_m \in \mathcal E$ and positivity of the entries of $B^t$ are not related;
furthermore, in the examples we consider the matrices $B^t$ typically have
entries of both signs.
\end{rem}
The convergence of the algorithm we presented in \S\ref{test}-\S\ref{bis}
follows immediately from Proposition~\ref{prop:analytic}.
\begin{cor}
    \label{cor:final}
    After applying the bisection method sufficiently many times we obtain an
    approximation to~$\dim_H X$ which is arbitrarily close to the true value.
\end{cor}
\begin{rem}
    We would like to note that although Proposition~\ref{prop:analytic}
    guarantees the convergence of our algorithm, it doesn't give any explicit estimates
    on how large~$m$ one will need to take in practical realisation. 
Our heuristic observation shows that it takes of order $N = -\frac{\log \varepsilon}{\log 2}$ 
iterations of the bisection method to find an interval of length~$\varepsilon$ containing 
the value of~$\dim_H X$ for a scheme of analytic contractions. 

    For instance, for many of our examples it is sufficient to take~$m=6$ to obtain an
    estimate accurate to~$4$ decimal places which makes the matrices small and the computation very fast.
    However, when we require greater accuracy we need to
    choose~$m$ larger to provide test functions
    for Lemma~\ref{tech} which will be a more close approximation of the
    eigenfunction.  For example, to verify $\dim_H(E_2)$
    to over $200$ decimal places we choose $m=275$. 

\end{rem}


\begin{example}[Contractions with less regularity] When the contractions  have less regularity the present interpolation
    method is not very efficient. For instance Falk and Nussbaum (\cite{FN18}, \S 3.3) considered
    the limit set~$X$ corresponding to the contractions
    $$
      T_1(x) = \frac{x + ax^{7/2}}{3 + 2a} \quad \mbox{ and } \quad
        T_2(x) = \frac{x + ax^{7/2}}{3 + 2a} +  \frac{2 + a}{3 + 2a}
    $$
    for $0 < a < 1$.   For example, when $a=\frac{1}{2}$ they show that
    $$
    \mathbf{0.733474}5730\,00780 \leq \dim_H(X) \leq \mathbf{0.733474}6222\,22678
    $$
    giving an estimate accurate to~$6$ decimal places.
    By increasing the number of Chebyshev nodes to $m=100$ we obtain test
    functions that, using Lemma~\ref{tech} give 
    $$
    \dim_H(X) = 0.7334746151\,5\pm 1.5\cdot10^{-10}.
    $$
    This is a modest improvement in accuracy to $9$ decimal places.
\end{example}

}

\section{Applications}
We will now apply the method in~\S\ref{sec:estimates} to the problems described
in the introduction. 
\subsection{Markov---Lagrange theorems} 
\label{ssec:Markoff}



The bounds on the Hausdorff dimension on various parts of the difference of
the Markov
and Lagrange spectra as stated in Theorem~\ref{MMtheorem} are built on the comprehensive
analysis of Matheus and Moreira~\cite{MM}, where the Hausdorff dimension
bounds are given in terms of the Hausdorff dimension of certain
linear-fractional Markov Iterated Function Schemes. 

Therefore we set ourselves  the task of establishing with an accuracy of
$\varepsilon = 10^{-5}$, say, the Hausdorff dimension of the limit sets
involved. We need to consider five different Iterated Function Schemes corresponding to five cases of
Theorem~\ref{MMtheorem}.

We begin by formulating a general framework which embraces many 
of our numerical results.
In the proof of Theorem~\ref{MMtheorem} the set $X_M$ is given in terms of sequences
from an alphabet $\{1,2,\ldots, d\}$ with certain forbidden
words $\underline w$ (of possibly varying lengths).  For example, in
Parts~2 and~3 the basic forbidden words are of length~$2$ and Markov condition is
easy to write. However, in Parts~1,~4 and~5 the basic forbidden words are of length~$3$ and~$4$. 
If the maximum length of a word is~$N$, then we can 
consider the alphabet of ``letters'' $\underline w = w_1w_2 \cdots w_{N-1}  \in \{1,2, \cdots, 
d\}^{N-1}$, which are sequences of length~$N-1$ and define a $dN\times
dN$ matrix $M^\prime$, where 
$$
M_{jk} = \begin{cases}
 1 & \mbox{ if } w^j_2 \cdots w^j_{N-1} = w^k_1 \cdots w^k_{N-2} \mbox{ and }
 w^j_1 \cdots w^j_{N-1} w^k_{N-1} \mbox{ is allowed; } \\
 0 & \mbox{ otherwise. }
 \end{cases}
$$
where we say that the word is allowed, if it doesn't contain any forbidden
subwords. 
We can then define a matrix $M$ by changing the entry in the row indexed by a
word $i_1i_2 \cdots i_N $ and column indexed by $j_1j_2 \cdots j_N$ to $0$ 
whenever the concatenation $i_1i_2 \cdots i_{N-1}j_{N-1}$ contains any of 
the forbidden words as a substring\footnote{In the special case $N=2$ this corresponds to the 
usual $1$-step Markov condition.}.    
The transformation associated to the word $\underline w$ is defined by the first term, 
i.e., $T_{\underline w} = T_{w_1}$.  

\subsubsection{Proof of Theorem~\ref{MMtheorem}} 

We can divide the proof of Theorem~\ref{MMtheorem} into  its five constituent parts.

\paragraph{{Part 1}:} $(\mathcal M \setminus \mathcal L) \cap (\sqrt5,\sqrt{13})$.
It is proved in \S B.1 of~\cite{MM} that 
\begin{equation}
    \label{MMp1:eq}
\dim_H(\mathcal M \setminus \mathcal L \cap (\sqrt5, \sqrt{13})) \leq 2 \dim_H(X_M) .
\end{equation}
where $X_M \subset [0,1]$ is a set of numbers whose continued fractions
expansions contain only digits $1$ and $2$, except that subsequences $121$ and
$212$ are not allowed.
We can relabel the contractions $T_{ij}$, $i,j \in \{1,2\}$ for the Markov
iterated function scheme
on~$[0,1]$  as 
  $$
    T_{1}(x) =
    T_{2}(x) = \frac{1}{1+x}, \
    T_{3}(x) =
        T_{4}(x) = \frac{1}{2+x} \ \hbox{ and }
    M= \left(\begin{smallmatrix}
        1& 1& 0&0\\
        0 & 0 & 0 &1 \\
        1 & 0 & 0 &0 \\
        0 & 0 & 1 &1 \\
    \end{smallmatrix}\right).
    $$
    The associated transfer operator is acting on the H\"older space
    of functions $C^\alpha(S)$ where $S = \oplus_{j=1}^4 [0,1]\times\{j\}$. 
    We choose  $m = 8$ and apply the bisection method 
    starting with the bounds
    $t^\prime_0 = 0.35$, $t^\prime_1 = 0.38$. It gives the estimate

    \begin{equation}
        \label{PVp1:eq}
     t_0 \colon =  0.3640546 < \dim_H X_M < 0.3640548 
    \end{equation}
    To justify the bounds, the functions $\underline f=(f_1,\ldots,f_4),\, \underline g=(g_1,\ldots,g_4) \in
    C^\alpha(S) $ are explicitly computed using eigenvectors of the matrices 
    $B^{t_0}$ and $B^{t_1}$, respectively. These are polynomials of degree~$7$ 
    given by
    $$
    f_j(x) = \sum_{k=0}^7 a^j_k x^k, \qquad g_j = \sum_{k=0}^7 b^j_k x^k, \quad
    j = 1,\ldots, 4.
    $$
    whose coefficients $a_k^j$, $b_k^j$ are tabled in \S\ref{Appendix I}. 
We can obtain the following inequalities
   \begin{equation}
     \label{testp1:eq}
\sup_S 
\frac{\mathcal L_{t_1} \underline f}{\underline f} 
< 1 - 10^{-9}; \qquad 
\inf_S \frac{\mathcal L_{t_0} \underline g}{\underline g}  > 1+10^{-7},
\end{equation}
and the bound~\eqref{PVp1:eq} follows from Lemma~\ref{tech}. 
Substituting the bounds from~\eqref{PVp1:eq} into the inequality~\eqref{MMp1:eq}
we obtain 
$$
\dim_H( (\mathcal M \setminus \mathcal L) \cap (\sqrt5, \sqrt{13})) < 0.7281096.
$$
\qed

The estimates from~\eqref{PVp1:eq} confirm the conjectured upper bound of
$\dim_H (X_M ) <0.365$ obtained in~\cite{MM}, \S B.1 using the periodic point method.

\paragraph{Part 2:} $(\mathcal M \setminus \mathcal L) \cap (\sqrt{13},3.84)$. 
In~\cite{MM}, \S B.2 it is shown that 
\begin{equation}
    \label{MMp2:eq}
\dim_H\left( (\mathcal M \setminus \mathcal L)\cap (\sqrt{13},3.84) \right) <
0.281266 + \dim_H X_M,
\end{equation}
where $X_M \subset [0,1]$ is a set of numbers whose continued fractions
expansions contain only digits $1,2,3$, except subsequences $13$ and $31$ are
not allowed. This is also the limit set of the IFS 
  $$
    T_1(x) = \frac{1}{1+x}, \
    T_2(x) = \frac{1}{2+x}, \
    T_3(x) = \frac{1}{3+x},
    \quad M= \left(\begin{smallmatrix}
        1 & 1& 0\\
        1 & 1 & 1 \\
        0 & 1 & 1 \\
    \end{smallmatrix}\right).
    $$
    Following the same strategy as above, with $S = \oplus_{j=1}^3 [0,1] \times
    \{j\}$  and $m=8$, and using $t_0^\prime = 0.56$ and
    $t_1^\prime = 0.58$ as initial guesses for bisection method\footnote{Based
    on~\cite{MM}, \S B.2.}  we get the following upper and
    lower bounds
    \begin{equation}
        \label{PVp2:eq}
    t_0 \colon = 0.5739612 \leq \dim_H(X_M) \leq 0.5739617 = \colon t_1.
    \end{equation}
    The leading eigenvectors of the corresponding matrices $B^{t_0}$ and
    $B^{t_1}$ give polynomial functions 
    $$
    f_j(x) = \sum_{k=0}^{7} a^j_k x^k, \qquad g_j = \sum_{k=0}^{7} b^j_k x^k, \quad
    j = 1,2,3.
    $$
    whose coefficients $a_k^j$, $b_k^j$ are tabled in \S\ref{Appendix II}. 
   We can obtain the following inequalities
   \begin{equation}
     \label{testp2:eq}
\sup_S 
\frac{\mathcal L_{t_1} \underline f}{\underline f} 
< 1 - 10^{-7}; \qquad 
\inf_S \frac{\mathcal L_{t_0} \underline g}{\underline g}  > 1+10^{-7},
\end{equation}
and the bound~\eqref{PVp2:eq} follows from Lemma~\ref{tech}. 
Substituting the bounds from~\eqref{PVp2:eq} into the inequality~\eqref{MMp2:eq}
we obtain 
$$
\dim_H( (\mathcal M \setminus \mathcal L) \cap (\sqrt{13},3.84)) < 0.855228.
$$
\qed

The upper bound in estimate~\eqref{PVp2:eq} confirms the conjectural bound $\dim_H(X_M)  < 0.574$ 
(\cite{MM},~\S B.2) which has been obtained using the periodic point
approach. 

\paragraph{Part 3:} $(\mathcal M \setminus \mathcal L) \cap (3.84,3.92)$.
The following inequality was established in~\cite{MM} \S B.3:
\begin{equation}
    \label{MMp3:eq}
\dim_H( (\mathcal M \setminus \mathcal L) \cap (3.84, 3.92)) <  \dim_H(X_M) +
0.25966,
\end{equation}
where $X_M\subset[0,1]$ is the set of all numbers such that its continued
fraction expansions contain only digits $1$, $2$, and $3$ with an extra
condition that subsequences $131$, $132$, $231$, and $313$ are not allowed. This
corresponds to a Markov IFS with $9$ contractions
$$
T_{ij} = \frac{j+x}{1+i(j+x)}, \quad i,j \in \{1,2,3\},
$$
with Markov condition given by the $9\times9$ matrix 
$$
M_{ij, kl} = \begin{cases}
        0& \mbox{ if }
        \{ijk, jkl \}  \cap \{131,132, 231, 313\} \ne \varnothing \cr
        1&  \mbox{ otherwise. }\cr
    \end{cases}
$$
Ordering the $2$-element sequences $ij$ in lexicographical order, we see that some columns of
the resulting $9\times 9$ matrix agree, more precisely, $M(j,1) \equiv M(j,2)$ and
$M(j,4) \equiv M(j,5) \equiv M(j,6)$ for all $1\le j\le 9$. 
Therefore, the matrix $B^t$ defined by~\eqref{eq:Btmatrix} has the same
property. Let $v^t = (\overline{v_1}, \ldots, \overline{v_d})$, where $\overline{v}_j \in
\mathbb R^m$ be its left eigenvector corresponding to the leading eigenvalue
$\lambda$. Then 
$$
\sum_{k=1}^d \overline{v_k} M(k,j) \cdot B_k^t = \lambda \overline{v_j}, \quad 
1 \le j \le 9.
$$
Therefore the equality between matrix columns implies $\overline{v_1} =
\overline{v_2}$  and $\overline{v_4}=\overline{v_5}=\overline{v_6}$. 

The bisection method with $m=8$ gives the following
lower and upper bounds on dimension 
\begin{equation}
    \label{PVp3:eq}
    t_0 \colon = 0.6113922 < \dim_H X_M < 0.6113925 = \colon t_1. 
\end{equation}
The coefficients of the test functions $\underline f=(f_1,\ldots,f_9)$ for $\mathcal
L_{t_1}$ and $\underline g=(g_1,\ldots,g_9)$ for $\mathcal L_{t_0}$, which are polynomials of
degree~$7$ are tabled in~\S\ref{Appendix III}. The equalities between elements of
the eigenvectors imply $f_1 = f_2$, $f_4=f_5=f_6$ and $g_1 = g_2$,
$g_4=g_5=g_6$.  Using ball arithmetic we get
bounds for the ratios
\begin{equation}
  \label{testp3:eq}
\sup_S 
\frac{\mathcal L_{t_1} \underline f}{\underline f} 
< 1 - 10^{-7}; \qquad 
\inf_S \frac{\mathcal L_{t_0} \underline g}{\underline g}  > 1+10^{-7}. 
\end{equation}
and the estimates~\eqref{PVp3:eq} follow from Lemma~\ref{tech}. 
Substituting the upper bound from~\eqref{PVp3:eq} into the inequality~\eqref{MMp3:eq}
we obtain 
$$
\dim_H( (\mathcal M \setminus \mathcal L) \cap (3.92,4.01)) < 
0.8710525.
$$

The estimates~\eqref{PVp3:eq} confirm the heuristic bound of $\dim_H (X_M) <
0.612$ given in~\cite{MM}, \S B.3.

In the remaining two cases, corresponding to intervals $(3.92,4.01)$ and
$(\sqrt{20},\sqrt{21})$ the Markov condition is a little more
complicated and instead of numbering the contractions defining the Iterated
Function Scheme by single numbers
it is more convenient to index them by pairs or triples.   

\paragraph{Part 4:} $\dim_H( (\mathcal M \setminus \mathcal L) \cap (3.92,4.01)
)$. 
In~\cite{MM}, \S B.4 the following inequality is proved:
\begin{equation}
    \label{MMp4:eq}
\dim_H( (\mathcal M \setminus \mathcal L) \cap (3.92,4.01)) <  \dim_H(X_M) + 0.167655.
\end{equation}
where $X_A\subset[0,1]$ is the set of all numbers such that its continued
fraction expansions contain only digits $1$, $2$, and $3$ with an extra
condition that subsequences $131$, $313$, $2312$, and $2132$ are not allowed. This
corresponds to a Markov iterated function scheme
$$
T_{ijk} = \frac{1+j(k+x)}{i+(ij+1)(k+x)}, \quad i,j,k \in \{1,2,3\},
$$
with Markov condition given by the $27\times27$ matrix 
$$
M(ijk, qrs) = \begin{cases}
        0& \mbox{ if }  \{ijk, jkq, kqr, qrs \}  \cap \{131,313 \} \ne
        \varnothing, \mbox{ or }  \cr
        0& \mbox{ if }  \{ijkq, jkqr, kqrs\} \cap \{2312,2132 \} \ne \varnothing  \cr
        1&  \mbox{ otherwise. }\cr
    \end{cases}
$$
Again we see that there are 
equalities between columns of the matrix $M$. In particular, for any triple $ijk$
\begin{align}
M(ijk,111) &= M(ijk,112) = M(ijk,113) \notag  \\
M(ijk,121) &= M(ijk,122) = M(ijk,123) \notag\\
M(ijk,131) &= M(ijk,313) = 0 \label{ap4:eq}\\
M(ijk,321) &= M(ijk,322) = M(ijk,323) \notag \\
M(ijk,331) &= M(ijk,332) = M(ijk,333) \notag \\ 
M(ijk,211) &= M(ijk,2rs) \quad 1 \le r,s \le 3. \notag
\end{align}
As in the previous case, these equalities imply that the corresponding
components of the eigenfunctions are identical. 
The bisection method with $m=8$ and $\varepsilon=6\cdot 10^{-8}$ and 
gives the following
lower and upper bounds on dimension 

\begin{equation}
    \label{PVp4:eq}
    t_0 \colon = 0.6433544 < \dim_H X_A < 0.6433548 = \colon t_1.
\end{equation}

The coefficients of the test functions $\underline f$ for $\mathcal
L_{t_1}$ and $\underline g$ for $\mathcal L_{t_0}$, which are polynomials of
degree~$7$ are tabled in \S\ref{Appendix IV}. Using ball arithmetic we get
bounds for the ratios
\begin{equation}
  \label{testp4:eq}
\sup_S 
\frac{\mathcal L_{t_1} \underline f}{\underline f} 
< 1 - 10^{-7}; \qquad 
\inf_S \frac{\mathcal L_{t_0} \underline g}{\underline g}  > 1+10^{-7},
\end{equation}
and the estimates~\eqref{PVp4:eq} follow from Lemma~\ref{tech}. 
Substituting the upper bound from~\eqref{PVp4:eq} into the inequality~\eqref{MMp4:eq}
we obtain 
$$
\dim_H( (\mathcal M \setminus \mathcal L) \cap (\sqrt{20},\sqrt{21})) <
0.8110098. 
$$
The upper bound from~\eqref{PVp4:eq} makes
rigorous the heuristic bound  $\dim_H(X_M) < 0.65$ in~\cite{MM}, \S B.5 which was based on
the non-validated periodic point method. 

\begin{rem} In~\cite{MM} there are also bounds on $\dim_H((\mathcal M\setminus
    \mathcal L)\cap (4.01, \sqrt{20}))$ which use estimates on $\dim_H(X_{1,2,3})$ due to Hensley. We reconfirm
and improve these in Table~\ref{tab:hensley}.
\end{rem}

\paragraph{Part 5:}
$\dim_H( (\mathcal M \setminus \mathcal L) \cap (\sqrt{20},\sqrt{21})
)$. 
In~\cite{MM}, \S B.5 Matheus and Moreira established the following inequality:
\begin{equation}
    \label{MMp5:eq}
\dim_H( (\mathcal M \setminus \mathcal L) \cap (\sqrt{20}, \sqrt{21}) <
\dim_H(X_M) + 0.172825,
\end{equation}
where $X_M \subset [0,1]$ is a set of numbers whose continued fractions
expansions contain only digits $1$, $2$, $3$, and $4$, except that subsequences
$14$, $24$, $41$ and $42$ are not allowed. This is the limit set of the Markov
Iterated Function Scheme 
    $$
    T_{1}(x) = \frac{1}{1+x}, \
    T_{2}(x) = \frac{1}{2+x}, \
    T_{3}(x) = \frac{1}{3+x}, \
    T_{4}(x) = \frac{1}{4+x}; \
    M= \left(\begin{smallmatrix}
        1& 1& 1&0\\
        1 & 1 & 1 &0 \\
        1 & 1 & 1 &1 \\
        0 & 0 & 1 &1 \\
    \end{smallmatrix}\right).
    $$
    The bisection method with~$m=10$ and~$\varepsilon=6\cdot10^{-8}$ as before and the initial
guess $t_0^\prime = 0.7$, $t_1^\prime=0.71$ gives upper and lower bounds on the
dimension:
\begin{equation}
\label{PVp5:eq}
t_0 \colon = 0.7093943 < \dim_H X_M < 0.7093945 = \colon t_1.
\end{equation}
The coefficients of the test functions $\underline f$ for $\mathcal
L_{t_1}$ and $\underline g$ for $\mathcal L_{t_0}$, which are polynomials of
degree~$9$ are tabled in \S\ref{Appendix V}. Using ball arithmetic we get
bounds for the ratios
\begin{equation}
  \label{testp5:eq}
\sup_S \frac{\mathcal L_{t_1} \underline f}{\underline f} 
< 1 - 10^{-7}; \qquad 
\inf_S \frac{\mathcal L_{t_0} \underline g}{\underline g}  > 1+10^{-8}, 
\end{equation}
and the estimates~\eqref{PVp5:eq} follow from Lemma~\ref{tech}. 
Substituting the upper bound from~\eqref{PVp5:eq} into the
inequality~\eqref{MMp5:eq} we obtain 
$$
\dim_H( (\mathcal M \setminus \mathcal L) \cap (3.84,3.92)) <  
 0.8822195.
$$


The upper bound from~\eqref{PVp5:eq} confirms a heuristic estimate $\dim_H(X_{M}) < 0.715$ (in~\cite{MM}, \S B.6)
obtained using the periodic points method.

\begin{rem}
    We can very easily increase the accuracy
    significantly, at the expense of greater computation, but
    the present estimates seem sufficient for our needs. 
\end{rem}

\begin{rem}
    The initial calculations were carried out on a MacBook Pro with a
    2.8 GHz Quad-Core Intel i7 with 16GB DDR3 2133MHz RAM running MacOS
    Catalina using Mathematica.
    The bounds have been confirmed rigorously using a C program written by the
    second author based on Arb
    library for arbitrary precision ball arithmetic. 
\end{rem}

\begin{rem}
The periodic point method does not allow such accurate estimates on the
computational error. In particular, it is possible to estimate the error in the case of 
Bernoulli iterated function schemes, but the estimate is ineffective for Markov systems.
Similarly, for the McMullen algorithm~\cite{McMullen} the errors are more difficult to estimate.
\end{rem}

\subsubsection{Lower bounds and the proof of Theorem \ref{good}}
\label{ss:e2proof}
A further estimate on the Hausdorff dimension of the difference of the Lagrange
and Markov spectra uses the dimension of the Cantor set of numbers in $(0,1)$
whose continued fraction expansion contains only digits~$1$ and~$2$: 
\begin{equation}
    \label{E2def:eq}
E_2 \colon = \left\{[0;a_1,a_2, a_3, \ldots ] \mid  a_n \in \{1,2\} \mbox{ for
all } n \in \mathbb N \right\}.
\end{equation}
In particular, Matheus and Moreira~\cite{MM} showed that $\mathcal M
\setminus \mathcal L$ contains the image under a Lipschitz
bijection of the set $E_2$ (\cite{MM}, Theorem 5.3) from which they immediately deduce
the following. 

\begin{thm}[Matheus---Moreira]\label{thm:MM}
    $\dim_H(\mathcal M \setminus \mathcal L) \geq \dim_H(E_2)$.
\end{thm}

\begin{rem} 
    Theorem \ref{thm:MM} is a corollary of the slightly   stronger local result that 
    $\dim_H(\mathcal M \setminus \mathcal L \cap   (3.7096, 3.7097)) \geq \dim_H(E_2)$
    cf. (\cite{MM}, Corollary 5.4).
\end{rem}

It is easy to see that the set $E_2$~is the limit set for the contractions $T_1, T_2: [0,1] \to [0,1]$ defined by 
$$
T_1(x) = \frac{1}{1+x} \mbox{ and } T_2(x) = \frac{1}{2+x}.
$$
In~\cite{JP02}  there is a validated value for the
Hausdorff dimension of $\dim_H(E_2)$ to $100$ decimal places using periodic
points method and a careful analysis of the error bounds. In this case the error estimates are easier because the system is Bernoulli. 
In~\cite{McMullen} Theorem~\ref{thm:MM} is combined with the numerical
value from~\cite{JP02} to give a lower bound on the dimension of
the difference of the Markov and Lagrange spectrum. 

We can use the approach in \S\ref{sec:estimates} this note to rigorously (re-)verify this bound.  We begin with the lower bound.
We choose $S = [0,1]$ since the iterated function scheme is Bernoulli.
Let $m=120$ and 
\begin{align}
    \label{e2t0:eq}
    t_0^\prime &= 0.5312805062\,7720514162\,4468647368\,4717854930\,5910901839\\
      &\qquad8779888397\,8039275295\,3564383134\,5918109570\,1811852398  \notag
      \\
    \label{e2t1:eq}
    t_1^\prime & = t_0 + 10^{-100}
\end{align}
(taken from~\cite{JP02}).
We can then use the Chebyshev---Lagrange interpolation to find a test
function $g\colon [0,1] \to \mathbb R$ which is a polynomial of degree $119$.  
We can then check that 
$$
    \inf_S \frac{\mathcal L_{t_0^\prime} g}{g}  >   1 + 10^{-100} 
$$
and then applying Lemma~\ref{tech} we can deduce that $\dim_H(E_2) > t_0$.

We proceed similarly to verify the upper bound $t_1^\prime$.
Namely, the Chebyshev---Lagrange interpolation gives another test
function $f\colon [0,1] \to \mathbb R$ which is also a polynomial of degree $119$
with the property that
$$
    \sup_S \frac{\mathcal L_{t_1^\prime} f}{f } <  1 - 10^{-101} .
$$
Thus by Lemma~\ref{tech} we conclude that $\dim_H(E_2) < t_1^\prime$.
\begin{rem}
This example demonstrates that, despite the fact that at first sight estimating
the ratio of the image of the test function and the function itself could be
potentially very time consuming and challenging, for many systems of particular
interest, the derivative $\left(\frac{\mathcal L_{q} f}{f }
\right)^\prime$ turns out to decrease sufficiently fast as $m\to \infty$ to make
realisation possible in practice.

In ~\cite{JP02} the estimates involved computing
$2^{25} = 33,554,432$ periodic points up to period $25$ and the  exponentially
increasing amount of data needed makes it impractical to improve the rigorous
estimate on~$\dim_H(E_2)$ significantly. 
On the other hand, using the approach via Chebyshev---Lagrange interpolation and
Lemma~\ref{tech} we were able to confirm this result same accuracy using only two 
$120\times120$ matrices, and it would require about $600$ matrices (of
increasing size from $6\times 6$ up to $120\times 120$) in total to
recompute this estimates starting with the initial guess $t_0=0$ and $t_1 = 1$. 
This represents a significant saving in memory usage at expense of computing
$400$ coefficients for the derivative estimates. 
\end{rem}

Moreover, we can now easily improve on the estimate using 
the bisection method combined with interpolation and Lemma~\ref{tech}
where the amount of data required by our analysis \emph{grows linearly} with the
accuracy required. Indeed, letting~$m=270$ and $\varepsilon = 10^{-200}$ we
apply the bisection method choosing $t_0^\prime$ given by~\eqref{e2t0:eq} and
$t_1^\prime$ given by~\eqref{e2t1:eq} as initial
guess. It gives 
\begin{align}
    t_0 &=
    0.5312805062\,7720514162\,4468647368\,4717854930\,5910901839\,8779888397
    \label{e2t00:eq} \\
    &\qquad8039275295\,3564383134\,5918109570\,1811852398\,8042805724\,3075187633
    \notag \\
    &\qquad4223893394\,8082230901\,7869596532\,8712235464\,2997948966\,3784033728
    \notag \\
    &\qquad7630454110\,1508045191\,3969768071\,2.\notag \\
    \intertext{and} 
    t_1 &= t_0 + 2\cdot10^{-201} \label{e2t11:eq}
\end{align}
Then we can use the interpolation method to construct  test
functions~$f$ and $g$ which are polynomials of 
degree~$269$ defined on the  unit interval\footnote{We omit a detailed listing
of all the~$540$
coefficients of~$\underline f$ and~$\underline g$.  However,
they are easily recovered Mathematica.}.  

We can then explicitly compute
$$
    \inf_S \frac{\mathcal L_{t_0} f}{f} > 1 + 10^{-213}, \qquad 
    \sup_S \frac{\mathcal L_{t_1} g}{g} < 1 - 10^{-211}. 
$$
which implies that $t_0 \leq \dim_H(E_2) \leq t_1$.

\subsection{Zaremba Theory}

\label{sec:Zaremba}
In the introduction we described interesting results of
Bourgain---Kontorovich~\cite{BK14}, Huang~\cite{Huang}, and
Kan~\cite{Kan16},~\cite{Kan17},~\cite{Kan19}, which made progress towards
the Zaremba Conjecture. 
These results have a slightly more general formulation, which we will now recall.
For a finite alphabet set $A \subset \mathbb N$
consider the  iterated function scheme
$$
T_n\colon [0,1] \to [0,1], \qquad
T_n(x) =  \frac{1}{x+n} \mbox{  for  } n \in A.
$$
and denote its limit set by $X_A$. 

For any $N\in \mathbb N$ we can in addition consider a set 
\begin{align*}
    D_A(N)&\colon =\cr
    &  \left\{q \in \mathbb N \mid 1 \leq q \leq N, \,  
    \exists p \in \mathbb N, (p,q)=1; 
    a_1, \cdots, a_n \in A \mbox{ with } \frac{p}{q} = [0;a_1, \cdots, a_n]
    \right\}.
\end{align*}
In particular, when $A = \{1,2, \cdots, m\}$ then $D_A$ reduces to $D_m$ as defined in the introduction.

We begin with the density one result~\cite{Huang},~\cite{Huang15}.
\begin{thm}[Bourgain---Kontorovich, Huang]\label{thm:densityoneplus}
    Let $A \subset \mathbb N$ be a finite subset for which
    that associated set $X_A$ satisfies  $\dim_H(X_A) >
    \frac{5}{6} = 0.83\dot3$.  
    Then
    $$
    \lim_{N \to +\infty} \frac{\#D_A(N)}{N} =1.
    $$
\end{thm}
The statement of Theorem~\ref{thm:Huang} corresponds 
to the particular choice of alphabet  $A = \{1,2,3,4,5\}$ in Theorem~\ref{thm:densityoneplus}.
Similarly, Theorem~\ref{thm:kan} has a slightly more
general formulation (from~\cite{Kan16}) as a positive density result.

We begin by recalling  the following useful notation. 
Given two real-valued functions~$f$ and~$g$ we say that $f \gg g$ if there exist
a constant~$c$ such that $f(x) > c g(x)$ for all~$x$ sufficiently large.  

The statement of Theorem~\ref{thm:kan} corresponds 
to the particular choice of alphabet of $A = \{1,2,3,4\}$ in Theorem~\ref{thm:kanplus}.

\begin{thm}[Kan~\cite{Kan16}, Theorem~1.4]\label{thm:kanplus}
    Let~$A \subset \mathbb N$ be  a finite set for which the  associated limit set $X_A$
    has dimension $\dim_H(X_A) > \frac{\sqrt{19}-2}{3} = 0.7862\ldots$ Then
    $$
    \# \left\{q \in \mathbb N \mid 1 \leq q \leq N \colon 
    \exists p \in \mathbb N; a_1, \cdots, a_n \in A \mbox{ with } \frac{p}{q} =
    [0;a_1, \cdots, a_n]
    \right\} \gg N.
    $$
\end{thm}
The  derivation of Theorem~\ref{thm:kanplus} is conditional on the inequality $\dim_H(E_4) >
\frac{\sqrt{19}-2}{3}$ 
which was based on the empirical computations by
Jenkinson~\cite{Jenkinson}, but which were rigorously
justified in~\cite{JP20}.  We will rigorously (re)confirm this inequality in the
next section using the approach in~\S\ref{sec:estimates}. 

In the case that the Hausdorff dimension of the limit set $X_A$ 
is smaller, in particular, 
$ \dim_H(X_A) < \frac{5}{6}$  there are still some interesting  lower bounds on $\#D_A(N)$. 
For convenience we denote (omitting dependence on $A$)
$$
\delta \colon=  \dim_H(X_A)
$$ 
then a classical result of Hensley showed that $\#D_A(N)\gg N^{2\delta}$~\cite{Hensley}.
Subsequently, this was refined in different ranges of~$\delta$ as follows:
\begin{enumerate}
    \item[i)]
        If  $\frac{1}{2} < \delta < \frac{5}{6}$ then  
        $
        \#D_A(N) \gg  N^{\delta + (2\delta -1)(1-\delta)/(5-\delta) - \varepsilon}, 
        $
        for any $\varepsilon > 0$~\cite{BK14}. 
        \item[ii)]
            If   $ \frac{\sqrt{17}-1}{4} 
            <\delta  < \frac{5}{6}$ then   $     \#D_A(N) \geq N^{1-\varepsilon} $
            for all $\varepsilon > 0$ \cite{Kan16}.
        \item[iii)]
            If  $3-\sqrt{5} <  \delta  <  \frac{\sqrt{17}-1}{4} $, 
            then 
            $\#D_A(N)\geq N^{1 + \frac{2\delta^2 + 5 \delta -5}{2\delta -1}
            -\varepsilon}$~\cite{Kan19}.
\end{enumerate}
As a concrete application, Kan considered the finite set $A=\{1,2,3,5\}$ and by applying
the inequality in iii) obtained the following lower bound.
\begin{thm}[Kan~\cite{Kan19}, Theorem~1.5, Remark~1.3]\label{kan-lower}
    For the alphabet $A=\{1,2,3,5\}$  one has     $\#D_A(N) \gg N^{0.85}$.
\end{thm}
This gave  an improvement on the bound of $\#D_A(N) \gg N^{0.80}$ arising from i).
However, this  required  that $\dim_H(X_A) > 3-\sqrt{5}$, an estimate conjectured by Jenkinson in~\cite{Jenkinson}, but which was not validated. 
We will present a rigorous bound in the next subsection.

To further illustrate this theme we will use the present method to 
confirm the following local version of the Zaremba conjecture proposed by Huang.

\begin{thm}[after Huang]
    Let $A = \{1,2,3,4,5\}$ and consider $D_A = \cup_{n \in \mathbb N} D_A(n)$, in other words 
    $$
    D_A \colon =
   \left\{q \in \mathbb N \mid \exists p \in \mathbb N, (p,q)=1 \mbox{ and } 
    a_1, \cdots, a_n \in A \mbox{ with } \frac{p}{q} = [0;a_1, \cdots, a_n]
    \right\}.
    $$
    Then for every $m>1$ we have that $D_A = \mathbb N (\kern-6pt \mod m)$.
    In other words, for every $m>1$ and every $q \in \mathbb N$ we have
     $q \kern2pt(\kern-6pt\mod m) \in D_A$.  
\end{thm}

\subsubsection{Dimension estimates for $E_5$.} 
\label{ss:dimdensity1}

A crucial ingredient in the analysis in the proof of density one
Theorem~\ref{thm:Huang} used in~\cite{Huang} is 
that the limit set $E_5$ for the iterated function scheme 
$$
T_j: [0,1] \to [0,1], \quad  T_j(x) = \frac{1}{j+x}, \quad 1 \leq j \leq 5
$$
satisfies $\dim_H(E_5) > \frac{5}{6} = 0.83\dot3$.  
In~\cite{JP20}, this was confirmed with rigorous bounds
$$
\dim_H(E_5) = 0.836829445\pm5\cdot 10^{-9} 
$$
using the periodic points method.  For this particular
example, the error estimates are more tractable because
the iterated function scheme is Bernoulli, rather than
just Markov. 

However, we can use the method in this note to reconfirm
this bound, and improve it, with very little effort, to
the following: 
\begin{thm}
  \label{thm:dimE5}
    $$  
    \dim_H(E_5) = 0.836829443680\pm 10^{-12}.
    $$
\end{thm}
\begin{proof}
    We can choose $S = [0,1]$ since the iterated function scheme is Bernoulli.
    Applying the bisection method with~$m=15$ and $\varepsilon = 10^{-11}$, we get
    lower and upper bounds 
    \begin{align*}
    t_0 &= 0.83682944368\,02, \mbox{ and } \\  
    t_1 &= t_0 + 2\cdot10^{-12} =    0.83682944368\,20.
    \end{align*}
    The Chebyshev---Lagrange interpolation method then gives two polynomials of
    degree~$14$ that can serve as test functions. Their coefficients are listed
    in \S\ref{Appendix VI-1}. 
    
    We can then explicitly compute
    \begin{equation}
      \label{dimE5:eq}
    \sup_{S} \frac{\mathcal L_{t_1}  f}{ f}  <   1-  10^{-13};  \qquad
    \inf_{S} \frac{\mathcal L_{t_0}  g}{ g}  >   1 +  10^{-13} 
  \end{equation}
    and the result follows from Lemma~\ref{tech}.

\end{proof}

\begin{rem}
    In terms of the practical application to Theorem~\ref{thm:Huang},
    there is no need to have accurate estimates of~$\dim_H(E_5)$, it is
    sufficient to show $\dim_H(X) > \frac{5}{6}$. This can be achieved 
    using a simple calculation ``by hand''. 
    We can take $$f(x) = \frac{2}{3} - \frac{11}{20} x + \frac{1}{3} x^2 - \frac{1}{10}x^3.$$
    We can then compute that~$\frac{\mathcal L_{5/6} f(x)}{f(x)} > 1.0029$.
       It then follows from Lemma~\ref{tech} that 
    $\dim_H(E_5) > \frac{5}{6}$.

\end{rem}

\subsubsection{Dimension estimates for $E_4$.} 
\label{ss:dime4}
A crucial ingredient in the analysis in the proof of 
Theorem~\ref{thm:kan} used in~\cite{Kan17} is 
that the limit set $E_4$ for the iterated function scheme 
$$
T_j: [0,1] \to [0,1], \quad  T_j(x) = \frac{1}{j+x}, \quad 1 \leq j \leq 4
$$
satisfies $\dim_H(E_4) > \frac{\sqrt{19}-2}{3}\approx0.7862\ldots$  

We can validate this result by showing the following bounds on the dimension 

\begin{thm}
  \label{thm:dimE4}
  $ \dim_H(E_4) = 0.7889455574\,83\pm10^{-12}$
\end{thm}
\begin{proof}
    We can choose $S = [0,1]$ since the iterated function scheme is Bernoulli.
    Applying the bisection method with~$m=15$ and $\varepsilon = 10^{-11}$, we
    obtain the lower and upper bounds 
    \begin{align*}
    t_0 &= 0.7889455574\,81  \mbox{ and } \\  
    t_1 &= t_0 + 2\cdot10^{-12} = 0.7889455574\,84. 
    \end{align*}
    The Chebyshev---Lagrange interpolation method then gives two polynomials of
    degree~$14$ that can serve as test functions. Their coefficients are listed
    in \S\ref{Appendix VI-2}. 
    
    We can then explicitly compute
    \begin{equation}
      \label{dimE4:eq}
    \sup_{S} \frac{\mathcal L_{t_1}  f}{ f}  <   1-  10^{-13};  \qquad
    \inf_{S} \frac{\mathcal L_{t_0}  g}{ g}  >   1 +  10^{-12} 
  \end{equation}
    and the result follows from Lemma~\ref{tech}.

\end{proof}
\begin{rem}
    In terms of the practical application to Theorem~\ref{thm:kan},
    there is no need to have accurate estimates of~$\dim_H(E_4)$, it is
    sufficient to show $\dim_H(X) > \frac{\sqrt{19}-2}{3}$. This can be achieved 
    using a simple calculation ``by hand''. 
    We can take 
    $$
    f(x) = \frac{27}{50} - \frac{11}{25} x + \frac{33}{100} x^2 - \frac{11}{50}
    x^3 +  \frac{21}{200} x^4 - \frac{1}{40}  x^5.
    $$ 
    We can then compute that $[\mathcal L_{\frac{\sqrt{19}-2}{3}} f](x)/f(x) > 1.00205$.
       It then follows from Lemma \ref{tech} that 
    $\dim_H(E_4) > \frac{\sqrt{19}-2}{3}$.

\end{rem}

\subsubsection{Dimension estimates for alphabet $A = \{1,2,3,5\}$.}
\begin{thm} 
  \label{thm:dim1235}
  Let $A = \{1,2,3,5\}$.  Then   
    $ \dim_H(X_A) = 0.7709149399\,375\pm1.5\cdot 10^{-12} $. 
\end{thm}
\begin{proof}
    The estimate can be recovered following the same approach as in the proof of
    theorem~\ref{thm:dimE4} and  coefficients 
    of the corresponding test functions listed in \S\ref{Appendix VI-3}. 
   We choose $S = [0,1]$.
    Applying the bisection method with~$m=16$ and $\varepsilon = 3\cdot10^{-12}$, we
    obtain the lower and upper bounds 
    \begin{align*}
    t_0 &= 0.7709149399\,36  \mbox{ and } \\  
    t_1 &= t_0 + 3\cdot10^{-12} =0.7709149399\,39. 
    \end{align*}
    The Chebyshev---Lagrange interpolation method then gives two polynomials of
    degree~$15$ that can serve as test functions. Their coefficients are listed
    in \S\ref{Appendix VI-3}. 
    
    We can then explicitly compute
    $$
    \sup_{S} \frac{\mathcal L_{t_1}  f}{ f}  <   1-  10^{-12},  \qquad
    \inf_{S} \frac{\mathcal L_{t_0}  g}{ g}  >   1 +  10^{-12}; 
    $$ 
    and the result follows from Lemma~\ref{tech}.

\end{proof}

\begin{rem}
    In terms of the practical application to Theorem~\ref{kan-lower},
    there is no need to have accurate estimates of~$\dim_H(E_{\{1,2,3,5\}})$, it is
    sufficient to show $\dim_H(X) > 3 - \sqrt{5}$. This can be achieved 
    using a simple calculation ``by hand''. 
    We can take $$f(x) =   \frac{9}{10}- \frac{2}{5}x.$$
    We can then compute that $[\mathcal L_{3 - \sqrt{5}} f](x)/f(x) > 1.00042$.
       It then follows from Lemma \ref{tech} that 
    $\dim_H(X) > 3 - \sqrt{5}$.
\end{rem}

\begin{rem} The periodic point method and McMullen's approach~\cite{McMullen} cannot give such accurate
estimates because of the prohibitive computer resources required.
In a recent paper~\cite{FN20} Falk and Nussbaum computed Hausdorff dimension of
the sets $E_5$, $E_4$, $E_{1235}$ and some of the Hensley examples we give
below. Their method is also rooted in the interpolation, but uses different
machinery.  
\end{rem}

\subsection[Hensley conjecture]{Counter-example to a conjecture of Hensley}

In~\cite{BK14} Bourgain and Kontorovich gave a counter-example to a conjecture of Hensley
(\cite{Hensley}, Conjecture 3, p.16).  The conjecture stated that for any finite alphabet $A \subset \mathbb N$
for which the Hausdorff dimension of the  associated  limit set (corresponding
to the iterated function scheme with contractions $T_j(z) = \frac{1}{j+x}$ for
$j \in A$) 
satisfies $\dim_H(X_A) > \frac{1}{2}$ the analogue of the Zaremba conjecture
holds true for~$q$ sufficiently large, i.e, 
there exists $q_0 > 0$ such that for any natural number $q \geq q_0$ there
exists $p<q$ and $a_1, \cdots, a_n \in A$ such that $\frac{p}q = [a_1, \cdots, a_n]$.

The construction of the counter-example  hinged on the observation 
that the denominators $q$ corresponding to such restricted continued fraction expansions cannot ever
satisfy $q = 3 (\kern-6pt \mod 4)$  and on showing that for the iterated function scheme 
$\{T_j=\frac{1}{x+j} \mid a = 2,4,6,8,10\}$ the limit set~$X$ has dimension
$\dim_H(X) > \frac{1}{2}$ (cf.~\cite[p.~139]{BK14}).  This was rigorously
confirmed in~\cite[Theorem~7]{JP20} by a fairly elementary argument,
where it was also suggested by a non-rigorous computation that 
\begin{equation}
  \label{bkexample:eq}
\dim_H{X}= 0.5173570309\,3701730466\,6628474836\,4397337 \ldots
\end{equation}
To rigorously justify this estimate, we apply the bisection method with $S = [0,1]$, $m =
40$, $\varepsilon=10^{-36}$ and the initial guess 
\begin{align*}
t_0 &= 0.5173570309\,3701730466\,6628474836\,4397337 \\
t_1 &= t_0 + 10^{-37}.
\end{align*}
The Chebyshev---Lagrange interpolation gives two polynomials $f$ and $g$ of
degree~$39$ which satisfy 
$$
\sup_S \frac{\mathcal L_{t_1} f}{f} < 1-10^{-38} \mbox{ and } 
\inf_S \frac{\mathcal L_{t_0} g}{g} > 1+10^{-37}.
$$
The equality~\eqref{bkexample:eq} follows from Lemma~\ref{tech}.

\begin{rem}[An elementary bound]
    As in the previous examples is not necessary to have a very precise
    knowledge of the value of $\dim_H(X_A)$ in order to establish that this is a
    counter example to the Hensley conjecture. 
    It would be sufficient to know that $\dim_H(X_A) > \frac{1}{2}$.   This can again be achieved 
    using a simple calculation. 
    We can consider instead the linear function 
    $ f(x) = 8 - 2x $.
    It is easy to compute its image under the transfer operator 
    $\mathcal L_{0.5}f$: 
    $$
    \begin{aligned}
      \mathcal L_{0.5} {f}(x) &= 
        \left(\frac{8}{2+x} - \frac{2}{(2+x)^2}\right)
        + \left(\frac{8}{4+x} - \frac{2}{(4+x)^2}\right)
        + \left(\frac{8}{6+x} - \frac{2}{(6+x)^2}\right)\cr
        \qquad &+ \left(\frac{8}{8+x} - \frac{2}{(8+x)^2}\right)
        + \left(\frac{8}{10+x} - \frac{2}{(10+x)^2}\right).
    \end{aligned}
    $$
    Clearly, this is a monotone decreasing function. 

    By taking derivatives or otherwise we can justify that 
    $$
    \inf_{S}\frac{\mathcal L_{0.5} {f} }{f} = \frac16\mathcal L_{0.5}f(1) > 1.
    $$ 
    The result now follows from Lemma~\ref{tech}. 
\end{rem}

Although the original Hensley conjecture is false, Moshchevitin and Shkredov recently proved an
interesting modular version. 
\begin{thm}[\cite{Shkredov}]
  \label{thm:shkredov}
    Let $A\subset \mathbb N$ be a finite set for which $\dim_H(X_A) >\frac{1}{2}$.
    Then for any prime~$p$ there exist~$q = 0 (\kern-6pt\mod p)$, $r$ (coprime to $q$) 
    and $a_1, \cdots, a_n \in A$  such that 
    $$
    \frac{r}{q} = [a_1, \cdots, a_n].
    $$
\end{thm}
Thus there is some interest in knowing which examples 
of $A \subset \mathbb N$ satisfy $\dim_H(X_A) >\frac{1}{2}$ so that this result applies.
For example, one can easily check that for 
$A_1=\{1,4,9\}$
we have $\dim_H(X_{A_1}) = 0.5007902321\,42100396 \pm 10^{-18}$ or 
for $A_2=\{2,3,6,9\}$ we have $\dim_H(X_{A_2}) = 0.5003228005\,96840463 \pm
10^{-18}$.

\subsection{Primes as denominators}

There is an interesting variation on Theorem~\ref{thm:densityoneplus}
where we consider only the denominators which are prime numbers.   

\begin{thm}[Bourgain---Kontorovich, Huang]\label{thm:primes}
    There are infinitely many prime numbers~$q$ which have a primitive
    root\footnote{In other words, there exists $n$ such that $a^n=1 \mod q$.} $a
    \kern-4pt\mod q$ such that the partial quotients
    of $\frac{a}q$ are bounded by $7$. 
\end{thm}

This was originally proved by Bourgain and Kontorovich with the weaker conclusion that
the partial quotients of $\frac{a}{q}$ are bounded by $51$.  The improvement of Huang was conditional on the 
Hausdorff dimension of  limit set~$E_6$ for the iterated function scheme
$\left\{T_j(x) = \frac{1}{x+j} \mid  1 \le j \le 6\right\}$ satisfying 
$\dim_H{E_6} > \frac{19}{22}$.  
In~\cite{JP20} it was rigorously shown using the periodic point method that
$$
 \dim_H{E_6} =   0.86761915\pm10^{-8}.
$$
Furthermore, there  was a heuristic estimate of~$\dim_H(E_6)=
0.8676191732401\ldots$ 
We can apply Chebyshev---Lagrange interpolation with $S=[0,1]$, $m=20$ to
confirm this estimate.
\begin{thm}
    $$ 
    \dim_H{E_6} =  0.8676191732\,4015\pm 10^{-13}.
    $$
\end{thm}
\begin{proof}
    By Chebyshev---Lagrange interpolation applied to the operators $\mathcal
    L_{t_0}$ and $\mathcal L_{t_1}$ we obtain two polynomials $f$ and $g$ of
    degree $19$ which satisfy 
    $$
    \sup_{S} \frac{\mathcal L_{t_1} f}{f} < 1 - 10^{-13} \mbox{ and }   
    \inf_{S} \frac{\mathcal L_{t_0} g}{g} > 1 + 10^{-14}.  
    $$ 
    The statement now follows from Lemma~\ref{tech}.
\end{proof}

\begin{rem}[An elementary bound]
    As in the previous two examples, it is not necessary to have a very precise knowledge of the value of
    $\dim_H(E_6)$ in order to establish the conditions necessary for Theorem~\ref{thm:primes}. 
    It would be sufficient to know that $\dim_H(E_6) > \frac{19}{22}$. This can again be achieved 
    using a simplified choice of~$f$, although it might be a slight exaggeration to say that this is entirely elementary.
    We can consider the degree~$3$ polynomial $f\colon [0,1] \to \mathbb R^+$ defined by 
    $  f(x) = 0.67 - 0.57 x + 0.35 x^2 - 0.107 x^3 $.
    Letting $t = \frac{19}{22}$ we  can consider  the image under the transfer
    operator~$\mathcal L_t$. In particular, one can readily check that 
    $$
    \inf_{S} \frac{\mathcal L_{t} f}{f}  > 1 + 10^{-4}.
    $$
    It follows from Lemma~\ref{tech} that $P(t) >0$ and thus we
    conclude that $\dim_H(E_6) > t =  \frac{19}{22}$.
\end{rem}

In the remainder of this section we will consider applications where the
alphabets, and thus the number of contractions in the iterated function scheme,
are infinite. 

\subsection[Modular results]{Modular results and countable iterated function schemes}

Given $N \geq 2$ and $0 < r \leq N$, we want to consider  a set 
$$
X_{r  (N)} = \left\{[0;a_1,a_2, a_3, \cdots ] \mid  a_n \equiv r \kern2pt(\kern-8pt\mod N) \right\}
$$
consisting of those numbers whose continued fraction expansion only has digits equal to $r$ (mod $N$).
This can be interpreted as a limit set for the countable family of contractions
$T_i(x) = \frac{1}{x+r+Nk}$ ($k \geq 0$).  However, unlike the case of a
finite iterated function scheme the limit set 
$X_{r  (N)}$ is not a compact set.

Similarly to the case of a finite alphabet,  a key ingredient  in determining 
the Hausdorff dimension $\dim_H(X_{r (N)})$  is to consider a one-parameter
family of transfer operators $\mathcal L_t: C^1([0,1] \to C^1([0,1])$ given by
\begin{equation}
    \label{ctbletransfer}
    (\mathcal L_t w)(x) = \sum_{k=0}^\infty (x+r + Nk)^{-2t} w \left( (x+r +
    Nk)^{-1}\right).
\end{equation}
It is well defined for $\Re(t) > \frac{1}{2}$.  A common approach to the
analysis of these operators is to truncate the series to a finite sum of first~$
K$, which contributes an error of $O\left(K^{-2t} \right)$ to the estimates of
leading eigenvalue. Then this would require $K$ to be chosen quite large 
 for even moderate error bounds.

A more successful alternative approach in the present context is to employ the
classical Hurwitz zeta function from analytic number theory.
\begin{definition} 
    \label{def:hurwicz}
The Hurwitz zeta function is a complex analytic function on a half-plane $\Re x
>0$, $\Re s >1$ defined by the series
 $$
\zeta(x, s) = \sum_{k=0}^\infty (x+k)^{-s}. 
 $$
\end{definition}
\noindent It can be extended to a meromorphic function on~$\mathbb C$ for $s \ne 1$.  
The famous Riemann zeta function is a particular case $\zeta(1,s)$. 

We would like to consider monomials $w_n(x):=x^n$, $n\ge0$ and to rewrite
$\mathcal L_t w_n$ using the Hurwitz zeta function as follows
 \begin{align}
     \label{eq:hurwicz}
 \mathcal L_tw_n(x) & = \sum_{k=0}^\infty (x + kN + r)^ {-2t} \cdot (x + kN +
 r)^{-n} \cr
  &=  {N} ^{-2t-n} \sum_{k=0}^\infty \Bigl(\frac{x+r}{N} + k\Bigr)^{-2t-n}\cr
& =  N^{-n-2t} \zeta\left(\frac{x+r}{N}, n + 2t\right).
\end{align}
 From a computational viewpoint, the advantage we gain from expressing the transfer operator in terms of the
 Hurwitz zeta function stems from fact that there are very efficient algorithms for evaluation
 of~$\zeta(x,s)$ to arbitrary numerical precision (cf.~\cite{Johansson} and references
 therein). In particular, the Hurwitz zeta function is implemented both in
 Mathematica and within the Arb library.
 

We can now return to our usual strategy to estimate $\dim(X_{r (N)})$.  We begin
with the following simple result. 

\begin{lemma} \label{lem:modular}
For $t > \frac{1}{2}$ the operator $\mathcal L_t$ has a simple maximal eigenvalue 
$e^{P(t)}$.  The function~$P$ is  real analytic and strictly decreasing 
on the interval $(\frac{1}{2}, +\infty)$.  Moreover, 
there is a unique $t_0 \in (\frac{1}{2}, +\infty)$ such that $P(t_0) = 0$ and $t_0=\dim(X_{r (N)})$.
\end{lemma}
\begin{proof}
 The analyticity comes from analytic perturbation theory and the simplicity of the maximal eigenvalue
 (see \cite{MU03}, Theorem 6.2.12).  The strict monotonicity of the function~$P$ 
 comes from the perturbation identity $P'(t) =  - 2  \int (\log x) h(x) d\mu(x)/\int h d\mu < 0$ where $\mathcal L_t h= e^{P(t)} h$
 and $\mathcal L_t^*\mu= e^{P(t)} \mu$.  
 The interpretation of $\dim(X_{r (N)})$ in terms of the pressure was shown in~\cite{MU03}, Theorem 4.2.13.
\end{proof}

Proceeding as in \S\ref{test} we can fix $m$, compute the zeros of the $m$'th
Chebyshev polynomials $\{y_k\}_{k=1}^m\in[0,1]$ and define Lagrange interpolation 
polynomials~$lp_k(x)$ for $0 \le k \le m$ as before using~\eqref{lagcheb:eq}. 
These polynomials can also be written in a standard form 
$$
lp_k(x) = \sum_{n=0}^{m-1} a_n^{(k)} x^n = \sum_{n=0}^{m-1} a_n^{(k)} w_n(x), \qquad
\mbox{ for } 0 \leq x \leq 1.
$$
Since the operator $\mathcal L_t$ is linear,  we can write
using equation (\ref{eq:hurwicz})
\begin{equation}
    (\mathcal L_t lp_k)(x) =  \sum_{n=0}^{m-1} a_n^{(k)}  (\mathcal L_t w_n) (x)
    =   \sum_{n=0}^{m-1} a_n^{(k)} N^{-n-2t} \zeta \left(\frac{x+r}{N}, n + 2t\right).
\end{equation}
Since the Hurwitz zeta function can be evaluated to arbitrary precision, we can now
compute the $m\times m$ matrix $(\mathcal L_t lp_k)(y_j)_{k,j=1}^m$, the
associated left eigenvector $v = (v_1, \cdots, v_m)$ and construct
a test function $f = \sum_{j=1}^m v_j lp_j$ to use in Lemma~\ref{tech}. 
\begin{example}
To illustrate the efficiency of the approach, we apply this general construction in
a number of cases, which we borrow from a recent work by Chousionis at
al.~\cite{CLU20}, where  
the estimates are obtained using  a combination of a truncation technique with
the method of Falk and Nussbaum. 
We apply the bisection method with $m=12$ and summarise our results in the
Table~\ref{table:clu}
below. For comparison, we include estimates from~\cite{CLU20}.
\begin{table}[h!]
\begin{center}
\begin{tabular}{ |c|l|c|c| } 
 \hline
 \multirow{2}{*}{$r(N)$} & \multirow{2}{*}{ New estimate $\dim_H(X_{r(N)})$ } &
  \multicolumn{2}{|c|}{Old bounds} \\
  \hhline{~~--}
  &  & $s_0$ & $s_1$ \\ 
 \hline
 2 (2)  &$ 0.7194980248\,366    \pm		3\cdot 10^{-13} 	 $	&	$ 0.719360$ & $ 0.719500$ \\ 
 1 (2)  &$ 0.8211764906\,5      \pm	    3.5\cdot 10^{-10} 	 $	&	$ 0.821160$ & $ 0.821177$ \\ 
 3 (3)  &$ 0.6407253143\,83684  \pm		2\cdot 10^{-15} 	 $	&	$ 0.639560$ & $ 0.640730$  \\ 
 3 (2)  &$ 0.6654623380\,4075   \pm 	2.5\cdot 10^{-13} 	 $	&	$ 0.664900$ & $ 0.665460$   \\ 
 3 (1)  &$ 0.7435862804\,5      \pm 	2.5\cdot 10^{-10} 	 $	&	$ 0.743520$ & $ 0.743586$  \\ 
 1 (8)  &$ 0.6194381921\,5      \pm		1.5\cdot 10^{-10} 	 $	&	N/A & N/A \\ 
 \hline
\end{tabular}
\end{center}
\caption{The estimates for $\dim_H(X_{r(N)})$ based on the bisection
method with Hurwitz function employed and the bounds on 
$\dim_H(X_{r(N)}) \in [s_0, s_1]$ from \cite{CLU20}.}
\label{table:clu}
\end{table}
\end{example}

\subsection{Lower bounds for deleted digits}

We can also use the method in the previous section to address the following
natural problem: {\it Given $N \geq 1$ give a  uniform upper bound on the dimension $\dim_H(X_{\mathcal A})$ where 
 $\mathcal A$ ranges over all families of  symbols  satisfying $\mathcal A \cap
 \{1, \cdots, N\} \neq \emptyset$.}
 
   By the natural monotonicity (by inclusion) of $\mathcal A \mapsto \dim_H(X_{\mathcal A})$ we see that such an upper bound will be given by $\dim_H(X_{{\mathcal A}_{N+1}})$ where we let 
   $${\mathcal A}_{N+1} = \{N+1, N+2, N+3, \cdots\}.$$
   As in the preceding subsection we can write the transfer operator 
   associated to the infinite alphabet ${\mathcal A}_{N+1}$
   acting on $w_n$ in terms of the Hurwitz zeta function:
    \begin{align}
     \label{eq:hurwicz1}
 \mathcal L_tw_n(x) & = \sum_{k=N+1}^\infty (x + k)^ {-2t} \cdot (x +k)^{-n} \cr
  &= \sum_{k=0}^\infty (x+N+1 + k)^{-2t-n}\cr
& =   \zeta\left(x+N+1, n + 2t\right).
\end{align}
   
   We have the natural analogue of Lemma \ref{lem:modular}.
   
   \begin{lemma} 
For $t > \frac{1}{2}$ the operator $\mathcal L_t$ has a simple maximal eigenvalue 
$e^{P(t)}$.  The function~$P$ is  real analytic and strictly decreasing 
on the interval $(\frac{1}{2}, +\infty)$.  Moreover, 
there is a unique $t_0 \in (\frac{1}{2}, +\infty)$ such that $P(t_0) = 0$ and $t_0=\dim(X_{\mathcal A_{N+1}})$.
\end{lemma}

Again proceeding as in \S\ref{test} we can fix $m$, compute the zeros of the $m$'th
Chebyshev polynomials $\{y_k\}_{k=1}^m\in[0,1]$ and define Lagrange interpolation 
polynomials~$lp_k(x)$ for $0 \le k \le m$ as before using~\eqref{lagcheb:eq}. 
By analogy with~\eqref{ctbletransfer} we can write using equation (\ref{eq:hurwicz1})
\begin{equation}
    (\mathcal L_t lp_k)(x) =  \sum_{n=0}^{m-1} a_n^{(k)}  (\mathcal L_t w_n) (x)
    =   \sum_{n=0}^{m-1} a_n^{(k)} \zeta\left(x+N+1, n + 2t\right).
\end{equation}
Since the Hurwitz zeta function can be evaluated to arbitrary precision, we can now
compute the $m\times m$ matrix $(\mathcal L_t lp_k)(y_j)_{k,j=1}^m$, the
associated left eigenvector $v = (v_1, \cdots, v_m)$ and construct
a test function $f = \sum\limits_{j=1}^m v_j lp_j$ to use in Lemma~\ref{tech}. 
The results are presented in Table~\ref{table:bounds}.
\begin{table}[h!]
\begin{center}
\begin{tabular}{ |c|c| } 
 \hline
 $N$ & Estimate on $\dim_H(X_{{\mathcal A}_{N+1}})$  \\ 
 \hline
1 &$ 0.840884586 \pm 10^{-8}$  \\ 
2 &$  0.785953471  \pm 10^{-8}$  \\ 
3  &$   0.757889122 \pm 10^{-8}$   \\ 
4  &$   0.757889122 \pm 10^{-8}$   \\ 
5  &$  0.728307126   \pm 10^{-8}$ \\ 
 \hline
\end{tabular}
\end{center}
\caption{Estimates on  $\dim_H(X_{{\mathcal A}_{N+1}})$ which give upper bounds on 
$\dim_H(X_{{\mathcal A}})$ for those alphabets $\mathcal A$ 
with  $\mathcal A\cap \{1, \cdots, N\}= \emptyset$.}
\label{table:bounds}
\end{table}

\subsection{Local obstructions}

In his thesis,  Huang makes an interesting conjecture on local obstructions for the Zaremba conjecture
(see~\cite{Huang}, p. 18).
More precisely, to every $m \in \mathbb N$ we can associate the ``modulo~$m$
map''  which we denote by 
$$
\pi_m \colon \mathbb Z_+ \to \mathbb Z_m = \{0,1,2, \cdots, m-1\}
\hbox{ given by }
\pi_m(q) = q\kern2pt (\kern-8pt \mod m).
$$ 
We say that a finite set $A \subset \mathbb
N$ has no local obstructions if for all $m \geq 1$, 
$$ 
\pi_m\left(  \left\{ q\in \mathbb N \Bigl| \exists p \in \mathbb N, (p,q)=1; a_1,
\cdots, a_n \in\{1,2,3,4,5\} \colon \frac{p}q = [0;a_1, \cdots, a_n]
\right\}\right) =  \mathbb Z_m
$$
i.e., for each $m$ the map $\pi_m$ is surjective.

\begin{conjecture}[Huang,~\cite{Huang}]\label{conj:huang}
If the limit set~$X_A$ of the alphabet $A\subset \mathbb N$ 
has dimension $\dim_H(X_A)> \frac56$ then there are no local obstructions.  
\end{conjecture}

Huang also reduced this to a statement about dimensions of specific limit sets.


\begin{prop}[\cite{Huang}, Theorem 1.3.11]\label{thm:huang}
In notation introduced above, a {\rm finite} alphabet~$A$ has no local obstructions 
provided  $\dim_H X_A >
     \max(\dim_H(X_{2 (2)}),\dim_H(X_{1 (8)}))$.
\end{prop}

In particular, in light of the use of Proposition~\ref{thm:huang}
to  establish Conjecture \ref{conj:huang} it is sufficient to know  that 
$\dim_H(X_{2 (2)})$ and $\dim_H(X_{1 (8)})$ both 
 have
dimension at most $\frac56$. 
Fortunately, we have rigorously established these inequalities as
Table~\ref{table:clu} shows. 

\begin{prop}
 Conjecture \ref{conj:huang} above is correct.
\end{prop}

\subsubsection{Truncation method}
The use of the Hurwitz zeta function works well for the modular examples considered above.  
For more general countable alphabets $ A \subset \mathbb N$ we may have to resort 
to a cruder approximation argument. To illustrate this we can consider the
restriction to a finite alphabet $A_N =
    A \cap \{1,2, \cdots, N\}$.
    If we let $\mathcal L_{A,t}f(x) = \sum_{n \in A} f((x+n)^{-1})(x+n)^{-2t}$
    and $\mathcal L_{A_N,t}f(x) = \sum_{n \in A_N} f((x+n)^{-1})(x+n)^{-2t}$.
    Given $t$ we can pick an $m$ and apply the Chebyshev---Lagrange interpolation with $m$ nodes
    to find a polynomial test function
    $$
    f_{N,m}(x) = \sum_{j=0}^{m-1} a_j x^j
    $$ 
    approximating the eigenfunction for $\mathcal L_{A_N,t}$.  But if we
    subsequently want to use this function in Lemma~\ref{tech} 
    to study~$\mathcal L_{A,t}$ we also need to obtain an upper bound for the remainder
    $$
    E (x) = \sum_{n \in A \setminus A_N}  f_{N,m} \left(\frac{1}{x+n}\right) \frac{1}{(x+n)^{2t}} 
    =  \sum_{j=0}^m a_j \sum_{n \in A \setminus A_N} \frac{1}{(x+n)^{j+2t}}.
    $$
    We can also bound 
    $$
    \sum_{n \in A \setminus A_N} (x+n)^{-j-2t}
    \leq \sum_{N+1}^\infty (x+n)^{-j-2t}
    \leq \int_{N}^\infty x^{-j-2t} dx
    = \frac{1}{(j+2t-1) \cdot N^{j+2t-1} }
    $$
    and then 
    $$
    |E(x)| \leq  
    \sum_{j=0}^m |a_j| \frac{1}{(j+2t-1)\cdot N^{j+2t-1}}.
    $$
    This bound is only polynomial in~$N$ giving this approach limited
    use in precise estimates, in comparison with limits sets generated by a
    finite number of contractions. 

However, this elementary approach leads to a simple 
direct proof of Conjecture \ref{conj:huang} 
(without resorting to introducing  the Hurwitz zeta function).

\begin{example}[An elementary bound for $\dim_H X_{0 (2)}$ revisited]
We  will use a  more  elementary proof based on the interpolation method. 
We can consider the transfer operator
$$
(\mathcal L_{t}f)(x) = \sum_{n=1}^\infty f\left( \frac{1}{x+2n}\right) \frac{1}{(x+2n)^{2t}}.
$$
We can separate the first term of the infinite series and estimate the remaining
sum by an integral. More precisely, we may write 
\begin{equation*}
(\mathcal L_{t}f)(x) = f\left( \frac{1}{x+2}\right) \frac{1}{(x+2)^{2t}} +
\sum_{n=2}^\infty f\left( \frac{1}{x+2n}\right) \frac{1}{(x+2n)^{2t}}.
\end{equation*}
We can consider the trivial test function $f = \bbone_{[0,1]}$ and obtain an
upper bound
\begin{align*}
    |E(x)| \colon &=
   \sum_{n=2}^\infty f\left( \frac{1}{x+2n}\right) \frac{1}{(x+2n)^{2t}}  \leq
    \|f\|_\infty  \left( \sum_{n=2}^\infty \frac{1}{(x+2n)^{2t}}\right)   \leq 
    \int_{1}^\infty \frac{1}{(2u)^{2t}} du \\ &
    =\frac{1}{(2t-1)2^{2t}} .
\end{align*}
In particular, if we take $t=\frac56$ then $|E(x)| \leq  3 \cdot 2^{-7/3}$.
Thus 
$$
\sup_{[0,1]} \frac{\mathcal L_{t} \bbone_{[0,1]} }{\bbone_{[0,1]}} 
\leq \sup_{[0,1]} \left(\frac{1}{(2+x)^{5/3}} + E(x) \right) \le 2^{-5/3} + 3 \cdot
2^{-7/3}<0.95.  
$$
We conclude that $P\left(\frac56\right) < 0$ and Lemma~\ref{tech}
implies that~$\dim_H(X_{even}) <\nolinebreak \frac{5}{6}$.
\end{example}

\begin{example}[An elementary bound for $\dim_H (X_{1 (8)})$ revisited]
This time we can take  $f(x) = 1 - \frac{x}{2}$ and consider the 
image under the transfer operator
\begin{equation}
  \label{trop:xoct:eq}
\mathcal L_{t}f(x) = f\left( \frac{1}{x+1}\right) \frac{1}{(x+1)^{2t}} +
\sum_{n=2}^\infty f\left( \frac{1}{x+8n-7}\right) \frac{1}{(x+8n-7)^{2t}}
\end{equation}
We have an upper bound for the remainder term
\begin{align*}
    |E(x)| \colon &= 
    \sum_{n=2}^\infty f\left( \frac{1}{x+8n-7}\right) \frac{1}{(x+8n-7)^{2t}}
    \\&=  \sum_{n=2}^\infty  \frac{1}{(x+8n-7)^{2t}} \left(1- \frac{1}{2(x+8n-7)} \right) \cr
    &\leq  \left(\int_1^\infty \frac{1}{(8u-7)^{2t}} du 
    - \frac{1}{2}\int_{2}^\infty \frac{1}{(8u-7)^{2t+1}} du \right) \cr
    &= \frac18\left( \frac{1}{2t-1} - \frac{1}{4t \cdot 9^{2t}} 
    \right).
\end{align*}
Substituting $t = \frac{5}{6}$ we obtain 
\begin{equation}
\label{E:xoct:eq}
|E(x)| \leq  \frac{3}{16} \left(1 - \frac{1}{5 \cdot 9^{5/3}} \right). 
\end{equation}
In particular, we can now estimate 
\begin{align*}
\sup_{[0,1]}  \frac{{\mathcal L}_{t} f(x)}{f(x)} 
&= \sup_{[0,1]} 
\left(\frac{f\bigl(\frac{1}{x+1}\bigr)}{f(x)(x+1)^{5/3}}+ \frac{E(x)}{f(x)}
\right) \leq
\sup_{[0,1]} \frac{f\bigl(\frac{1}{x+1}\bigr)}{f(x)(x+1)^{5/3}} +
\frac{\sup_{[0,1]}E(x)}{\inf_{[0,1]} f(x)} \\ &= 
\sup_{[0,1]} \frac{1-\frac{1}{2(x+1)}}{(1-\frac{x}{2})(x+1)^{5/3}} + \frac{3}{8}
\left(1 - \frac{1}{5 \cdot 9^{5/3}} \right) <1 .
\end{align*}
The result follows from Lemma~\ref{tech}.
\end{example}


\subsection{Symmetric Schottky group}

We can represent the limit set $X_\Gamma \subset \{z\in \mathbb C \mid
|z|=1\}$ of a Fuchsian Schottky group as the limit set of an associated  Markov
iterated function scheme.  To construct the contractions it is more convenient
to use the alternative model for hyperbolic space consisting of the upper
half-plane $\mathbb H^2 = \{x+iy\ \mid y > 0\}$ 
supplied with the Poincar\'e metric $ds^2 = \frac{dx^2+dy^2}{y^2}$.
Geodesics on the upper half plane~$\mathbb H^2$ are either circular arcs which
meet the boundary orthogonally or vertical lines. 
    Applying the transformation $ T: \mathbb D \to \mathbb H^2$ given by $ T : z \mapsto -i
    \frac{z-1}{z+1}$    
    we obtain three geodesics on~$\mathbb H^2$. 
     The group generated by reflections with respect to three geodesics with end
points $e^{\pi i(2j\pm1)/6}$, $j = 0,1,2$ in the unit disk when transferred to the half-plane becomes a group generated by
reflections with respect to half-circles, see Figure~\ref{fig:schottky}. 
\begin{figure}[h!]
    \begin{center}
        \includegraphics{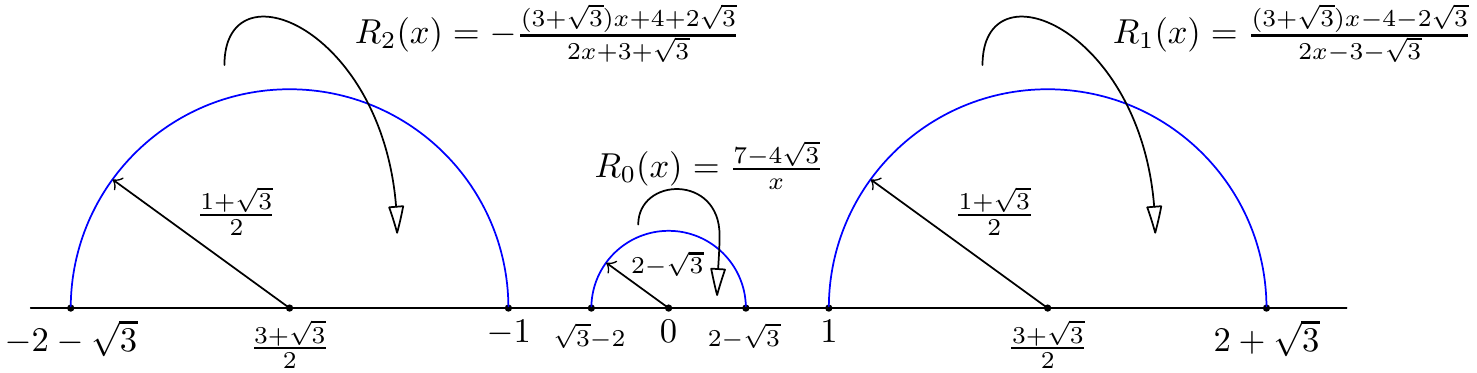}
        \end{center}
        \caption{Group $\Gamma$ generated by three reflections $R_0$, $R_1$,
        $R_2$. }
      \label{fig:schottky}
    \end{figure}

A reflection with
respect to a circle of radius~$r$ centred at~$c$ is given by the formula
\begin{equation}
    R(z) = \frac{r^2}{(z-c)} + c. 
    \label{refhyp:eq}
\end{equation}
 To compute the radius and the centre of the circle of reflection, we calculate
 the end points by the formula $T(e^{i2\varphi}) = \tan \varphi$ and 
applying~\eqref{refhyp:eq}, we obtain
\begin{equation}
    \label{rjref:eq}
    R_0(z) = \frac{7-4\sqrt{3}}{z}, \quad
    R_1(z) = \frac{(3 + \sqrt{3})z - 4 - 2\sqrt{3}}{2z-3 - \sqrt{3} }, \quad
    R_2(z) = -\frac{(3 + \sqrt{3})z + 4 + 2\sqrt{3}}{2z+3 + \sqrt{3}}. 
\end{equation}
    Then the limit set  $X_\Gamma \subset \cup_{j=0}^2 X_j  \subset \mathbb R$
    consists of accumulation points  of the set 
    $$
    \left\{ R_{j_1} R_{j_2} \cdots R_{j_n} (i) \mid 
    j_1, j_2, \cdots, j_n \in \{0,1,2\} \mbox{ where } j_r \neq j_{r+1} \mbox{
    for } 1 \leq r \leq n-1 \right\}.
    $$
    Since $T :\{z \colon |z| =1 \} \to \mathbb R \cup \{\infty\}$ is a conformal
    map we know that $X_\Gamma$ has the same dimension as the corresponding
    limit set in the unit circle. 

    \paragraph{Proof of Theorem \ref{mcm}.}
In order to apply the technique developed in \S\ref{sec:estimates} we need to
define a Markov iterated function scheme consisting of contractions
whose limit set coincides with~$X_\Gamma$. 
For instance we may consider the three intervals enclosed by the geodesics, more
precisely, we define
\begin{equation}
  \label{fuchsX:eq}
  X_0 := [-2+\sqrt{3}, 2-\sqrt{3}], \quad X_1 :=
  [1, 2+\sqrt{3} ], \quad X_2 := [-2-\sqrt{3}, -1].
\end{equation}
Then the limit set $X_\Gamma$ for $\Gamma$ can be identified with the limit set 
$X_A \subset \cup_{j=0}^2 X_j  \subset \mathbb R$ of the Markov iterated
function scheme  with contractions 
$R_j \colon \cup_{k \neq j} X_k \to X_j$ for $j=0,1,2$  
with transition matrix
$$
M = \begin{pmatrix}
  0&1&1\\
  1&0&1\\
  1&1&0\\
\end{pmatrix}.
$$
The associated transfer operator takes the form  
\begin{align*}
(\mathcal L_t \underline f)_0(z) &= f_1 (R_1(z)) \cdot |R^\prime_1(z)|^t +
f_2(R_2(z)) \cdot |R^\prime_2(z)|^t, \quad z \in X_0\\
(\mathcal L_t \underline f)_1(z) &= f_0 (R_0(z)) \cdot |R^\prime_0(z)|^t +
f_2(R_2(z)) \cdot |R^\prime_2(z)|^t, \quad z \in X_1\\
(\mathcal L_t \underline f)_2(z) &= f_0 (R_0(z)) \cdot |R^\prime_0(z)|^t +
f_1(R_1(z)) \cdot |R^\prime_1(z)|^t, \quad z \in X_2.
\end{align*}
Looking at the formulae~\eqref{rjref:eq} we may observe that $R_1(z) = -R_2(-z)$, 
$R_1^\prime(z) = R_2^\prime(-z)$, and $R_0(z)=-R_0(-z)$ and therefore $\mathcal
L_t$ preserves the subspace 
$$
V_0:=\{(f_0,f_1,f_2) \in C^\alpha(S)\mid f_1(z) =
f_2(-z), f_0(z) = f_0(-z) \}.
$$

We can apply the bisection method to get rigorous estimates on the dimension of
the limit set $X_\Gamma = X_A$ with the setting $S = \cup_{j=0}^2 X_j$,   
$m = 20$ and $\varepsilon=10^{-7}$. We take the interpolation nodes to be zeros
of Chebyshev polynomials transferred to each of the intervals $X_j$ affinely. It gives 
\begin{equation}
  \label{dimfuchs:eq}
t_0 \colon = 0.29554647 < \dim_H X_A < 0.29554648= \colon t_1.
\end{equation}
Lagrange---Chebyshev interpolation gives test functions 
$\underline f,\underline g \in V_0$,
whose components are presented in Figure~\ref{mcfigure}.

\begin{figure}[h]
    \begin{center}    
\begin{tabular}{ccc}
\includegraphics{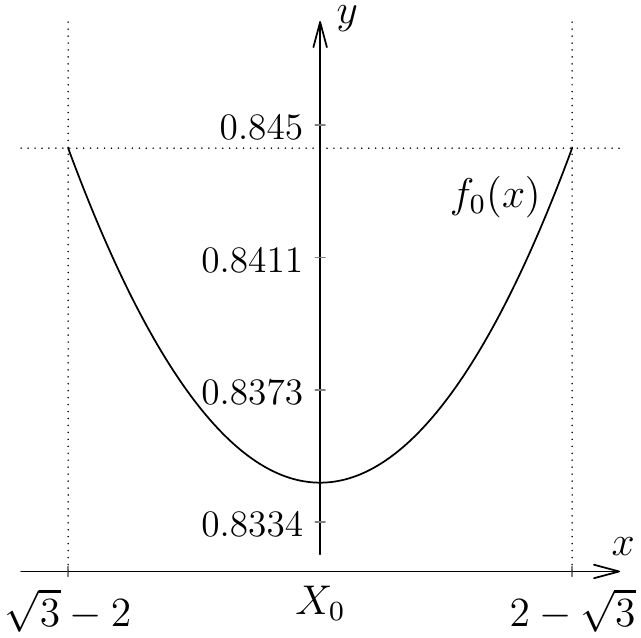}
& &
\includegraphics{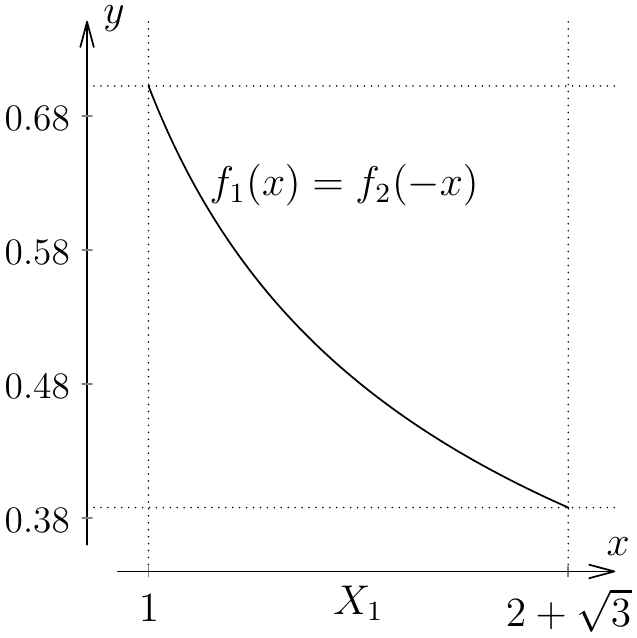}
\\
\end{tabular}
\end{center}
\caption{Theorem \ref{mcm}: 
The graphs of $f_0:X_0 \to \mathbb R$ and $f_1:X_1 \to \mathbb R$  
(the plots of $g_k$, $k=0,1$ are similar). The plot of $f_2$ is the mirror image
of $f_1$.} 
\label{mcfigure}
\end{figure}

We can estimate numerically
$$
\sup_S \frac{\mathcal L_{t_1} \underline f}{\underline f}  < 1-10^{-10}, \qquad
\inf_S\frac{\mathcal L_{t_0} \underline g}{\underline g}  > 1 + 10^{-8}. 
$$
    The result follows from Lemma~\ref{tech}.

        There is a simple connection between the dimension $\dim_H(X)$ of the
        limit set $X$ and the smallest eigenvalue $\lambda_0>0$
        of the Laplace---Beltrami operator on the non-compact surface $\mathbb
        H^2/\Gamma$ \cite{Sullivan}. More precisely, $ \lambda_0 =
        \dim_H(X_\Gamma)(1-\dim_H(X_\Gamma))$. Applying the
        estimates~\eqref{dimfuchs:eq}, we
        obtain  $\lambda_0 = 0.2081987565\pm 2.5\cdot10^{-9}$. 
    \qed 

    \begin{rem}
        By increasing $m$ it is an easy matter to get better estimates on the dimension of the limit set.
        For example, taking $m=25$ we can improve the bounds to
        $\dim(X_\Gamma) = 0.2955464798\,845\pm 4.5
        \cdot10^{-12}$   and $ \lambda_0 = 0.208198758112 \pm 2.5\cdot
        10^{-13}$. The coefficients of the corresponding test functions are
        given in~\S\ref{Appendix VII}. 
    \end{rem}


    \subsubsection{Other symmetric Schottky groups}

    More generally, McMullen~\cite{McMullen} considered the Schottky  group $\Gamma_\theta =
    \langle R_0, R_1, R_2\rangle$ generated by reflections $R_0, R_1, R_2: \mathbb D^2 \to \mathbb D^2$ 
    in three symmetrically placed geodesics (with respect to the Poincar\'e metric) with six end 
    points~$e^{2 \pi i j/3 \pm \theta/2}$, $j=0,1,2$ on the unit circle (Figure~\ref{schottky1}).

    \begin{figure}[h!]
    \begin{center}
    \includegraphics{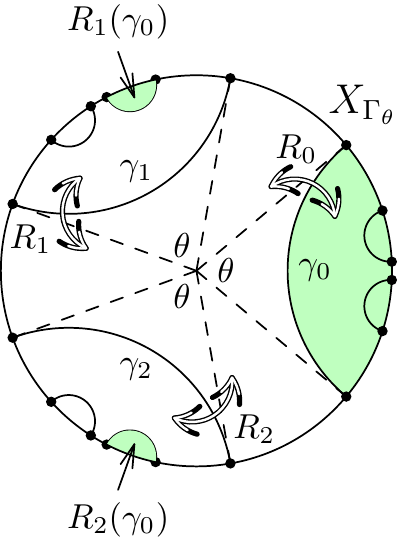}
    \end{center}
    \caption{The group $\Gamma_\theta$ generated by three reflections in geodesics~$\gamma_0$, $\gamma_1$, and $\gamma_2$.} 
    \label{schottky1}
    \end{figure}
    Similarly to the special case that $\theta = \frac{\pi}{3}$ which we have already considered,
    one can transform the unit disk $\mathbb D^2$ to the upper half plane $\mathbb H^2$ and compute the centres and
    the radii of reflections
    \begin{align*}
    c_j & = \frac12\left(\tan\left(\frac{\pi j }3 + \frac\theta4\right)+
    \tan\left(\frac{\pi j}3 - \frac\theta 4\right)\right), \mbox{ and }\\
    r_j & = \frac12\left|\tan\left(\frac{\pi j }3 + \frac\theta4\right)-
    \tan\left(\frac{\pi j}3 - \frac\theta 4\right)\right|, \ j = 0,1,2.
    \end{align*}
    The limit set $X_{\Gamma_\theta}$ is again defined as the accumulation points of the orbit $\Gamma_\theta i$.  
    We can introduce a corresponding Markov iterated function scheme whose limit set coincides with  
    $X_{\Gamma_\theta}$.  
    We can consider two representative examples and estimate the Hausdorff dimension of the associated limit set.
    \begin{example}[$\theta= 2\pi/9$]
    In this case setting $m=15$ 
    we can obtain an estimate to 11 decimal places of the form
    $$
    \dim_H(X_{\Gamma_\theta}) = 0.2177658102\,55 \pm 5 \cdot 10^{-12}.
    $$
    This agrees with McMullen's result (given to $8$ decimal places).
    %
    \end{example}

    \begin{example}[$\theta= \pi/9$]
    In this case we can let $m=12$ to deduce an estimate to 11 decimal places of the form
    $$
    \dim_H(X_{\Gamma_\theta}) = 0.1511836820\,35 \pm 5 \cdot 10^{-12}.
    $$
    This agrees with McMullen's result (given to $8$ decimal places).
\end{example}


    \subsection{Other iterated function schemes}
    \label{sec:moreExamples}
    To conclude we will collect together  a number of other examples of iterated
    function schemes that have attracted attention of other authors 
    and give estimates on the Hausdorff dimension of their limit sets.

    \subsubsection{Non-linear fractional example}
    So far we have been studying iterated function schemes generated by
    linear fractional  transformations. 
    Following~\cite{McMullen}, \S6  we will consider 
    a simple example of a map of $\mathbb H^2$ of a different nature. 
    For any $0 < t \leq 1$ we can define
    $$
    f_t(z) = \frac{z}{t} - \frac{1}{z}.
    $$
    If $t = 1$ then the real line is $f$-invariant and 
    there is no strictly  smaller closed invariant set.  If $0 < t < 1$ then
    there there exist a $f$-invariant Cantor $X_t \subset \mathbb
    R$~\cite{Ruelle82}.
    \begin{example}[$t = \frac{1}{2}$] 
        The map $f(z) =  2z - \frac{1}{z}$ has a limit set $X \subset [-1,1]$
        and there are two inverse branches $T_1, T_2\colon [-1,1] \to [-1,1]$ given by
        
        \begin{align}
            T_1(x) &= \frac{1}{4} (x - \sqrt{8 + x^2]})\cr
            T_2(x) &= \frac{1}{4} (x + \sqrt{8 + x^2}),\label{nonlin:eq}
        \end{align}
        which define a Bernoulli system on $[-1,1]$. The transfer operator
        defined by~\eqref{trop:eq} takes the form 
        $$
        (\mathcal L_t f)(x) = f(T_1(x)) |T_1^\prime(x)|^t +  f(T_2(x))
        |T_2^\prime(x)|^t.
        $$
        We may observe that $T_1(x) = -T_2(-x)$ and
        $T_1^\prime(x)=T^\prime_2(-x)$. It follows that the transfer operator
        preserves subspaces consisiting of odd and even functions.
        Applying the bisection method with $S=[-1,1]$, $m=10$ we obtain
        that\footnote{Table 14 in~\cite{McMullen} gives an estimate
        $0.49344815$, which is correct except for the last two significant figures.  An
        inconsequential typographical mistake in~\cite{McMullen} is that
        there is an incorrect sign in the equation in the caption to Table~14.} 
        $$
        \dim_H X = 0.4934480908\,025 \pm  5\cdot 10^{-13}.
        $$
        The corresponding test functions $f$ and $g$ for $t_0 = 0.493448088\,02$ and $t_1 =
        0.4934480908\,03$, respectively, turn out to be even and given by
        $$
        f(x) = \sum_{n=0}^{7} a_{2n} x^{2n} \qquad g(x) = \sum_{n=0}^7 b_{2n}
        x^{2n}, 
        $$
        which are plotted in Figure~\ref{blasfig} and whose coefficients are
        given in \S\ref{Appendix VIII}.

            \begin{figure}[h!!]
                \centerline{   \includegraphics{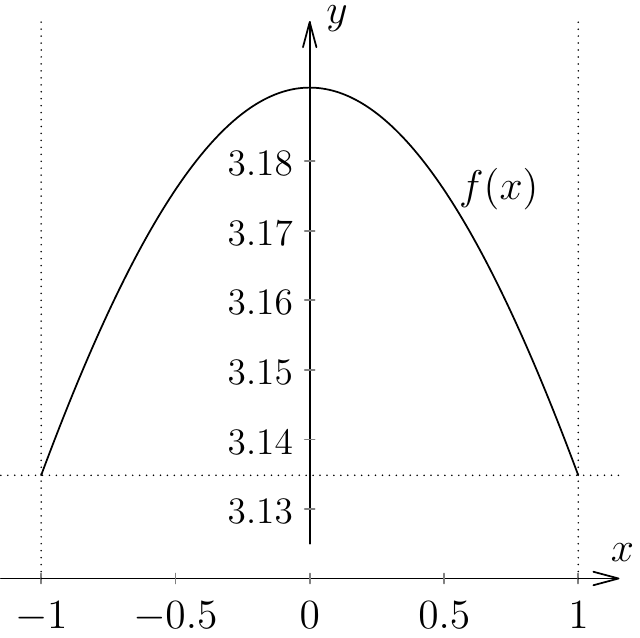}  }
                \caption{A plot of the function $f$ for the
                system~\eqref{nonlin:eq} (the function $g$ being similar)}
            \label{blasfig}
            \end{figure}

        We can also compute
        $$
        \inf_S \frac{\mathcal L_{t_0} f}{f} >  1 + 10^{-13} \qquad 
        \sup_S\frac{\mathcal L_{t_1} g}{g} < 1 - 10^{-13}.
        $$
        to justify the dimension estimates above.

    \end{example}

    \begin{rem}
        These estimates can easily be improved  by increasing the number of Chebyshev points.
        For example, letting $m = 20$ gives a better estimate
        $$
        \dim_H(X) = 0.4934480908\,02613 \pm 10^{-15}.
        $$
    \end{rem}

    \subsubsection{Hensley Examples}
    \label{ss:hensley}
    In a well known article from 1992, Hensley~\cite{Hensley92} presented an algorithm for
    calculating the Hausdorff dimension of the limit sets for suitable iterated
    function schemes. In this article there was included a table containing
    estimates on selected examples which were good for the computational
    resources available at the time and continued to be quoted up to the present time. 
    It is a simple matter to apply the method~in \S\ref{sec:estimates} 
    to improve these estimates. We give lower and upper bounds in
    Table~\ref{tab:hensley}.  
     \begin{table}
         \begin{center}
        \begin{tabular}{ |l|l|c|c| } 
            \hline
            Alphabet $A$ & Hensley Estimate & $d$ \\ 
            \hline
         $  1,2				   $ & $   0.5312805062\,772051416   $ & $   0.5312805062\,7720514162\,4    $  \\   
         $  1,3				   $ & $   0.4544827                 $ & $   0.4544890776\,6182874384\,5    $  \\   
         $  1,4				   $ & $   0.4111827                 $ & $   0.4111827247\,7479177684\,4    $  \\   
         $  2,3				   $ & $   0.337437                  $ & $   0.3374367808\,0606363630\,4    $  \\   
         $  2,4				   $ & $   0.306313                  $ & $   0.3063127680\,5278403027\,7    $  \\   
         $  3,4				   $ & $   0.263737                  $ & $   0.2637374828\,9742655875\,0    $  \\   
         $  1,2,3			   $ & $   0.7056609080\,            $ & $   0.7056609080\,2873823060\,7    $  \\   
         $  1,2,4			   $ & $   0.66922149                $ & $   0.6692214869\,1028607643\,2    $  \\   
         $  1,3,4			   $ & $   0.6042422606\,9111965     $ & $   0.6042422577\,5648956551\,0    $  \\   
         $  2,3,4			   $ & $   0.480696                  $ & $   0.4806962223\,1757304132\,2    $  \\   
         $  1,2,3,4		       $ & $   0.788946                  $ & $   0.7889455574\,8315397254\,0    $  \\   
         $  1,2,7			   $ & $   0.6179036954\,6338        $ & $   0.6179036954\,6337565066\,2    $  \\   
         $  1,3,7			   $ & $   0.55324225                $ & $   0.5532422505\,6731096881\,6    $  \\   
         $  1,4,7			   $ & $   0.51788376                $ & $   0.5178837570\,0691696528\,4    $  \\   
         $  2,3,7			   $ & $   0.43801241                $ & $   0.4380124057\,1403118230\,1    $  \\   
         $  2,4,7			   $ & $   0.410329                  $ & $   0.4103293158\,3768700408\,7    $  \\   
         $  3,4,7			   $ & $   0.36757914                $ & $   0.3675791395\,9190093176\,3    $  \\   
         $  1,2,3,7            $ & $   0.75026306                $ & $   0.7502630613\,3714304325\,2    $  \\   
         $  1,2,4,7            $ & $   0.7185418875\,            $ & $   0.7185418874\,7036994981\,8    $  \\   
         $  2,3,4,7            $ & $   0.540036                  $ & $   0.5400358121\,5475951902\,6    $  \\   
         $  1,2,3,4,7          $ & $   0.820004                  $ & $   0.8200039471\,2686900746\,5    $  \\   
         $  10,11              $ & $   0.146921                  $ & $   0.1469212353\,9078346331\,0    $  \\   
         $  100,10000          $ & $   0.052247                  $ & $   0.0522465926\,3865887865\,1    $  \\   
         $  2,7                $ & $   0.26022398                $ & $   0.2602238774\,2217867170\,7    $  \\   
         $  1,3,4,7            $ & $   0.66015538                $ & $   0.6601553798\,3237807776\,6    $  \\   
         $  1,7                $ & $   0.34623824                $ & $   0.3462382435\,3395787983\,0    $  \\   
         $  4,7                $ & $   0.2052533419\,4           $ & $   0.2052534193\,6736493221\,5    $  \\   
         $  3,7                $ & $   0.2249239471\,918         $ & $   0.2249239471\,9177898918\,3    $  \\   
         $  1,2,3,4,5          $ & $   0.8368294437\,            $ & $   0.8368294436\,8120882244\,2    $  \\   
         $  2,3,4,5            $ & $   0.55963645                $ & $   0.5596364501\,6477671331\,0    $  \\   
         $  2,3,5              $ & $   0.4616137                 $ & $   0.4616136840\,1828922267\,4    $  \\   
         $  1, 500             $ & $   0.1094760117\,37          $ & $   0.1094760117\,3723275274\,5    $  \\   
            \hline
        \end{tabular}
    \end{center}
        \caption{Numerical data for Hensley examples from~\cite{Hensley92};
        $\dim_H X_A = d \pm 10^{-20}$. }
        \label{tab:hensley}

    \end{table}
    
    \begin{table}[h!]
\begin{center}
       \scalebox{0.9}{%
       \begin{tabular}{ |l|l|c| } 
 \hline
 \multirow{2}{*}{Alphabet $A$} & Jenkinson &\multirow{2}{*}{  $d$ } \\ 
 & Estimate  & \\
 \hline
$  1,3,8       $	&	$ 0.5438  $	&	$ 0.5438505824\,0696620129\,10871  $  \\	 
$  1,3,6       $	&	$ 0.5652  $	&	$ 0.5652752192\,8623250537\,07768  $  \\	
$  1,3,5       $	&	$ 0.5813  $	&	$ 0.5813668211\,3469731449\,44763  $  \\	
$  1,2,10      $	&	$ 0.5951  $	&	$ 0.5951117365\,4560755184\,18957  $  \\	
$  1,3,4       $	&	$ 0.6042  $	&	$ 0.6042422576\,9541848140\,50596  $  \\	
$  1,2,7, 40   $	&	$ 0.6265  $	&	$ 0.6265741168\,9229403866\,4271   $  \\	
$  1,2,5       $	&	$ 0.6460  $	&	$ 0.6460620828\,3482621991\,26074  $  \\	
$  1,2,5, 40   $	&	$ 0.6532  $	&	$ 0.6532480771\,5774872727\,88226  $  \\	
$  1,2,4       $	&	$ 0.6692  $	&	$ 0.6692214868\,6131601289\,10582  $  \\	
$  1,2,4,40    $	&	$ 0.6754  $	&	$ 0.6754204446\,6970405658\,86491  $  \\	
$  1,2,4,15    $	&	$ 0.6899  $	&	$ 0.6899117699\,4923640369\,39765  $  \\	
$  1,2,4,6     $	&	$ 0.7275  $	&	$ 0.7275240485\,5844736070\,17215  $  \\	
$  1,2,4,5     $	&	$ 0.7400  $	&	$ 0.7400268606\,0207750663\,59866  $  \\	
$  1,2,3,6     $	&	$ 0.7588  $	&	$ 0.7588596765\,7522348478\,34758  $  \\	
$  1,2,3,5     $	&	$ 0.7709  $	&	$ 0.7709149398\,4418222560\,66922  $  \\	
$  1,2,3,4,10  $	&	$ 0.8081  $	&	$ 0.8081711218\,9508471948\,06225  $  \\	
$  1,2,3,4,6   $	&	$ 0.8269  $	&	$ 0.8269084945\,9163116837\,24267  $  \\	
$  1,2,3,4,5,9 $	&	$ 0.8541  $	&	$ 0.8541484705\,3932261542\,70362  $  \\	
$  1,2,3,4,5,7 $	&	$ 0.8616  $	&	$ 0.8616561744\,0626491056\,99743  $  \\	
$  1,2,3,4,5,6   $	&	$ 0.8676  $	&	$ 0.8676191730\,6718378091\,2243   $  \\	
$  1,2,3,4,5,6,8 $	&	$ 0.8851  $	&	$ 0.8851175915\,5644894823\,12343  $  \\	
$  1,2,3,4,5,6,7 $	&	$ 0.8889  $	&	$ 0.8889553164\,9195167843\,64394  $  \\	
$  1,2,\ldots, 8   $	&	$ 0.9045  $	&	$ 0.9045526893\,2916142728\,20095  $  \\
$  1,2,\ldots, 9   $	&	$ 0.9164  $	&	$ 0.9164211122\,6835174040\,64645  $  \\
$  1,2,\ldots, 10  $	&	$ 0.9257  $	&	$ 0.9257375908\,8754612367\,25506  $  \\
$  1,2,\ldots, 13  $	&	$ 0.9445  $	&	$ 0.9445341091\,7126158776\,76671  $  \\
$  1,2,\ldots, 18  $	&	$ 0.961   $	&	$ 0.9611931848\,1599230516\,44346  $  \\
$  1,2,\ldots, 34  $	&	$ 0.980   $	&	$ 0.9804196247\,7958255969\,58015  $  \\
 \hline
\end{tabular}
}
\end{center}
    \caption{Numerical data for Jenkinson examples from~\cite{Jenkinson};
    $\dim_H X_A = d \pm 2\cdot10^{-24}$. }
        \label{tab:jenkinson}
      \end{table}
    In Table~1 in another paper by Hensley~\cite{Hensley}, there are a number of
    numerical results on Hausdorff dimension various of limit sets.  Let us
    consider two typical examples from the
    list. 

    \paragraph{(i).} Let $X_{1,2,7} = \left\{[0;a_1,a_2, a_3, \cdots] \mid  a_n \in \{1,2, 7\}\right\}$.
    Hensley presents an estimate 
    $$
    \dim_H(X_{1,2,7}) = 0.6179036954\,6338,
    $$
    accurate to~$13$ decimal places. The
    bisection method with $S=[0,1]$, $\varepsilon = 10^{-23}$ and $m=30$ gives 
    $$
    \dim_H(X) = 0.6179036954\,6337565066\,3413 \pm 10^{-24},
    $$
    with the corresponding test functions $\underline f$ and $\underline g$
    satisfying 
    $$
    \inf_S\frac{\mathcal L_{t_0} g}{g}  >  1+10^{-24}, \qquad 
    \sup_S\frac{\mathcal L_{t_1} f}{f}  <  1-10^{-24}. 
    $$

    \paragraph{(ii).} Let  $X_{1,3,4} = \left\{[0;a_1,a_2, a_3, \cdots]  \mid  a_n \in \{1,3, 4\}\right\}$. 
    Hensley presents an estimate 
    $\dim_H(X_{1,3,5}) = 0.6042422606\,9111965$.  However, this is only accurate
    to seven decimal places (there seeming to be a typographical error) and applying the
    bisection method with $S=[0,1]$, $\varepsilon = 10^{-23}$ and $m=30$ we can  correct the estimate as follows. 
    $$
    \dim_H(X) = 0.6042422577\,5648956551\,0773 \pm 10^{-24}.
    $$
    \subsubsection{Other limit sets}
    In~\cite{Moreira} Moreira considered a limit set~$X$ for the IFS  
    $$
    T_1(x) = \frac{1}{1+x} \mbox{ and } T_2(x) = \frac{1}{2+\frac{1}{2+x}}.
    $$
    After Theorem 3.4 therein he  gives a rigorous  estimate  $0.353 <
    \dim_H(X) < 0.3572$\footnote{Their application was superseded by other work,
    so the interest in this bound is mainly academic.}.
    Applying the bisection method with $\varepsilon=10^{-30}$ and $m=40$ we
    obtain 
    $$
    \dim_H X = 0.3554004768\,3384079791\,6306289490\,45 \pm 5\cdot10^{-32}.
    $$

    There are also additional examples studied by Jenkinson, in connection with his numerical investigation of the 
Texan conjecture.  It is a simple matter to apply the bisection method to compute intervals $[t_0,t_1]$ containing the actual values
and these are presented in Table~\ref{tab:jenkinson}.

    \appendix 
    \section{Coefficients for polynomials}
    
    For completeness, we collect together the coefficients of the polynomials
    which appear in the proofs of the theorems.  The exceptions to this is
    Theorem~\ref{good} where the polynomials are of degree~$200$. 
    However, in all of these examples  the reader may easily reconstruct these
    polynomials using the method described in \S\ref{test} and
    \S\ref{bis}. 
    
    \smallskip
    \centerline{\emph{Coefficients given in this section are exact rational
    numbers.}} 

    \subsection{Estimates of $\dim_H \mathcal M \setminus \mathcal L$}

    \subsubsection{Part 1: $(\mathcal M \setminus \mathcal
    L)\cap(\sqrt5,\sqrt{13})$} 
    \label{Appendix I}
    We present coefficients of the test functions $\underline f = (f_1,f_2,f_3,f_4)$ and
   
    $\underline g=(g_1,g_2,g_3,g_4)$ 
    used in~\eqref{testp1:eq}.  
    $$
    f_j = \sum_{k=0}^7 a_k^j x^k \quad 
    g_j = \sum_{k=0}^7 b_k^j x^k, \qquad j = 1,2,3,4.
    $$
    Straightforward calculation shows that the functions~$f_j$ and~$g_j$ are
    monotone decreasing on $[0,1]$ and achieve
    their minima at~$1$. For convenience, we give a lower bound $s_j <  \min(f_j(1), g_j(1))$.    
    \begin{center} 
    \hspace{-5mm}\begin{tabular}{|c | l ||c | l||c|l||c|l| }
        \hline
        \multicolumn{2}{|c||}{ $f_1$ } & \multicolumn{2}{|c||}{ $f_2$ } &
        \multicolumn{2}{|c||}{ $f_3$ } & \multicolumn{2}{|c|}{ $f_4$ } \\
        \hline
        $a^1_0$&$  \phantom{-}0.9719420630    $ & $a^2_0$&$  \phantom{-}0.6996881504    $    &  $a^3_0$&$  \phantom{-}0.5151068706  $  &  $a^4_0$&$    \phantom{-}1.0035153909  $  \\ 
        $a^1_1$&$  -0.4083622154    $ &	$a^2_1$&$  -0.3257856180    $    &	$a^3_1$&$  -0.1503102708  $  &	$a^4_1$&$   -0.3675009146  $  \\ 
        $a^1_2$&$  \phantom{-}0.2105627503    $ &	$a^2_2$&$  \phantom{-}0.1808409982    $    &	$a^3_2$&$  \phantom{-}0.0521402712  $  &	$a^4_2$&$   \phantom{-}0.1667939317  $  \\ 
        $a^1_3$&$  -0.1166905024    $ &	$a^2_3$&$  -0.1054480473    $    &	$a^3_3$&$  -0.0190434708  $  &	$a^4_3$&$   -0.0824212096  $  \\ 
        $a^1_4$&$  \phantom{-}0.0649256757    $ &	$a^2_4$&$  \phantom{-}0.0606297675    $    &	$a^3_4$&$  \phantom{-}0.0070462566  $  &	$a^4_4$&$   \phantom{-}0.0416454207  $  \\ 
        $a^1_5$&$  -0.0321113965    $ &	$a^2_5$&$  -0.0305675512    $    &	$a^3_5$&$  -0.0024652813  $  &	$a^4_5$&$   -0.0191792606  $  \\ 
        $a^1_6$&$  \phantom{-}0.0114531320    $ &	$a^2_6$&$  \phantom{-}0.0110148596    $    &	$a^3_6$&$  \phantom{-}0.0006866462  $  &	$a^4_6$&$   \phantom{-}0.0065321353  $  \\ 
        $a^1_7$&$  -0.0020311821    $ &	$a^2_7$&$  -0.0019639179    $    &	$a^3_7$&$  -0.0001041463  $  &	$a^4_7$&$   -0.0011267545  $  \\ 
        \hline
        $s_1 $&$   \phantom{-}0.6    $&$ s_2 $&$    \phantom{-}0.4$ & $ s_3 $ & $\phantom{-}0.4$ & $s_4$ & $    \phantom{-}0.7$ \\
        \hline
    \end{tabular}

    \begin{tabular}{|c|l||c|l||c|l||c|l| }
      \hline
      \multicolumn{2}{|c||}{$g_1$} & \multicolumn{2}{|c||}{$g_2$} &
      \multicolumn{2}{|c||}{$g_3$} & \multicolumn{2}{|c|}{$g_4$} \\
        \hline
            $b^1_0$&$  \phantom{-}0.9719420489   $  &  $b^2_0$&$  \phantom{-}0.6996881913   $  &  $b^3_0$&$  \phantom{-}0.5151068139  $   &	$b^4_0$&$   \phantom{-}1.0035153929    $\\ 
        	$b^1_1$&$  -0.4083624405   $  &	$b^2_1$&$  -0.3257858144   $  &	$b^3_1$&$  -0.1503103363  $   &	$b^4_1$&$  -0.3675011259    $\\
        	$b^1_2$&$  \phantom{-}0.2105629174   $  &	$b^2_2$&$  \phantom{-}0.1808411480   $  &	$b^3_2$&$  \phantom{-}0.0521403058  $   &	$b^4_2$&$  \phantom{-}0.1667940700    $\\
        	$b^1_3$&$  -0.1166906125   $  &	$b^2_3$&$  -0.1054481493   $  &	$b^3_3$&$  -0.0190434862  $   &	$b^4_3$&$  -0.0824212916    $\\
        	$b^1_4$&$  \phantom{-}0.0649257436   $  &	$b^2_4$&$  \phantom{-}0.0606298318   $  &	$b^3_4$&$  \phantom{-}0.0070462630  $   &	$b^4_4$&$  \phantom{-}0.0416454670    $\\
        	$b^1_5$&$  -0.0321114321   $  &	$b^2_5$&$  -0.0305675854   $  &	$b^3_5$&$  -0.0024652837  $   &	$b^4_5$&$  -0.0191792835    $\\  
        	$b^1_6$&$  \phantom{-}0.0114531451   $  &	$b^2_6$&$  \phantom{-}0.0110148723   $  &	$b^3_6$&$  \phantom{-}0.0006866469  $   &	$b^4_6$&$  \phantom{-}0.0065321434    $\\  
        	$b^1_7$&$  -0.0020311845   $  &	$b^2_7$&$  -0.0019639202   $  &	$b^3_7$&$  -0.0001041464  $   &	$b^4_7$&$  -0.0011267559    $\\  
        \hline
            $s_1 $&$   \phantom{-}0.6 $&$ s_2 $&$ \phantom{-}0.4$ & $ s_3 $ & $\phantom{-}0.4$ & $s_4$ & $    \phantom{-}0.7$ \\
        \hline
    \end{tabular}
    \end{center}

    \subsubsection{Part 2: $(\mathcal M \setminus \mathcal
    L)\cap(\sqrt{13},3.84)$}
    \label{Appendix II}
    We present coefficients of the test functions $\underline f = (f_1,f_2,f_3)$ and
    $\underline g=(g_1,g_2,g_3)$ 
    used in~\eqref{testp2:eq}.  
    $$
    f_j = \sum_{k=0}^7 a_k^j x^k \quad 
    g_j = \sum_{k=0}^7 b_k^j x^k, \quad j =1,2,3.
    $$
    Straightforward calculation shows that the functions~$f_j$ and~$g_j$ are
    monotone decreasing on $[0,1]$ and achieve
    their minima at~$1$. For convenience, we give a lower bound $s_j <  \min(f_j(1), g_j(1))$.    
    \begin{center}
    \begin{tabular}{|c|l|c|l|c|l| }
      \hline
      \multicolumn{2}{|c|}{$f_1$} & \multicolumn{2}{|c|}{$f_2$} &  \multicolumn{2}{|c|}{$f_3$} \\
      \hline
        $a^1_0$& $ \phantom{-}0.8909247279  $  &  $a^2_0$& $  \phantom{-}1.0057862651  $  &  $a^3_0$& $  \phantom{-}0.4637543240  $  \\ 
        $a^1_1$& $ -0.5666958388  $  &  $a^2_1$& $  -0.6057959903  $  &  $a^3_1$& $  -0.1989455013  $  \\ 
        $a^1_2$& $ \phantom{-}0.3563184411  $  &  $a^2_2$& $  \phantom{-}0.3687733983  $  &  $a^3_2$& $  \phantom{-}0.0812067365  $  \\ 
        $a^1_3$& $ -0.2268317558  $  &  $a^2_3$& $  -0.2307067719  $  &  $a^3_3$& $  -0.0328166697  $  \\ 
        $a^1_4$& $ \phantom{-}0.1399500506  $  &  $a^2_4$& $  \phantom{-}0.1411336470  $  &  $a^3_4$& $  \phantom{-}0.0130694788  $  \\ 
        $a^1_5$& $ -0.0741972949  $  &  $a^2_5$& $  -0.0745402261  $  &  $a^3_5$& $  -0.0048273012  $  \\ 
        $a^1_6$& $ \phantom{-}0.0275558152  $  &  $a^2_6$& $  \phantom{-}0.0276379657  $  &  $a^3_6$& $  \phantom{-}0.0013941859  $  \\ 
        $a^1_7$& $ -0.0049924560  $  &  $a^2_7$& $  -0.0050037264  $  &  $a^3_7$& $  -0.0002161328  $  \\ 
      \hline
        $s_1$ & $ \phantom{-}0.5$ & $s_2$ & $\phantom{-}0.6$ & $s_3$ & $\phantom{-}0.3$  \\
      \hline 
    \end{tabular}
    
    \begin{tabular}{|c|l|c|l|c|l| }
        \hline
        \multicolumn{2}{|c|}{ $g_1$} & \multicolumn{2}{|c|}{ $g_2$} & \multicolumn{2}{|c|}{ $g_3$} \\
        \hline
        $b^1_0$& $   \phantom{-}0.8909246136  $  &  $b^2_0$& $  \phantom{-}1.0057862594  $  &  $b^3_0$& $ \phantom{-}0.4637545188  $   \\ 
        $b^1_1$& $   -0.5666952424  $  &  $b^2_1$& $  -0.6057953968  $  &  $b^3_1$& $ -0.1989454076  $   \\ 
        $b^1_2$& $   \phantom{-}0.3563178857  $  &  $b^2_2$& $  \phantom{-}0.3687728381  $  &  $b^3_2$& $ \phantom{-}0.0812066594  $   \\ 
        $b^1_3$& $   -0.2268313254  $  &  $b^2_3$& $  -0.2307063388  $  &  $b^3_3$& $ -0.0328166280  $   \\ 
        $b^1_4$& $   \phantom{-}0.1399497528  $  &  $b^2_4$& $  \phantom{-}0.1411333481  $  &  $b^3_4$& $ \phantom{-}0.0130694591  $   \\ 
        $b^1_5$& $   -0.0741971263  $  &  $b^2_5$& $  -0.0745400571  $  &  $b^3_5$& $ -0.0048272932  $   \\ 
        $b^1_6$& $   \phantom{-}0.0275557503  $  &  $b^2_6$& $  \phantom{-}0.0276379006  $  &  $b^3_6$& $ \phantom{-}0.0013941834  $   \\ 
        $b^1_7$& $   -0.0049924440  $  &  $b^2_7$& $  -0.0050037144  $  &  $b^3_7$& $ -0.0002161324  $   \\ 
        \hline
        $s_1$ &  $   \phantom{-}0.5$ & $s_2$ & $\phantom{-}0.6$ & $ s_3$ & $ \phantom{-}0.3$ \\
        \hline
    \end{tabular}
    \end{center}
  
    \subsubsection{Part 3: $(\mathcal M \setminus \mathcal
    L)\cap(3.84,3.92)$}
    \label{Appendix III}
   We present coefficients of the polynomial test functions $\underline f =
    (f_1,\ldots,f_7)$ and  $\underline g=(g_1,\ldots,g_7)$ 
    used in~\eqref{testp3:eq}.  It follows from the equality between columns of the
    transition matrix $M$ that certain components are identical. 
    \begin{align*}
        f_1 &= f_2 = \sum_{k=0}^{7} a_k^1 x^k, \quad &&
        g_1 = g_2 = \sum_{k=0}^{7} b_k^1 x^k; \\ 
        f_4 &= f_5 = f_6 = \sum_{k=0}^{7} a_k^4 x^k, \quad &&
        g_4 = g_5 = g_6 = \sum_{k=0}^{7} b_k^4 x^k; \\ 
        f_j &= \sum_{k=0}^{7} a_k^j x^k, \quad &&
        g_j = \sum_{k=0}^{7} b_k^j x^k; \quad j =3,7,8,9.
    \end{align*}
    Straightforward calculation shows that the functions~$f_j$ and~$g_j$ are
    monotone decreasing on $[0,1]$ and achieve
    their minima at~$1$. For convenience, we give a lower bound $s_j <  \min(f_j(1), g_j(1))$.    

    \begin{center}
    \begin{tabular}{ |c|l||c|l||c|l| }
        \hline
        \multicolumn{2}{|c||}{ $f_1$} & \multicolumn{2}{|c||}{ $f_3$} &
        \multicolumn{2}{|c|}{ $f_4$} \\
        \hline
        $a^1_0$& $  \phantom{-}0.8752491446  $  &   $a^3_0$& $   \phantom{-}0.8242977486 $   &  $a^4_0$& $  \phantom{-}0.9435673129 $ \\ 
        $a^1_1$& $  -0.5862938673  $  &   $a^3_1$& $   -0.5673413247 $   &  $a^4_1$& $  -0.6108763701 $ \\ 
        $a^1_2$& $  \phantom{-}0.3817921234  $  &   $a^3_2$& $   \phantom{-}0.3753845810 $   &  $a^4_2$& $  \phantom{-}0.3898319394 $ \\ 
        $a^1_3$& $  -0.2509102043  $  &   $a^3_3$& $   -0.2488176120 $   &  $a^4_3$& $  -0.2534505061 $ \\ 
        $a^1_4$& $  \phantom{-}0.1594107381  $  &   $a^3_4$& $   \phantom{-}0.1587442151 $   &  $a^4_4$& $  \phantom{-}0.1601941335 $ \\ 
        $a^1_5$& $  -0.0866513835  $  &   $a^3_5$& $   -0.0864513792 $   &  $a^4_5$& $  -0.0868796783 $ \\ 
        $a^1_6$& $  \phantom{-}0.0328162711  $  &   $a^3_6$& $   \phantom{-}0.0327671074 $   &  $a^4_6$& $  \phantom{-}0.0328711315 $ \\ 
        $a^1_7$& $  -0.0060331410  $  &   $a^3_7$& $   -0.0060262852 $   &  $a^4_7$& $  -0.0060406787 $ \\ 
        \hline
        $s_1$ & $\phantom{-}0.4$  & $s_3$ & $\phantom{-}0.4$ &  $s_4$ & $\phantom{-}0.5$ \\   
        \hline
    \end{tabular}

    \begin{tabular}{ |c|l||c|l||c|l| }
        \hline
        \multicolumn{2}{|c||}{ $f_7$} & \multicolumn{2}{|c||}{ $f_8$} &
        \multicolumn{2}{|c|}{ $f_9$} \\
        \hline
         $a^7_0$ & $  \phantom{-}0.2183470684  $   &   $a^8_0$& $  \phantom{-}0.5232653503   $ &   $a^9_0$& $    \phantom{-}1.0058222185  $  \\ 
         $a^7_1$ & $  -0.0767355580  $   &   $a^8_1$& $  -0.2253827130   $ &   $a^9_1$& $   -0.6095554508  $  \\ 
         $a^7_2$ & $  \phantom{-}0.0245616045  $   &   $a^8_2$& $  \phantom{-}0.0906682129   $ &   $a^9_2$& $   \phantom{-}0.3705808500  $  \\ 
         $a^7_3$ & $  -0.0076106356  $   &   $a^8_3$& $  -0.0360797100   $ &   $a^9_3$& $   -0.2326605853  $  \\ 
         $a^7_4$ & $  \phantom{-}0.0023068479  $   &   $a^8_4$& $  \phantom{-}0.0141972990   $ &   $a^9_4$& $   \phantom{-}0.1434620328  $  \\ 
         $a^7_5$ & $  -0.0006627640  $   &   $a^8_5$& $  -0.0052036970   $ &   $a^9_5$& $   -0.0765436603  $  \\ 
         $a^7_6$ & $  \phantom{-}0.0001576564  $   &   $a^8_6$& $  \phantom{-}0.0014967921   $ &   $a^9_6$& $   \phantom{-}0.0286722918  $  \\ 
         $a^7_7$ & $  -0.0000215279  $   &   $a^8_7$& $  -0.0002316038   $ &   $a^9_7$& $   -0.0052379012  $  \\ 
        \hline
         $s_7$ & $\phantom{-}0.1$ & $s_8$ & $\phantom{-}0.3$ & $s_9$  & $\phantom{-}0.5$ \\
        \hline
    \end{tabular}
    
    \begin{tabular}{|c|l||c|l||c|l| }
        \hline
        \multicolumn{2}{|c||}{ $g_1$} & \multicolumn{2}{|c||}{ $g_3$} &
        \multicolumn{2}{|c|}{ $g_4$} \\
        \hline
        $b^1_0$& $ \phantom{-}0.8752491756  $  &   $b^3_0$& $ \phantom{-}0.8242978103  $  &    $b^4_0$& $  \phantom{-}0.9435673145 $   \\
        $b^1_1$& $ -0.5862941999  $  &   $b^3_1$& $ -0.5673416594  $  &    $b^4_1$& $  -0.6108767042 $   \\
        $b^1_2$& $ \phantom{-}0.3817924531  $  &   $b^3_2$& $ \phantom{-}0.3753849098  $  &    $b^4_2$& $  \phantom{-}0.3898322718 $   \\
        $b^1_3$& $ -0.2509104702  $  &   $b^3_3$& $ -0.2488178772  $  &    $b^4_3$& $  -0.2534507734 $   \\
        $b^1_4$& $ \phantom{-}0.1594109280  $  &   $b^3_4$& $ \phantom{-}0.1587444046  $  &    $b^4_4$& $  \phantom{-}0.1601943239 $   \\
        $b^1_5$& $ -0.0866514938  $  &   $b^3_5$& $ -0.0864514894  $  &    $b^4_5$& $  -0.0868797888 $   \\
        $b^1_6$& $ \phantom{-}0.0328163145  $  &   $b^3_6$& $ \phantom{-}0.0327671507  $  &    $b^4_6$& $  \phantom{-}0.0328711749 $   \\
        $b^1_7$& $ -0.0060331491  $  &   $b^3_7$& $ -0.0060262933  $  &    $b^4_7$& $  -0.0060406868 $   \\
        \hline
        $s_1$ & $\phantom{-}0.4$  & $s_3$ & $\phantom{-}0.4$ &  $s_4$ & $\phantom{-}0.5$ \\   
        \hline
    \end{tabular}

    \begin{tabular}{|c|l||c|l||c|l| }
        \hline
        \multicolumn{2}{|c||}{ $g_7$} & \multicolumn{2}{|c||}{ $g_8$} &
        \multicolumn{2}{|c||}{ $g_9$} \\
        \hline
      $b^7_0$& $ \phantom{-}0.2183469977   $  & $b^8_0$& $  \phantom{-}0.5232652393 $   &  $b^9_0$& $  \phantom{-}1.0058222217    $  \\ 
      $b^7_1$& $ -0.0767355702   $  & $b^8_1$& $  -0.2253827769 $   &  $b^9_1$& $ -0.6095557922    $  \\ 
      $b^7_2$& $ \phantom{-}0.0245616149   $  & $b^8_2$& $  \phantom{-}0.0906682634 $   &  $b^9_2$& $ \phantom{-}0.3705811764    $  \\ 
      $b^7_3$& $ -0.0076106401   $  & $b^8_3$& $  -0.0360797369 $   &  $b^9_3$& $ -0.2326608405    $  \\ 
      $b^7_4$& $ \phantom{-}0.0023068496   $  & $b^8_4$& $  \phantom{-}0.0141973115 $   &  $b^9_4$& $ \phantom{-}0.1434622111    $  \\ 
      $b^7_5$& $ -0.0006627646   $  & $b^8_5$& $  -0.0052037021 $   &  $b^9_5$& $ -0.0765437625    $  \\ 
      $b^7_6$& $ \phantom{-}0.0001576566   $  & $b^8_6$& $  \phantom{-}0.0014967937 $   &  $b^9_6$& $ \phantom{-}0.0286723317    $  \\ 
      $b^7_7$& $ -0.0000215279   $  & $b^8_7$& $  -0.0002316041 $   &  $b^9_7$& $ -0.0052379086    $  \\ 
        \hline
      $s_7$ & $\phantom{-}0.1$ & $s_8$ & $\phantom{-}0.3$ & $s_9$ & $\phantom{-}0.5$ \\
        \hline
    \end{tabular}
    \end{center}

    \subsubsection{Part 4: $(\mathcal M \setminus \mathcal
    L)\cap(3.92,4.01)$}
    \label{Appendix IV}
 
    It follows from the equalities between columns of the transition matrix~$M$~\eqref{ap4:eq}, that
    components of the test function $\underline f =
    (f_{111},\ldots,f_{333})$ and  $\underline g=(g_{111},\ldots,g_{333})$ 
    used in~\eqref{testp4:eq} satisfy the following identities.
\begin{align*}
f_{111} &= f_{112} = f_{113} = \sum_{k=0}^7 a_k^{111} x^k  &&
g_{111} = g_{112} = g_{113} = \sum_{k=0}^7 b_k^{111} x^k,  \\
f_{211} &= f_{2rs} = \sum_{k=0}^7 a_k^{211} x^k &&
g_{211} = g_{2rs} = \sum_{k=0}^7 b_k^{211} x^k, \ 1\le r, s \le 3,\\
f_{121} &= f_{122} = f_{123} = \sum_{k=0}^7 a_k^{121} x^k  &&
g_{121} = g_{122} = g_{123}  = \sum_{k=0}^7 b_k^{121} x^k,  \\
f_{321} &= f_{322} = f_{323} = \sum_{k=0}^7 a_k^{123} x^k &&
g_{321} = g_{322} = g_{323} = \sum_{k=0}^7 b_k^{123} x^k, \\
f_{331} &= f_{332} = f_{333} = \sum_{k=0}^7 a_k^{133} x^k &&
g_{331} = g_{332} = g_{333} = \sum_{k=0}^7 b_k^{133} x^k.
\end{align*}
    Straightforward calculation shows that the functions~$f_j$ and~$g_j$ are
    monotone decreasing on $[0,1]$ and achieve
    their minima at~$1$. For convenience, we give a lower bound $s_{qrs} <
    \min(f_{qrs}(1), g_{qrs}(1))$.    

    \begin{center}
    \begin{tabular}{ |c|l||c|l||c|l| }
        \hline
        \multicolumn{2}{|c||}{ $f_{111}$} & \multicolumn{2}{|c||}{ $f_{121}$} &
        \multicolumn{2}{|c|}{ $f_{132}$} \\
        \hline
        $a_0^{111}$& $  \phantom{-}0.9884535079  $ & $a_0^{121}$& $  \phantom{-}0.9249363898 $ & $a_0^{132}$& $  \phantom{-}0.5745876575  $   \\ 
        $a_1^{111}$& $  -0.6902235127  $ & $a_1^{121}$& $  -0.6661893079 $ & $a_1^{132}$& $  -0.4908751097  $   \\
        $a_2^{111}$& $  \phantom{-}0.4629840472  $ & $a_2^{121}$& $  \phantom{-}0.4549027221 $ & $a_2^{132}$& $  \phantom{-}0.3761466270  $   \\
        $a_3^{111}$& $  -0.3124708511  $ & $a_3^{121}$& $  -0.3098685263 $ & $a_3^{132}$& $  -0.2757231992  $   \\
        $a_4^{111}$& $  \phantom{-}0.2030567394  $ & $a_4^{121}$& $  \phantom{-}0.2022427698 $ & $a_4^{132}$& $  \phantom{-}0.1878939498  $   \\
        $a_5^{111}$& $  -0.1122167788  $ & $a_5^{121}$& $  -0.1119770597 $ & $a_5^{132}$& $  -0.1064665719  $   \\
        $a_6^{111}$& $  \phantom{-}0.0429529760  $ & $a_6^{121}$& $  \phantom{-}0.0428949582 $ & $a_6^{132}$& $  \phantom{-}0.0412628171  $   \\
        $a_7^{111}$& $  -0.0079480141  $ & $a_7^{121}$& $  -0.0079400087 $ & $a_7^{132}$& $  -0.0076832222  $   \\
        \hline
        $s_{111}$ & $\phantom{-}0.5$ & $s_{121} $ & $\phantom{-}0.4$ & $s_{132}$ & $\phantom{-}0.2$ \\
        \hline
    \end{tabular}
 
    \begin{tabular}{ |c|l||c|l||c|l| }
        \hline
        \multicolumn{2}{|c||}{ $f_{133}$} & \multicolumn{2}{|c||}{ $f_{211}$} &
        \multicolumn{2}{|c|}{ $f_{311}$} \\
        \hline
    $a_0^{133}$& $ \phantom{-}0.8813943456  $   &   $a_0^{211}$&$   1.0003583962 $ &  $a_0^{311}$& $  \phantom{-}0.5081203394$   \\ 
    $a_1^{133}$& $ -0.6491441142  $   &   $a_1^{211}$&$  -0.6654322134 $ &  $a_1^{311}$& $  -0.2326160640$   \\ 
    $a_2^{133}$& $ \phantom{-}0.4489736082  $   &   $a_2^{211}$&$  \phantom{-}0.4295333392 $ &  $a_2^{311}$& $  \phantom{-}0.0972736889$   \\ 
    $a_3^{133}$& $ -0.3078936043  $   &   $a_3^{211}$&$  -0.2812396746 $ &  $a_3^{311}$& $  -0.0398841238$   \\ 
    $a_4^{133}$& $ \phantom{-}0.2016042974  $   &   $a_4^{211}$&$  \phantom{-}0.1785878079 $ &  $a_4^{311}$& $  \phantom{-}0.0160797615$   \\ 
    $a_5^{133}$& $ -0.1117833249  $   &   $a_5^{211}$&$  -0.0971302423 $ &  $a_5^{311}$& $  -0.0060041586$   \\ 
    $a_6^{133}$& $ \phantom{-}0.0428469743  $   &   $a_6^{211}$&$  \phantom{-}0.0368077301 $ &  $a_6^{311}$& $  \phantom{-}0.0017487411$   \\ 
    $a_7^{133}$& $ -0.0079332871  $   &   $a_7^{211}$&$  -0.0067699556 $ &  $a_7^{311}$& $  -0.0002726247$   \\ 
        \hline
        $s_{133}$ & $\phantom{-}0.4$ & $s_{211} $ & $\phantom{-}0.5$ & $s_{311}$ & $\phantom{-}0.3$ \\   
        \hline
    \end{tabular}

    \begin{tabular}{ |c|l||c|l||c|l| }
        \hline
        \multicolumn{2}{|c||}{ $f_{312}$} & \multicolumn{2}{|c||}{ $f_{321}$} &
        \multicolumn{2}{|c|}{ $f_{331}$} \\
        \hline
    $a_0^{312}$& $  \phantom{-}0.2013136513 $   &   $a_0^{321} $& $ \phantom{-}0.8072631514 $  &  $a_0^{331}  $& $  \phantom{-}1.0061848844  $   \\  
    $a_1^{312}$& $  -0.0743470595 $   &   $a_1^{321} $& $ -0.4648188145 $  &  $a_1^{331}  $& $  -0.6476491585  $   \\  
    $a_2^{312}$& $  \phantom{-}0.0244467077 $   &   $a_2^{321} $& $ \phantom{-}0.2575884636 $  &  $a_2^{331}  $& $  \phantom{-}0.4067501349  $   \\  
    $a_3^{312}$& $  -0.0077137186 $   &   $a_3^{321} $& $ -0.1452652325 $  &  $a_3^{331}  $& $  -0.2603687665  $   \\  
    $a_4^{312}$& $  \phantom{-}0.0023694138 $   &   $a_4^{321} $& $ \phantom{-}0.0809083385 $  &  $a_4^{331}  $& $  \phantom{-}0.1624071150  $   \\  
    $a_5^{312}$& $  -0.0006874054 $   &   $a_5^{321} $& $ -0.0398683005 $  &  $a_5^{331}  $& $  -0.0872158384  $   \\  
    $a_6^{312}$& $  \phantom{-}0.0001645838 $   &   $a_6^{321} $& $ \phantom{-}0.0141619841 $  &  $a_6^{331}  $& $  \phantom{-}0.0327839671  $   \\  
    $a_7^{312}$& $  -0.0000225598 $   &   $a_7^{321} $& $ -0.0025036826 $  &  $a_7^{331}  $& $  -0.0060002449  $   \\  
        \hline
        $s_{312}$ & $\phantom{-}0.1$ & $s_{321} $ & $\phantom{-}0.4$ & $s_{331}$ & $\phantom{-}0.5$ \\   
        \hline
    \end{tabular}

    \begin{tabular}{ |c|l||c|l||c|l| }
        \hline
        \multicolumn{2}{|c||}{ $g_{111}$} & \multicolumn{2}{|c||}{ $g_{121}$} &
        \multicolumn{2}{|c|}{ $g_{132}$} \\
        \hline
        $b_0^{111}$& $  \phantom{-}0.9884535050   $ & $b_0^{121}$& $ \phantom{-}0.9249364233 $ & $b_0^{132}$& $  \phantom{-}0.5745877791  $   \\ 
        $b_1^{111}$& $  -0.6902239809   $ & $b_1^{121}$& $ -0.6661897749 $ & $b_1^{132}$& $  -0.4908755105  $   \\
        $b_2^{111}$& $  \phantom{-}0.4629845398   $ & $b_2^{121}$& $ \phantom{-}0.4549032114 $ & $b_2^{132}$& $  \phantom{-}0.3761470588  $   \\
        $b_3^{111}$& $  -0.3124712630   $ & $b_3^{121}$& $ -0.3098689364 $ & $b_3^{132}$& $  -0.2757235761  $   \\
        $b_4^{111}$& $  \phantom{-}0.2030570414   $ & $b_4^{121}$& $ \phantom{-}0.2022430710 $ & $b_4^{132}$& $  \phantom{-}0.1878942344  $   \\
        $b_5^{111}$& $  -0.1122169575   $ & $b_5^{121}$& $ -0.1119772381 $ & $b_5^{132}$& $  -0.1064667431  $   \\
        $b_6^{111}$& $  \phantom{-}0.0429530470   $ & $b_6^{121}$& $ \phantom{-}0.0428950291 $ & $b_6^{132}$& $  \phantom{-}0.0412628857  $   \\
        $b_7^{111}$& $  -0.0079480275   $ & $b_7^{121}$& $ -0.0079400220 $ & $b_7^{132}$& $  -0.0076832353  $   \\
        \hline
        $s_{111}$ & $\phantom{-}0.5$ & $s_{121} $ & $\phantom{-}0.4$ & $s_{132}$ & $\phantom{-}0.2$ \\
        \hline
    \end{tabular}
 
    \begin{tabular}{ |c|l||c|l||c|l| }
        \hline
        \multicolumn{2}{|c||}{ $g_{133}$} & \multicolumn{2}{|c||}{ $g_{211}$} &
        \multicolumn{2}{|c|}{ $g_{311}$} \\
        \hline
    $b_0^{133}$& $ \phantom{-}0.8813944143  $   &   $b_0^{211}$&$   1.0003583991  $ &  $b_0^{311}$& $ \phantom{-}0.5081201992   $   \\ 
    $b_1^{133}$& $ -0.6491445844  $   &   $b_1^{211}$&$  -0.6654326742  $ &  $b_1^{311}$& $ -0.2326161465   $   \\ 
    $b_2^{133}$& $ \phantom{-}0.4489740965  $   &   $b_2^{211}$&$  \phantom{-}0.4295338070  $ &  $b_2^{311}$& $ \phantom{-}0.0972737581   $   \\ 
    $b_3^{133}$& $ -0.3078940138  $   &   $b_3^{211}$&$  -0.2812400555  $ &  $b_3^{311}$& $ -0.0398841619   $   \\ 
    $b_4^{133}$& $ \phantom{-}0.2016045985  $   &   $b_4^{211}$&$  \phantom{-}0.1785880816  $ &  $b_4^{311}$& $ \phantom{-}0.0160797798   $   \\ 
    $b_5^{133}$& $ -0.1117835035  $   &   $b_5^{211}$&$  -0.0971304019  $ &  $b_5^{311}$& $ -0.0060041662   $   \\ 
    $b_6^{133}$& $ \phantom{-}0.0428470454  $   &   $b_6^{211}$&$  \phantom{-}0.0368077929  $ &  $b_6^{311}$& $ \phantom{-}0.0017487435   $   \\ 
    $b_7^{133}$& $ -0.0079333006  $   &   $b_7^{211}$&$  -0.0067699674  $ &  $b_7^{311}$& $ -0.0002726251   $   \\ 
        \hline
        $s_{133}$ & $\phantom{-}0.4$ & $s_{211} $ & $\phantom{-}0.5$ & $s_{311}$ & $\phantom{-}0.3$ \\   
        \hline
    \end{tabular}

    \begin{tabular}{ |c|l||c|l||c|l| }
        \hline
        \multicolumn{2}{|c||}{ $g_{312}$} & \multicolumn{2}{|c||}{ $g_{321}$} &
        \multicolumn{2}{|c|}{ $g_{331}$} \\
        \hline
    $b_0^{312}$& $ \phantom{-}0.2013135639  $   &   $b_0^{321} $& $  \phantom{-}0.8072631281 $  &  $b_0^{331}  $& $   1.0061848887  $   \\  
    $b_1^{312}$& $ -0.0743470726  $   &   $b_1^{321} $& $  -0.4648191338 $  &  $b_1^{331}  $& $  -0.6476496175  $   \\  
    $b_2^{312}$& $ \phantom{-}0.0244467204  $   &   $b_2^{321} $& $  \phantom{-}0.2575887534 $  &  $b_2^{331}  $& $  \phantom{-}0.4067505913  $   \\  
    $b_3^{312}$& $ -0.0077137243  $   &   $b_3^{321} $& $  -0.1452654416 $  &  $b_3^{331}  $& $  -0.2603691317  $   \\  
    $b_4^{312}$& $ \phantom{-}0.0023694160  $   &   $b_4^{321} $& $  \phantom{-}0.0809084732 $  &  $b_4^{331}  $& $  \phantom{-}0.1624073740  $   \\  
    $b_5^{312}$& $ -0.0006874061  $   &   $b_5^{321} $& $  -0.0398683727 $  &  $b_5^{331}  $& $  -0.0872159881  $   \\  
    $b_6^{312}$& $ \phantom{-}0.0001645840  $   &   $b_6^{321} $& $  \phantom{-}0.0141620110 $  &  $b_6^{331}  $& $  \phantom{-}0.0327840257  $   \\  
    $b_7^{312}$& $ -0.0000225598  $   &   $b_7^{321} $& $  -0.0025036874 $  &  $b_7^{331}  $& $  -0.0060002559  $   \\  
        \hline
        $s_{312}$ & $\phantom{-}0.1$ & $s_{321} $ & $\phantom{-}0.4$ & $s_{331}$ & $\phantom{-}0.5$ \\   
        \hline
    \end{tabular}
    \end{center}

    \subsubsection{Part 5: $(\mathcal M \setminus \mathcal
    L)\cap(\sqrt{20},\sqrt{21})$}
    \label{Appendix V}
 
    We present coefficients of the polynomial components of test functions $\underline f =
    (f_1,f_2,f_3,f_4)$ and $\underline g=(g_1,g_2,g_3,g_4)$ 
    used in~\eqref{testp5:eq}. It follows from the equality between the first
    and second columns of the matrix $M$ that $f_1 = f_2$ and $g_1 = g_2$. 
    $$
    f_j = \sum_{k=0}^9 a_k^j x^k \quad 
    g_j = \sum_{k=0}^9 b_k^j x^k, \quad j =1,2,3,4.
    $$
    Straightforward calculation shows that the functions~$f_j$ and~$g_j$ are
    monotone decreasing on $[0,1]$ and achieve
    their minima at~$1$. For convenience, we give a lower bound $s_j <  \min(f_j(1), g_j(1))$.    
    \begin{center}
    \begin{tabular}{|c|l||c|l||c|l| }
        \hline
        \multicolumn{2}{|c|}{$f_1$} & \multicolumn{2}{|c|}{$f_3$} &  \multicolumn{2}{|c|}{$f_4$} \\
        \hline
          $a^1_0$& $  \phantom{-}0.9799928531   $  & $a^3_0$& $  \phantom{-}1.0045893915    $  &  $a^4_0$& $  \phantom{-}0.1934516264    $   \\ 
          $a^1_1$& $  -0.7406258897   $  & $a^3_1$& $  -0.7487821988    $  &  $a^4_1$& $  -0.0765314787    $   \\ 
          $a^1_2$& $  \phantom{-}0.5273926877   $  & $a^3_2$& $  \phantom{-}0.5296982401    $  &  $a^4_2$& $  \phantom{-}0.0259667329    $   \\ 
          $a^1_3$& $  -0.3800229054   $  & $a^3_3$& $  -0.3806370202    $  &  $a^4_3$& $  -0.0083489461    $   \\ 
          $a^1_4$& $  \phantom{-}0.2763248479   $  & $a^3_4$& $  \phantom{-}0.2764834085    $  &  $a^4_4$& $  \phantom{-}0.0026152698    $   \\ 
          $a^1_5$& $  -0.1955487122   $  & $a^3_5$& $  -0.1955888334    $  &  $a^4_5$& $  -0.0008056462    $   \\ 
          $a^1_6$& $  \phantom{-}0.1241717000   $  & $a^3_6$& $  \phantom{-}0.1241816277    $  &  $a^4_6$& $  \phantom{-}0.0002419581    $   \\ 
          $a^1_7$& $  -0.0622409671   $  & $a^3_7$& $  -0.0622432782    $  &  $a^4_7$& $  -0.0000670179    $   \\ 
          $a^1_8$& $  \phantom{-}0.0206092394   $  & $a^3_8$& $  \phantom{-}0.0206096817    $  &  $a^4_8$& $  \phantom{-}0.0000146636    $   \\ 
          $a^1_9$& $  -0.0032442178   $  & $a^3_9$& $  -0.0032442666    $  &  $a^4_9$& $  -0.0000017681    $   \\ 
          \hline
          $s_1$ & $\phantom{-}0.4$ & $s_3$ & $\phantom{-}0.4$ & $s_4$ & $\phantom{-}0.05$ \\
        \hline
    \end{tabular}
         
     \begin{tabular}{|c|l||c|l||c|l| }
        \hline
        \multicolumn{2}{|c|}{$g_1$} & \multicolumn{2}{|c|}{$g_3$} &  \multicolumn{2}{|c|}{$g_4$} \\
        \hline
          $b^1_0$& $  \phantom{-}0.9799928382   $  & $b^3_0$& $  \phantom{-}1.0045893900    $  &  $b^4_0$& $  \phantom{-}0.1934516765    $   \\ 
          $b^1_1$& $  -0.7406256446   $  & $b^3_1$& $  -0.7487819559    $  &  $b^4_1$& $  -0.0765314768    $   \\ 
          $b^1_2$& $  \phantom{-}0.5273924149   $  & $b^3_2$& $  \phantom{-}0.5296979674    $  &  $b^4_2$& $  \phantom{-}0.0259667280    $   \\ 
          $b^1_3$& $  -0.3800226607   $  & $b^3_3$& $  -0.3806367755    $  &  $b^4_3$& $  -0.0083489435    $   \\ 
          $b^1_4$& $  \phantom{-}0.2763246450   $  & $b^3_4$& $  \phantom{-}0.2764832054    $  &  $b^4_4$& $  \phantom{-}0.0026152687    $   \\ 
          $b^1_5$& $  -0.1955485561   $  & $b^3_5$& $  -0.1955886769    $  &  $b^4_5$& $  -0.0008056458    $   \\ 
          $b^1_6$& $  \phantom{-}0.1241715958   $  & $b^3_6$& $  \phantom{-}0.1241815227    $  &  $b^4_6$& $  \phantom{-}0.0002419578    $   \\ 
          $b^1_7$& $  -0.0622409134   $  & $b^3_7$& $  -0.0622432239    $  &  $b^4_7$& $  -0.0000670177    $   \\ 
          $b^1_8$& $  \phantom{-}0.0206092213   $  & $b^3_8$& $  \phantom{-}0.0206096634    $  &  $b^4_8$& $  \phantom{-}0.0000146636    $   \\ 
          $b^1_9$& $  -0.0032442150   $  & $b^3_9$& $  -0.0032442637    $  &  $b^4_9$& $  -0.0000017681    $   \\ 
        \hline
          $s_1$ & $\phantom{-}0.4$ & $s_3$ & $\phantom{-}0.4$ & $s_4$ & $\phantom{-}0.05$ \\
        \hline
    \end{tabular}
    \end{center}
    
\subsection{Zaremba theory}

    \subsubsection{$\dim_H(E_5)$}
    \label{Appendix VI-1}
    We present coefficients of the polynomial test functions~$f$ and~$g$
    used in~\eqref{dimE5:eq}.  
    $$
    f= \sum_{k=0}^{15} a_k x^k, \quad 
    g = \sum_{k=0}^{15} b_k x^k. 
    $$
    Similarly to the previous examples, the functions $f$ and $g$ are monotone and they can be bounded
    from below by their value at~$1$. In particular, we have $f(1), g(1) \ge
    0.4$.  
    \begin{center}
    \begin{tabular}{|c|l||c|l|}
        \hline
 \multicolumn{2}{|c||}{$f$} & \multicolumn{2}{|c|}{$g$} \\
        \hline
        $a_0 		$&$   \phantom{-}1.002075775192587   $  &  $b_0     $&$  \phantom{-}1.002075775192580   $  \\ 
        $a_1 		$&$   -0.863832791554195   $  &  $b_1     $&$   -0.863832791551160   $  \\ 
        $a_2 		$&$   \phantom{-}0.694904605500679   $  &  $b_2     $&$   \phantom{-}0.694904605500778   $  \\ 
        $a_3 		$&$   -0.563982609501401   $  &  $b_3     $&$   -0.563982609938672   $  \\ 
        $a_4 		$&$   \phantom{-}0.462936795376894   $  &  $b_4     $&$   \phantom{-}0.462936803438964   $  \\ 
        $a_5 		$&$   -0.382787786278096   $  &  $b_5     $&$   -0.382787860060200   $  \\ 
        $a_6 		$&$   \phantom{-}0.317798375235142   $  &  $b_6     $&$   \phantom{-}0.317798788806115   $  \\ 
        $a_7 		$&$   -0.263717774563333   $  &  $b_7     $&$   -0.263719322983292   $  \\ 
        $a_8 		$&$   \phantom{-}0.216187532791877   $  &  $b_8     $&$   \phantom{-}0.216191581202906   $  \\ 
        $a_9 		$&$   -0.170354610076857   $  &  $b_9  $&$   -0.170362167218400   $  \\ 
        $a_{10}     $&$   \phantom{-}0.123046705346496   $  &  $b_{10}  $&$   \phantom{-}0.123056842712685   $  \\ 
        $a_{11}     $&$   -0.076412879097916   $  &  $b_{11}  $&$   -0.076422574631579   $  \\ 
        $a_{12}     $&$   \phantom{-}0.037844499485800   $  &  $b_{12}  $&$   \phantom{-}0.037850946915568   $  \\ 
        $a_{13}     $&$   -0.013650349050295   $  &  $b_{13}  $&$   -0.013653178641107   $  \\ 
        $a_{14}     $&$   \phantom{-}0.003133011516183   $  &  $b_{14}  $&$   \phantom{-}0.003133747261017   $  \\ 
        $a_{15}     $&$   -0.000339760445059   $  &  $b_{15}  $&$   -0.000339846126736   $  \\
        \hline
    \end{tabular}
    \end{center}
    
   \subsubsection{$\dim_H(E_4)$}
   \label{Appendix VI-2}
   Coefficients of the polynomial test functions~$f$ and~$g$
    used in~\eqref{dimE4:eq}.  
    $$
    f= \sum_{k=0}^{15} a_k x^k, \quad 
    g = \sum_{k=0}^{15} b_k x^k. 
    $$
    Similar to the previous examples, the functions $f$ and $g$ are monotone and they can be bounded
    from below by their value at~$1$. In particular, we have $f(1), g(1) \ge
    0.4$.  
    \begin{center}
    \begin{tabular}{ |c|l||c|l| }
        \hline
        \multicolumn{2}{|c|}{$f$} &  \multicolumn{2}{|c|}{$g$} \\
        \hline
        $a_0    $&$    \phantom{-}1.001981557057916      $  &  $b_0    $&$  \phantom{-}1.001981557057906       $ \\ 
        $a_1    $&$    -0.824549641777407      $  &  $b_1    $&$   -0.824549641773357       $ \\ 
        $a_2    $&$    \phantom{-}0.632377740413372      $  &  $b_2    $&$   \phantom{-}0.632377740394589       $ \\ 
        $a_3    $&$    -0.489751892709023      $  &  $b_3    $&$   -0.489751892404770       $ \\ 
        $a_4    $&$    \phantom{-}0.384644919648888      $  &  $b_4    $&$   \phantom{-}0.384644915987248       $ \\ 
        $a_5    $&$    -0.305093638471734      $  &  $b_5    $&$   -0.305093610668140       $ \\ 
        $a_6    $&$    \phantom{-}0.243481281464752      $  &  $b_6    $&$   \phantom{-}0.243481142895873       $ \\ 
        $a_7    $&$    -0.194635007255869      $  &  $b_7    $&$   -0.194634539676458       $ \\ 
        $a_8    $&$    \phantom{-}0.154196605135780      $  &  $b_8    $&$   \phantom{-}0.154195521320957       $ \\ 
        $a_{9}  $&$    -0.118018014953805      $  &  $b_{9}  $&$   -0.118016296613633       $ \\ 
        $a_{10} $&$    \phantom{-}0.083325052299187      $  &  $b_{10} $&$   \phantom{-}0.083323249134991       $ \\ 
        $a_{11} $&$    -0.050893608677143      $  &  $b_{11} $&$   -0.050892477178423       $ \\ 
        $a_{12} $&$    \phantom{-}0.024910421954700      $  &  $b_{12} $&$   \phantom{-}0.024910155247198       $ \\ 
        $a_{13} $&$    -0.008908424701076      $  &  $b_{13} $&$   -0.008908569288906       $ \\ 
        $a_{14} $&$    \phantom{-}0.002031148876994      $  &  $b_{14} $&$   \phantom{-}0.002031269948929       $ \\ 
        $a_{15} $&$    -0.000219053588808      $  &  $b_{15} $&$   -0.000219079665840       $ \\ 
        \hline
    \end{tabular}
\end{center}

\subsubsection{$\dim_H(E_{1235})$}
   \label{Appendix VI-3}
   Coefficients of the polynomial test functions~$f$ and~$g$
   corresponding to $t_1 = 0.7709149399\,36$ and $t_0 = t_1 + 3\cdot 10^{-12}$
   respectively. 
    $$
    f= \sum_{k=0}^{15} a_k x^k, \quad 
    g = \sum_{k=0}^{15} b_k x^k.
    $$
    Similar to the previous examples, the functions $f$ and $g$ are monotone and they can be bounded
    from below by their value at~$1$. In particular, we have $f(1), g(1) \ge
    0.4$.  
    \begin{center}
    \begin{tabular}{ |c |l ||c|l| }
        \hline
       \multicolumn{2}{|c|}{$f$} & \multicolumn{2}{|c|}{$g$} \\
        \hline
        $a_0    $&$   \phantom{-}1.001943825796940    $   &    $b_0    $&$   \phantom{-}1.001943825796930    $  \\ 
        $a_1    $&$   -0.808837588421526    $   &    $b_1    $&$   -0.808837588417559    $  \\ 
        $a_2    $&$   \phantom{-}0.615494218966991    $   &    $b_2    $&$   \phantom{-}0.615494218954915    $  \\ 
        $a_3    $&$   -0.474372511899657    $   &    $b_3    $&$   -0.474372511806688    $  \\ 
        $a_4    $&$   \phantom{-}0.371576736077143    $   &    $b_4    $&$   \phantom{-}0.371576735828591    $  \\ 
        $a_5    $&$   -0.294687522275668    $   &    $b_5    $&$   -0.294687526007311    $  \\ 
        $a_6    $&$   \phantom{-}0.235767134415525    $   &    $b_6    $&$   \phantom{-}0.235767178162078    $  \\ 
        $a_7    $&$   -0.189403864881314    $   &    $b_7    $&$   -0.189404097267897    $  \\ 
        $a_8    $&$   \phantom{-}0.151071322536495    $   &    $b_8    $&$   \phantom{-}0.151072090158663    $  \\ 
        $a_{9}  $&$   -0.116493763140853    $   &    $b_{9}  $&$   -0.116495485397309    $  \\ 
        $a_{10} $&$   \phantom{-}0.082822264062997    $   &    $b_{10} $&$   \phantom{-}0.082824975856055    $  \\ 
        $a_{11} $&$   -0.050871691397334    $   &    $b_{11} $&$   -0.050874703061709    $  \\ 
        $a_{12} $&$   \phantom{-}0.025002624053740    $   &    $b_{12} $&$   \phantom{-}0.025004941184306    $  \\ 
        $a_{13} $&$   -0.008966855413747    $   &    $b_{13} $&$   -0.008968032838312    $  \\ 
        $a_{14} $&$   \phantom{-}0.002048316877335    $   &    $b_{14} $&$   \phantom{-}0.002048672642558    $  \\ 
        $a_{15} $&$   -0.000221169553698    $   &    $b_{15} $&$   -0.000221217982471    $  \\ 
        \hline
    \end{tabular}
\end{center}

    \subsection{Fuchsian Schottky groups}
    \label{Appendix VII} 
    Coefficients of the polynomial test functions $\underline f = (f_0,f_1,f_2)$
    and  $\underline g=(g_0,g_1,g_2)$, where each $f_j$ and $g_j$ for  $j=0,1,2$
    is defined on the interval $X_j$, respectively.
    $$
     f_j = \sum_{k=0}^{24} a_k^j x^k, \quad 
     g_j = \sum_{k=0}^{24} b_k^j x^k, \quad j = 0,1,2. 
    $$
  The functions $f_0$ and $g_0$ are even and therefore their odd coefficients
  vanish: $a^0_{2k+1} = b^0_{2k+1} \equiv 0$. Moreover, $f_1(x) =
  f_2(-x)$ and $g_1(x) = g_2(-x)$ therefore their even coefficients agree and
  their odd coefficients  have the opposite signs: 
  $a^1_{2k} = a^2_{2k}$, $a^1_{2k+1} = -a^2_{2k+1}$;
  $b^1_{2k} = b^2_{2k}$, $b^1_{2k+1} = -b^2_{2k+1}$.

  \bigskip

    \scalebox{0.85}{%
    \begin{tabular}{ |c  | r | }
        \hline
        \multicolumn{2}{|c|}{$f_1$} \\
        \hline
          $a^1_0   	$&$    2.8770704827462130787849251492030     $ \\
          $a^1_1   	$&$  -11.4401515732168029128293331165380     $ \\
          $a^1_2   	$&$   41.2783414403998127337761764215520     $ \\
          $a^1_3   	$&$ -111.0688407519945991015658074607160     $ \\
          $a^1_4   	$&$  229.2359826830719504436183231429020     $ \\
          $a^1_5   	$&$ -374.1839317918421651508580498351190     $ \\
          $a^1_6   	$&$  494.2486182891231866860550011193190     $ \\
          $a^1_7   	$&$ -537.1550718818002043450127874148360     $ \\
          $a^1_8   	$&$  486.1852872370976132985936371604760     $ \\
          $a^1_{9} 	$&$ -369.6791357119376158018979254258860     $ \\
          $a^1_{10}    $&$  237.5656093180881359651809836821330     $ \\
          $a^1_{11}    $&$ -129.5193450984147750120850142699940     $ \\
          $a^1_{12}    $&$   60.0196876244694040557343257845630     $ \\
          $a^1_{13}	$&$  -23.6421320534647605005954912805860     $ \\
          $a^1_{14}    $&$    7.9021374699139922100387246571970     $ \\
          $a^1_{15}	$&$   -2.2327574635963721814894503262740     $ \\
          $a^1_{16}    $&$    0.5301045742621054501101376372390     $ \\
          $a^1_{17}	$&$   -0.1048257433528365630840057396770     $ \\  
          $a^1_{18}    $&$    0.0170498220734007394058868515300     $ \\ 
          $a^1_{19}    $&$   -0.0022409149968744747201392201870     $ \\ 
          $a^1_{20}	$&$    0.0002320048619867762156560576330     $ \\
          $a^1_{21}    $&$   -0.0000182076188328122444526479340     $ \\
          $a^1_{22}	$&$    1.017711979517475899435180\cdot10^{-6}\phantom{18}  $ \\
          $a^1_{23}    $&$   -3.60864141942206682494120\cdot10^{-8}\phantom{188}   $ \\
          $a^1_{24}    $&$    6.09941161252684275320\cdot10^{-10}\phantom{88888}     $ \\
        \hline
        \multicolumn{2}{|c|}{$f_0$} \\
        \hline
          $a^0_0    $&$  0.988509120501286981726492299097        $\\
          $a^0_2    $&$  0.157542850847455232142931716289        $\\
          $a^0_4    $&$  0.042904006074342082007020639894        $\\
          $a^0_6    $&$  0.013311224077669383683707536427        $\\
          $a^0_8    $&$  0.004429850096065885316293787561        $\\
          $a^0_{10} $&$  0.001543440394063281474032350116        $\\
          $a^0_{12} $&$  0.000555351173373256765356036270        $\\
          $a^0_{14} $&$  0.000204517408718030083947390234        $\\ 
          $a^0_{16} $&$  0.000076609680108249015238001359        $\\ 
          $a^0_{18} $&$  0.000029061155169646923244681761        $\\ 
          $a^0_{20} $&$  0.000011137276040685233840494967        $\\ 
          $a^0_{22} $&$  4.2275802352717104321010360\cdot10^{-6}   $\\ 
          $a^0_{24} $&$  1.9876322730511567993854040\cdot10^{-6}   $\\ 
        \hline
        \end{tabular}
    }
    \scalebox{0.85}{%
    \begin{tabular}{|c | r | }
        \hline
        \multicolumn{2}{|c|}{$g_1$}\\
        \hline
     $b^1_0  		$&$     2.877070483741637235339396793709    $ \\
     $b^1_1 		$&$   -11.440151580102508283298777329005  $ \\
     $b^1_2  		$&$    41.278341467154591279141888564443   $ \\
     $b^1_3 		$&$  -111.068840826093040583068331287781 $ \\
     $b^1_4       $&$   229.235982838419674330781933540686  $ \\
     $b^1_5 		$&$  -374.183932047938578510658766660908 $ \\
     $b^1_6 		$&$   494.248618629707611522216515713345  $ \\
     $b^1_7  		$&$  -537.155072253799253891555878102640 $ \\
     $b^1_8 		$&$   486.185287575077537070728753265866  $ \\
     $b^1_{9} 	$&$  -369.679135969693151986947450814820 $ \\
     $b^1_{10} 	$&$   237.565609484126480000044027316791  $ \\
     $b^1_{11} 	$&$  -129.519345189116330520934728232633 $ \\
     $b^1_{12} 	$&$    60.019687666569911991602610502470   $ \\
     $b^1_{13} 	$&$   -23.642132070071466367919223736091  $ \\
     $b^1_{14} 	$&$     7.902137475471208082330326553248    $ \\
     $b^1_{15} 	$&$    -2.232757465168180629595168192690   $ \\
     $b^1_{16} 	$&$     0.530104574635619914117715916573    $ \\
     $b^1_{17} 	$&$    -0.104825743426755206788900861842   $ \\ 
     $b^1_{18} 	$&$     0.017049822085431863890249313609    $ \\ 
     $b^1_{19} 	$&$    -0.002240914998456736922493194695   $ \\ 
     $b^1_{20} 	$&$     0.000232004862150679716335333676    $ \\ 
     $b^1_{21} 	$&$    -0.000018207618845681630026751529   $ \\ 
     $b^1_{22} 	$&$     1.0177119802371267732818260\cdot10^{-6}  $ \\ 
     $b^1_{23} 	$&$    -3.60864142197485386059250\cdot10^{-8}   $ \\ 
     $b^1_{24} 	$&$     6.099411616843192081320\cdot10^{-10}    $ \\ 
        \hline
        \multicolumn{2}{|c|}{$g_0$}\\
        \hline
         $b^0_0    $&$  0.988509120496506332431826934199       $  \\
         $b^0_2    $&$  0.157542850912623683440552637873       $  \\ 
         $b^0_4    $&$  0.042904006097255254396958615314       $  \\ 
         $b^0_6    $&$  0.013311224085738465758389603856       $  \\ 
         $b^0_8    $&$  0.004429850098979983540304963280       $  \\ 
         $b^0_{10} $&$  0.001543440395140704805733043742       $  \\ 
         $b^0_{12} $&$  0.000555351173779220956618567711       $  \\ 
         $b^0_{14} $&$  0.000204517408873232993380371066       $  \\ 
         $b^0_{16} $&$  0.000076609680168235804587352199       $  \\ 
         $b^0_{18} $&$  0.000029061155193021143267002387       $  \\ 
         $b^0_{20} $&$  0.000011137276049855568024571829       $  \\ 
         $b^0_{22} $&$  4.2275802388232323523285170\cdot10^{-6}  $  \\ 
         $b^0_{24} $&$  1.9876322747582711067418060\cdot10^{-6}  $  \\ 
        \hline
          \end{tabular}
    }

   \subsection{Non-linear example}
    \label{Appendix VIII}
   Coefficients of the polynomial test functions~$f$ and~$g$
   corresponding to $t_0 = 0.4934480908\,02$ and $t_1 = t_0 +  10^{-12}$
   respectively.
    $$
    f= \sum_{k=0}^{7} a_{2k} x^{2k}, \quad 
    g = \sum_{k=0}^{7} b_{2k} x^{2k}.
    $$
  \begin{center}
\begin{tabular}{ |c| l ||c|l| }
        \hline
        \multicolumn{2}{|c|}{$f$} & \multicolumn{2}{|c|}{$g$} \\
        \hline
 $a_0    $&$ \phantom{-}0.260509445190371               $&$  b_0     $&$  \phantom{-}0.260509445190371             $  \\ 
 $a_2    $&$ -4.88560887219182 \cdot10^{-3}   $&$  b_2     $&$  -4.88560887219092\cdot10^{-3}  $  \\ 
 $a_4    $&$ \phantom{-} 3.65942002704976 \cdot10^{-4}   $&$  b_4     $&$ \phantom{-}  3.65942002704768\cdot10^{-4}  $  \\ 
 $a_6    $&$ -3.48521776330216 \cdot10^{-5}   $&$  b_6     $&$  -3.48521776329928\cdot10^{-5}  $  \\ 
 $a_8    $&$ \phantom{-} 3.61101766657243 \cdot10^{-6}   $&$  b_8     $&$ \phantom{-}  3.61101766656879\cdot10^{-6}  $ \\ 
 $a_{10} $&$ -3.88502731592718 \cdot10^{-7}   $&$  b_{10}  $&$  -3.88502731592271\cdot10^{-7}  $ \\ 
 $a_{12} $&$ \phantom{-} 4.09696977259504 \cdot10^{-8}   $&$  b_{12}  $&$ \phantom{-}  4.09696977258989\cdot10^{-8}  $ \\ 
 $a_{14} $&$ -3.23465524617897 \cdot10^{-9}   $&$  b_{14}  $&$  -3.23465524617468\cdot10^{-9}  $ \\ 
        \hline
\end{tabular}
\end{center}

\medskip
      
  {\footnotesize
  \noindent
  \textsc{P. Vytnova, Department of Mathematics, Warwick University, Coventry,
  CV4 7AL, UK}
  \noindent
  \textit{E-mail address}: \texttt{P.Vytnova@warwick.ac.uk}
  \par
  \addvspace{\medskipamount}
  \noindent
  \textsc{M. Pollicott, Department of Mathematics, Warwick University, Coventry, CV4 7AL, UK.}
  \noindent
  \textit{E-mail address}: \texttt{masdbl@warwick.ac.uk}
}

      \end{document}